\documentclass[12pt]{article}
\usepackage{a4wide}
\usepackage{amsmath,amssymb,amsthm}
\usepackage[mathscr]{eucal}
\usepackage{tikz}
\usetikzlibrary{positioning,automata}
\usetikzlibrary{arrows,calc}
\usetikzlibrary{decorations.markings,plotmarks}
\usepackage{graphics}
\usepackage{color}
\usepackage{hyperref}
\usepackage{dsfont}
\usepackage{enumerate}
\usepackage{mathtools}
\usepackage{cite}

\newtheorem{theorem}{Theorem}
\newtheorem{proposition}[theorem]{Proposition}
\newtheorem{corollary}[theorem]{Corollary}
\newtheorem{lemma}[theorem]{Lemma}
\newtheorem{fact}[theorem]{Fact}

\theoremstyle{definition}
\newtheorem{definition}[theorem]{Definition}

\newtheorem{problem}[theorem]{Problem}

\newtheorem{conjecture}[theorem]{Conjecture}
\newtheorem{remark}[theorem]{Remark}
\newtheorem*{notation}{Notation}

\newtheorem{obs}[theorem]{Observation}

\numberwithin{theorem}{section}

\DeclareMathOperator{\Aut}{Aut}
\DeclareMathOperator{\Iso}{Iso}
\DeclareMathOperator{\Sym}{Sym}
\DeclareMathOperator{\dist}{dist}
\DeclareMathOperator{\supp}{supp}
\DeclareMathOperator{\mindeg}{mindeg}
\DeclareMathOperator{\motion}{motion}

\newcommand{\eee}{\mathrm{e}}

\newcommand{\footremember}[2]{%
    \footnote{#2}
    \newcounter{#1}
    \setcounter{#1}{\value{footnote}}%
}

\title{On the automorphism groups of distance-regular graphs and
 rank-4 primitive coherent configurations}
\author{Bohdan Kivva\footremember{alley}{University of Chicago, e-mail: bkivva@uchicago.edu}}

\begin{document}
\maketitle

\vspace{-0.8cm}
\begin{abstract}
The minimal degree of a permutation group $G$ is the minimum number of points not fixed by non-identity elements of $G$. Lower bounds on the minimal degree have strong structural consequences on $G$. In 2014 Babai proved that the automorphism group of a strongly regular graph with $n$ vertices has minimal degree $\geq c n$, with known exceptions. Strongly regular graphs correspond to primitive coherent configurations of rank 3. We extend Babai's result to primitive coherent configurations of rank 4. We also show that the result extends to non-geometric primitive distance-regular graphs of bounded diameter. The proofs combine structural and spectral methods.

\end{abstract}

\tableofcontents

\section{Introduction}

Let $\sigma$ be a permutation of a set $\Omega$. The number of points not fixed by $\sigma$ is called the \textit{degree} of the permutation $\sigma$. Let $G$ be a permutation group on the set $\Omega$. The minimum of the degrees of non-identity elements in $G$ is called the \textit{minimal degree} of $G$. One of the classical problems in the theory of permutation groups is to classify the primitive permutation groups whose minimal degree is small (see \cite{Wielandt-classical}). The study of minimal degree goes back to 19th century, in particular Bochert in 1892 \cite{Bochert} proved that doubly transitive permutation group of degree $n$ has minimal degree $\geq n/4-1$ with obvious exceptions. More recently, Liebeck \cite{Liebeck} characterized primitive permutation groups with small minimal degree (see below).

The \textit{thickness} of a group $G$ is the greatest $t$ for which the alternating group $A_t$ is involved as a quotient group of a subgroup of $G$ \cite{Babai-str-reg}. Lower bounds on the minimal degree have strong structural consequences on $G$. In 1934 Wielandt \cite{Wielandt} showed that a linear lower bound on the minimal degree implies a logarithmic upper bound on the thickness of the group. Later, in \cite{thickness-primitive} Babai, Cameron and P\'alfy showed that primitive permutation groups with bounded thickness have polynomially bounded order. (See Section \ref{sec-thickness} for details.)

Switching from symmetry assumptions to regularity, in 2014 Babai \cite{Babai-str-reg} proved that the minimal degree of the automorphism group of a strongly regular graph is linear in the number of vertices, with known exceptions. 

\begin{definition}\label{def-motion}
For combinatorial structure $\mathcal{X}$ we will use term \textit{motion} to denote the minimal degree of the automorphism group $\Aut(\mathcal{X})$ of $\mathcal{X}$; we use notation 
\begin{equation}
\motion(\mathcal{X}) = \mindeg(\Aut(\mathcal{X})).
\end{equation}
\end{definition}

\begin{theorem}[Babai \cite{Babai-str-reg}]\label{babai-str-reg-thm}
Let $X$ be a strongly regular graph on $n\geq 29$ vertices. Then one of the following is true.
\begin{enumerate}
\item We have $\motion(X)\geq n/8$.
\item Graph $X$ or its complement is a triangular graph $T(s)$, a lattice graph $L_{2}(s)$ or a union of cliques.
\end{enumerate} 
%Assume that neither $X$ nor its complement is a triangular graph $T(s)$, a lattice graph $L_{2}(s)$ or a union of cliques. Then the  minimal degree of $\Aut(X)$ is at least $n/8$.
\end{theorem}

We are interested in proving lower bounds of the form $\motion(\mathcal{X})\geq cn$ for some constant $c>0$. Here and throughout this paper, $n$ will stand for the number of vertices.

\subsection{Coherent configurations: history and conjectures}

Strongly regular graphs can be seen as a special case of a more general class of highly regular combinatorial structures called \textit{coherent configurations}, combinatorial generalizations of the orbital structure of permutation groups; the case of orbitals is called ``Schurian coherent configurations" (see Section~\ref{sec-prelim} for definitions). More specifically, strongly regular graphs are essentially the coherent configurations of rank 3. From this  point of view, Theorem~\ref{babai-str-reg-thm} is a step toward a characterization of the coherent configurations with automorphism groups of small minimal degree.

The history of coherent configurations goes back to Schur's paper \cite{Schur} in 1933. Schur used coherent configurations to study permutation groups through their orbital configurations. Later Bose and Shimamoto \cite{Bose-Shimamoto} studied a special class of coherent configurations, called \textit{association schemes}. Coherent configurations in their full generality were independently introduced by Weisfeiler and Leman \cite{Weisfeiler}, \cite{Weisfeiler-Leman} and D. Higman \cite{Higman} around 1968. Higman developed the representation theory of coherent configurations and applied it to the permutation groups. At the same time, a related algebraic theory of coherent configurations, called ``cellular algebras,'' was introduced by Weisfeiler and Leman. Special classes of association schemes such as strongly regular graphs and, more generally, distance-regular graphs have been the subject of intensive study in algebraic combinatorics. A combinatorial study of coherent configurations in their full generality was initiated by Babai in \cite{Babai-annals}. Coherent configurations play an important role in the study of the Graph Isomorphism problem, adding combinatorial divide-and-conquer tools to the problem. This approach was used by Babai in his recent breakthrough paper~\cite{Babai-GI} to prove that the Graph Isomorphism problem can be solved in quasi-polynomial time. Recently, the representation theory of coherent configurations found unexpected applications in theory of computing, specifically to the complexity of matrix multiplication \cite{Cohn-Umans}.

Primitive permutation groups of large order were classified by Cameron \cite{Cameron} in 1981 using the Classification of the Finite Simple Groups. In particular Cameron proved that all primitive groups of order greater than $n^{\log_2(n)+1}$ are what are now called \textit{Cameron groups} (see Def. \ref{def-Cameron-groups}). \textit{Cameron schemes} are orbital configurations of Cameron groups.

A few years after Cameron's classification of large primitive groups, Liebeck \cite{Liebeck} obtained a similar classification of primitive groups with small minimal degree. He proved that all primitive groups of degree $n$ with minimal degree less than $\frac{n}{9\log_2(n)}$ are  Cameron groups. In 1991 \cite{Liebeck-Saxl} Liebeck and Saxl impoved lower bound on minimal degree of primitive non-Cameron groups to $n/3$. 

``Primitive coherent configurations'' are combinatorial generalizations of the orbital configurations of primitive permutation groups (see Section \ref{sec-prelim}). The following conjecture, due to Babai \cite[Conjecture 24]{Babai-str-reg}, says that all primitive coherent configurations with many automorphisms are  Cameron schemes.

\begin{conjecture}[Babai \cite{Babai-str-reg}]\label{conj-1}
For every $\varepsilon>0$, there is some $N_{\varepsilon}$ such that if $\mathfrak{X}$ is a primitive coherent configuration on $n> N_\varepsilon$ vertices and $|\Aut(\mathfrak{X})|\geq\exp(n^{\varepsilon})$, then $\mathfrak{X}$ is a Cameron scheme.
\end{conjecture}

If the conjecture is confirmed, it will give a combinatorial extension of Cameron's classification. The first progress on this conjecture was made by Babai in 1981 \cite{Babai-annals} where the conjecture was confirmed for all $\varepsilon>\frac{1}{2}$. As a byproduct he solved a then 100-year-old problem on primitive, but not doubly transitive groups, giving a nearly tight bound on their order. Recently (2015), Sun and Wilmes in \cite{Sun-Wilmes} made a first significant progress since that time, confirming the conjecture for all $\varepsilon>\frac{1}{3}$. In the case of rank 3, Chen et al. \cite{Chen-Sun-Teng} confirm the conjecture for $\varepsilon>\frac{9}{37}$.

A conjecture of the same flavor can be stated in terms of the  minimal degree of a permutation group.

\begin{conjecture}[Babai]\label{conj-3}
There exists $\gamma>0$ such that if $\mathfrak{X}$ is a primitive coherent configuration on $n$ vertices and $\motion(\mathfrak{X})<\gamma n$, then $\mathfrak{X}$ is a Cameron scheme.
\end{conjecture}

A weaker version of this conjecture asks to prove existence of such $\gamma$ for all primitive coherent configurations of fixed rank $r\geq 2$.

 Note that a homogeneous coherent configuration of rank 3 is either a strongly regular graph or a strongly regular tournament. It was proven in \cite{Babai-annals} that a strongly regular tournament has at most $n^{O(\log(n))}$ automorphisms, and the motion is at least $n/4$. Thus, Conjecture~\ref{conj-3} in the case of rank $3$ follows from the Theorem~\ref{babai-str-reg-thm}. One of our main results extends this to rank 4, see Theorem \ref{main-coherent}.

In fact, it is believed that a much stronger version of Conjecture \ref{conj-1} is true. 
 
 \begin{conjecture}[Babai \cite{Babai-ICM}, Conjecture 12.1]
There exists a polynomial $p$ such that the following hold.
Let $\mathfrak{X}$ be a non-Cameron primitive coherent configuration with $n$ vertices. Then
\vspace{-0.15cm}
\begin{enumerate}[(a)]
\setlength{\itemsep}{-3pt}
 \item  $\theta(\Aut(\mathfrak{X}))\le p(\log n)$ (where $\theta$ denotes the
        thickness, see Def.~\ref{def-thickness}) 
 \item  $|\Aut(\mathfrak{X})|\le \exp(p(\log n))$
\end{enumerate}
\end{conjecture}

Clearly, (a) follows from (b). As discussed above, Babai proved (a) for rank-3 coherent configurations and our Theorem \ref{main-coherent} extends this to rank 4.

\subsection{Our results}

In this paper we study the motion of primitive coherent configurations of rank 4 and of primitive distance-regular graphs. Distance-regular graphs induce an important family of coherent configurations, called \textit{metric schemes}. They are the coherent configurations, where colors of the pairs of vertices are defined by the distance metric in the graph. 

For primitive coherent configurations of rank 4 we confirm Conjecture \ref{conj-3}. We note that the only Cameron schemes of rank $\leq 4$ are Hamming and Johnson schemes, so our result is the following.
 
{
\renewcommand{\thetheorem}{\ref{main-coherent}}
\begin{theorem}
Let $\mathfrak{X}$ be a primitive coherent configuration of rank $4$ on $n$ vertices . Then one of the following is true.
\begin{enumerate}
\item We have $\motion(\mathfrak{X})\geq \gamma n$, where $\gamma>0$ is an absolute constant.
\item The configuration $\mathfrak{X}$ is a Hamming scheme or a Johnson scheme.
\end{enumerate}
\end{theorem}
\addtocounter{theorem}{-1}
}

In our main result for distance-regular graphs we show that a primitive distance-regular graph $X$ of diameter $d$ is  geometric (see Definition \ref{def-geom}), or  the motion of $X$ is linear in the number of vertices.

{
\renewcommand{\thetheorem}{\ref{main-general-case}}
\begin{theorem}
For any $d\geq 3$ there exist constants $\gamma_d>0$ and $m_d\in \mathbb{N}$, such that for any primitive distance-regular graph $X$ with $n$ vertices and diameter $d$ one of the following is true. 
\begin{enumerate}
\item We have $\motion(X)\geq \gamma_d n$.
\item $X$ is a geometric distance-regular graph with smallest eigenvalue $-m$, where $m\leq m_d$. 
\end{enumerate}
Furthermore, one can set $m_d = \lceil 2(d-1)(d-2)^{\log_2(d-2)} \rceil$.
\end{theorem}
\addtocounter{theorem}{-1}
} 

The key ingredient of the proof is the following new inequality  for the intersection numbers of a distance-regular graph (restated in a slightly modified form later as Lemma \ref{b-c-ineq}). Essentially it says that if $b_j$ is large and $c_{j+1}$ is small, then both $c_{j+2}$ and $b_{j+1}$ cannot be small simultaneously.

\begin{lemma}
Let $X$ be a primitive distance-regular graph of diameter $d\geq 3$. Assume $C = b_j/{c_{j+1}}>1$ for some $1 \leq j\leq d-2$. Then for any $1\leq s\leq j+1$
\[
b_{j+1}\left(\sum\limits_{t=1}^{s}\frac{1}{b_{t-1}}+\sum\limits_{t=1}^{j+2-s}\frac{1}{b_{t-1}}\right)+c_{j+2}\sum\limits_{t=1}^{j+1}\frac{1}{b_{t-1}}\geq 1-\frac{4}{C-1}.
\]

\end{lemma}

In the case of distance-regular graphs of diameter $3$ (metric coherent configuration of rank $4$) we are able to give a complete  classification of the geometric case. 

{
\renewcommand{\thetheorem}{\ref{main-thm}}
\begin{theorem}
Let $X$ be a distance-regular graph of diameter $3$ on $n$ vertices. Then one of the following is true.
\begin{enumerate}
\item We have $\motion(X)\geq \gamma n$ for some absolute constant $\gamma>0$.
\item $X$ is the Johnson graph $J(s,3)$ for $s\geq 7$, or the Hamming graph $H(3,s)$ for $s\geq 3$. 
\item $X$ is the cocktail-party graph.

\end{enumerate}
\end{theorem}
\addtocounter{theorem}{-1}
}

\noindent Note that here we do not assume primitivity of the graph.

\subsection{The main tools}\label{sec-tools} 

We follow Babai's approach \cite{Babai-str-reg}. In that paper, Babai used a combination of an old combinatorial tool \cite{Babai-annals} and introduced  spectral tool.  We will use the same tools to get bounds on the motion of distance-regular graphs and coherent configurations.

\begin{definition}
 A \textit{configuration} $\mathcal{X}$ of rank $r$ on the set $V$ is a pair $(V, c)$, where $c$ is a map $c:V\times V \rightarrow \{0, 1,..., r-1\}$ such that 
\begin{enumerate}[(i)]
\item $c(v,v) \neq c(u,w)$ for any $v,u,w\in V$ with $u\neq w$,
\item for any $i<r$ there is $i^{*}<r$ such that $c(u,v) = i$ implies $c(v,u) = i^{*}$ for all $u,v \in V$.
\end{enumerate}
\end{definition} 

\begin{definition}\label{def-dist-conf}
For a graph $X = (V, E)$ the configuration $\mathcal{X} = (V, c)$, where colors $c(u, v) = \dist(u,v)$ are defined by distance metric in $X$, is called the \textit{distance configuration} of $X$.
\end{definition}

\subsubsection{Combinatorial tool}

\begin{definition}[Babai \cite{Babai-annals}]\label{def-disting}
A pair of vertices $u$ and $v$ is \textit{distinguished} by vertex $x$ in a configuration $\mathcal{X}$ if the colors $c(x,u)$ and $c(x,v)$ are distinct.
\end{definition}

For graphs we will apply this definition to corresponding distance configuration. 

A set $S$ of vertices of a configuration $\mathcal{X}$ is \textit{distinguishing} if every pair of distinct vertices in $\mathcal{X}$ is distinguished by at least one element of $S$. Note that the pointwise stabilizer of a distinguishing set in $\Aut(\mathcal{X})$ is trivial. 

 Then the first tool is the following observation.
\begin{obs}\label{obs1}
Let $\mathcal{X}$ be  a configuration with $n$ vertices. If each pair of distinct vertices $u,v$ of $\mathcal{X}$ is distinguished by at least $m$ vertices, then $\motion(\mathcal{X})\geq m$.
\end{obs} 
\begin{proof} Indeed, let $\sigma\in \Aut(\mathcal{X})$ be any non-trivial automorphism of $\mathcal{X}$. Let $u$ be a vertex not fixed by $\sigma$. No fixed point of $\sigma$ distinguishes $u$ and $\sigma(u)$, so the degree of $\sigma$ is $\geq m$.
\end{proof}

 In \cite{Babai-annals} Babai showed that bounds on the numbers of vertices, which distinguish a pair, can be effectively transformed to the bounds on the order of automorphism group of a graph.  He used the lemma below to make the dramatic improvement of the bound on the order of an uniprimitive group in terms of its degree. In fact, it is the main tool, currently available, which was used to attack Conjecture \ref{conj-1}.
 
 \begin{lemma}[Babai {\cite[Lemma~5.4]{Babai-annals}} ]\label{Babai-disting-order-group} Let $\mathcal{X}$ be a configuration and suppose that any two distinct vertices are distinguished by at least $m$ vertices. Then there is a distinguishing  set of size at most $2n\log(n)/m+1$. Therefore, in particular, the order of the automorphism group satisfies $|\Aut(\mathfrak{X})| \leq n^{1+2n\log(n)/m}$. 
\end{lemma}

\subsubsection{Spectral tool} 

For a $k$-regular graph $X$ let $k = \xi_1\geq \xi_2\geq...\geq \xi_n$ denote the eigenvalues of the adjacency matrix of $X$. We call quantity $\xi = \xi(X) = \max\{ |\xi_i|: 2\leq i\leq n\}$ the \textit{zero-weight spectral radius} of $X$. The second tool follows from the Expander Mixing Lemma.

\begin{lemma}[Babai, {\cite[Proposition~12]{Babai-str-reg}}]\label{mixing-lemma-tool}
Let $X$ be a regular graph of degree $k$ with zero-weight spectral radius of $\xi$. Suppose every pair of vertices in $X$ has at most $q$ common neighbors. Then every non-identity automorphism of $X$ has at most $n\cdot\frac{q+\xi}{k}$ fixed points.
\end{lemma}

\subsubsection{Structural tool}

Together with the two tools mentioned, an important ingredient of many of our proofs is Metsch's geometricity criteria (see Theorem~\ref{Metsch}). Metsch's criteria gives sufficient conditions on the parameters of a graph to have a clique geometry.

\subsection{General outline of the proofs}

The proofs of Theorem~\ref{main-general-case} (a partial solution for distance-regular graphs of bounded diameter)  and Theorem~\ref{main-thm} (a complete solution in the diameter-3 case) are quite similar. For a distance-regular graph $X$, using Observation \ref{obs1}, we show that either the motion (see Def. \ref{def-motion}) of $X$ is linear in the number of vertices, or strong restrictions on the parameters of $X$ hold. It turns out that these restrictions suffice to yield good estimates of the eigenvalues of $X$ in the case of primitive distance-regular graphs and in the case of diameter 3. Thus, when  the number $\lambda$ of common neighbors for any pair of adjacent vertices  is small enough, we use Lemma \ref{mixing-lemma-tool} to get a lower linear bound on the motion. In the remaining case, we use Metsch's criteria for geometricity (see Theorem~\ref{Metsch}) to show that $X$ is a geometric distance-regular graph with smallest eigenvalue $-m$, where $m$ is bounded above by a function of the diameter. In the case of diameter $3$ we obtain a tight bound on $m$, reducing the problem to geometric distance-regular graphs with smallest eigenvalue $-3$ and we use the almost complete classification of such graphs obtained by Bang~\cite{Bang-diam-3} and Bang, Koolen~\cite{Bang-Koolen-diam-3} (see Theorem \ref{bang-koolen} in this paper). No analogous classification for the geometric distance-regular graphs with smallest eigenvalue $-m$ appears to be known for $m\geq 4$. We mention the following conjecture.

\begin{conjecture}[Bang, Koolen  {\cite[Conjecture 7.4]{Bang-Koolen-conj}}]\label{conj-bang-k}
For a fixed integer $m\geq 2$, any geometric distance-regular graph with smallest eigenvalue $-m$, diameter $d\geq 3$ and $\mu\geq 2$ (where $\mu$ stands for the number of common neighbors of any pair of vertices at distance~2) is either a Johnson graph, or a Hamming graph, or a Grassmann graph,  or a bilinear forms graph, or the number of vertices is bounded above by a function of $m$.
\end{conjecture}

We note that  for a Grassmann graphs and for bilinear forms graphs the motion is $\geq cn$ for some absolute constant $c>0$. So in Theorem \ref{main-general-case}, only the status of the case $\mu = 1$ remains open under the Bang-Koolen conjecture. 

Despite the fact that the proofs of Theorems \ref{main-thm} and \ref{main-general-case} follow similar strategies, and the fact that the problem for diameter $3$ appears to be just a special case of the problem for fixed diameter, treated in Theorem \ref{main-general-case}, we still keep separate proofs for them. This is done for two reasons. First, in the case of diameter $3$ we do not assume primitivity of the graph $X$. The second reason is that the eigenvalue bound given in Theorem \ref{main-general-case} is not tight enough. It is still possible to get desired bound from Theorem \ref{main-general-case} using additional trick, but essentially this argument is what is done in the proof we provide. 

Our strategy of proving Theorem \ref{main-coherent} (coherent configurations of rank 4) is again similar to the strategy for distance-regular graphs. All primitive coherent configurations of rank 4 naturally split into three families: distance-regular graphs of diameter 3, association schemes where each of three constituents is of diameter 2, and coherent configurations with two oriented colors. Since the result for the distance-regular case is proven in Theorem \ref{main-thm} and the case of oriented colors is easy, most of our effort goes into the case of association schemes with three constituents of diameter 2. First, we apply Observation \ref{obs1} to get that either the motion is linear, or the parameters of the configuration satisfy certain inequalities. This allows us to approximate the eigenvalues of the constituent graphs by relatively simple expressions. Thus our goal becomes to show that for at least  one of the constituents Lemma~\ref{mixing-lemma-tool} can be applied effectively. In the cases when we are not able to do this, we show that one of the constituents is a line graph. This allows us to use the classification of edge- and co-edge-regular graphs with smallest eigenvalue $-2$ by Seidel \cite{Siedel} (see Theorem \ref{geom-eig-2}). Now we are in a better position to apply  Lemma~\ref{mixing-lemma-tool} to one of the constituents. The reduction to Lemma~\ref{mixing-lemma-tool} in the case when one of the constituents is a line graph requires several new ideas, as well as analogs of the arguments used by Sun and Wilmes in the proof of Lemma 3.5 in \cite{Sun-Wilmes} and arguments used by Metsch in the proof of the main Theorem in {\cite{Metsch}}. See the introduction to Section \ref{sec-coherent} for a more detailed outline.

\subsection{Organization of the paper}

The paper is organized as follows. In Section \ref{sec-prelim} we introduce definitions and concepts that will be used throughout the paper. In Section \ref{sec-dist-reg} we discuss the proofs of the main results for distance regular graphs of diameter 3, and for primitive distance regular graphs of bounded diameter. 

In Sections \ref{sec-large-deg} and \ref{sec-spectral} we show that a non-geometric non-bipartite distance-regular graph of diameter 3 satisfies statement of Theorem \ref{main-thm}. More specifically, in Section \ref{sec-large-deg} we consider the case of vertex degree $k>n\gamma$ linear in the number of vertices; and in Section \ref{sec-spectral} we reduce the remaining case $k\leq \gamma n$ to the case of geometric graphs with smallest eigenvalue $-3$, or to the case of bipartite graphs. The case of bipartite distance-regular graphs of diameter~3 is considered in Section \ref{sec-bipartite}. In Section \ref{sec-geom} we present required background on geometric distance-regular graphs. We finish the proof of Theorem \ref{main-thm} in Section \ref{sec-diam-3-proof}. 

 The tradeoff inequality for the intersection numbers of a distance regular is proved in Section \ref{sec-tradeoff}.  In Sections \ref{sec-dist-number-gen-drg} and \ref{sec-drg-approx} we estimate the minimal distinguishing number and the spectral gap of a distance-regular graph in terms of its intersection numbers. In Section~\ref{sec-general} we finally prove Theorem \ref{main-general-case}, i.e, show that non-geometric primitive distance-regular graphs have motion linear in the number of vertices.

In Section \ref{sec-coherent} we treat the primitive coherent configurations of rank 4. Section \ref{sec-coh-approx} provides bounds on the eigenvalues of the constituents of a coherent configuration of rank~$4$ with certain properties.  In Section \ref{sec-coh-reduction} we reduce the problem to coherent configurations with a constituent, which possesses a rather restrictive structure. More specifically, one of the constituents in the remaining cases is a line graph (and is strongly regular in most cases). This cases are treated in  Section \ref{sec-subsec-str-reg}. Finally, in Section \ref{sec-coherent-thm-subsec} we combine obtained results into the proof of Theorem \ref{main-coherent}. 

Section \ref{sec-summary} summarizes the results and discusses approaches to the general problem (Conjecture \ref{conj-3}).

\section*{Acknowledgments}

The author is grateful to Professor L\'aszl\'o Babai for introducing him to the problems discussed in the paper and suggesting possible ways of approaching them, his invaluable assistance in framing the results, and constant support and encouragement.

\section{Preliminaries}\label{sec-prelim}
In this section we introduce coherent configurations, distance regular graphs and other related definitions and concepts that will be used throughout the paper. Distance-regular graphs are the subject of the monograph \cite{BCN} and survey  article \cite{Koolen-survey}. For more about coherent configurations we refer to \cite{Babai-GI}.

\subsection{Basic concepts and notation for graphs and groups}

Let $X$ be a graph. We will always denote by $n$ the number of vertices of $X$ and if $X$ is regular we denote by $k$ its degree. Denote by $\lambda = \lambda(X)$ the minimum number of common neighbors for pairs of adjacent vertices in $X$. Denote by $\mu = \mu(X)$ the maximum number of common neighbors for pairs of vertices at distance 2. We will denote the diameter of $X$ by $d$. If the graph is disconnected, then its diameter is defined to be $\infty$. Denote by $q(X)$ the maximum number of common neighbors of two distinct vertices in $X$.

Let $A$ be the adjacency matrix of $X$. Suppose that $X$ is $k$-regular. Then the all-ones vector is an eigenvector of $A$ with eigenvalue $k$. We will call them the \textit{trivial eigenvector} and the \textit{trivial eigenvalue}. All other eigenvalues of $A$ have absolute value not greater than $k$. We call them \textit{non-trivial} eigenvalues.

\begin{definition}
For a graph $X = (V(X), E(X))$ the \textit{line graph} $L(X)$ is defined as the graph on the vertex set $E(X)$, for which $e_1, e_2\in E(X)$ are adjacent in $L(X)$ if and only if they are incident to a common vertex in $X$.  
\end{definition}

Let $N(v)$ be the set of neighbors of vertex $v$ in $X$ and $N_i(v) = \{ w\in X| \dist(v,w) = i\}$ be the set of vertices at distance $i$ from $v$ in the graph $X$. 

Let $G$ be a transitive permutation group on set $\Omega$. A $G$-invariant partition $\Omega = B_1 \sqcup B_2 \sqcup ... \sqcup B_t$ is called a system of imprimitivity for $G$.

\begin{definition}
A transitive permutation group is called \textit{primitive} if it does not admit any non-trivial system of imprimitivity.
\end{definition}

\subsection{Coherent configurations}
Our terminology follows \cite{Babai-GI}. Denote $[m] =\{1,2,..., m\}$.

\noindent Let $V$ be a finite set, elements of which will be referred to as the vertices of a configuration.
\begin{definition}\label{def-conf} A \textit{configuration} $\mathfrak{X}$ of rank $r$ on the set $V$ is a pair $(V, c)$, where $c$ is a map $c:V\times V \rightarrow \{0, 1,..., r-1\}$ such that 
\begin{enumerate}[(i)]
\item $c(v,v) \neq c(u,w)$ for any $v,u,w\in V$ with $u\neq w$,
\item for any $i<r$ there is $i^{*}<r$ such that $c(u,v) = i$ implies $c(v,u) = i^{*}$ for all $u,v \in V$.
\end{enumerate}
\end{definition}
The value $c(u,v)$ is called the \textit{color} of a pair $(u,v)$. The color $c(u,v)$ is a \textit{vertex color} if $u = v$, and is an \textit{edge color} if $u\neq v$. Then condition (i) says that edge colors are different from vertex colors, and condition (ii) says that the color of a pair $(u,v)$ determines the color of $(v,u)$. 

For every $i<r$ consider the set $R_{i} = \{(u,v): c(u,v) = i\}$ of pairs of color $i$ and consider a digraph $X_i = (V, R_i)$. We will refer to both $R_i$ and $X_i$ as  the color-$i$ \textit{constituent} of $\mathfrak{X}$.
Clearly, there are two possibilities: if $i = i^*$, then color $i$ and the corresponding constituent $X_i$ are called \textit{undirected}; if $i\neq i^*$, then $(i^*)^* = i$ and color $i$ together with the corresponding constituent $X_i$ are called \textit{oriented}. Clearly, $\{R_i\}_{i< r}$ form a partition of $V\times V$. 

 We denote the adjacency matrix of the digraph $X_i$ by $A_i$. The adjacency matrices of the constituents satisfy
\begin{equation}\label{eq-constit-sum}
\sum\limits_{i = 0}^{r-1} A_i = J_{|V|} = J,
\end{equation}  
where $J$ denotes the all-ones matrix.

Note that conditions (1) and (2) of Definition \ref{def-conf} in matrix language mean the following. There exists a set $\mathcal{D}$ of colors, such that the identity matrix can be represented as a sum $\sum\limits_{i\in \mathcal{D}}A_i = I$. For any color $i$ we have $A_i^T = A_{i^*}$.

For the set of colors $\mathcal{I}$ we use notation $X_{\mathcal{I}}$ to denote the digraph on the set of vertices $V$, where arc $(x,y)$ is in $X_{\mathcal{I}}$ if and only if $c(x,y) \in \mathcal{I}$. For small sets we omit braces, for example, $X_{1,2}$ will be written in place of $X_{\{1,2\}}$.
\begin{definition}
A configuration $\mathfrak{X}$ is \textit{coherent} if 
\begin{enumerate}[(i)]
\item[(iii)] for all $i,j,t<r$ there is an \textit{intersection number} $p_{i,j}^{t}$ such that for all $u,v\in V$ if $c(u,v) = t$, then there exist exactly $p_{i,j}^{t}$ vertices $w\in V$ with $c(u,w) = i$ and $c(w,v) = j$.
\end{enumerate}
\end{definition}

The definition of a coherent configuration has several simple, but important, consequences. Let $\mathfrak{X}$ be a coherent configuration. Every edge color is aware of the colors of its tail and head. That is, for every edge color $i$ there exist vertex colors $i_{-}$ and $i_{+}$ such that if $c(u,v) = i$, then $c(u,u) = i_{-}$ and $c(v,v) = i_{+}$. Indeed, they are the only colors for which $p_{i,i_{+}}^{i}$ and $p_{i_{-}, i}^{i}$ are non-zero. Moreover, for every color $i$ its in-degree and out-degree are well-defined, as $k_{i}^{-} = p_{i^{*}, i}^{i_{+}}$ and $k_{i}^{+} = p_{i, i^{*}}^{i_{-}}$, respectively.

Observe that existence of the intersection numbers is equivalent to the following conditions on the adjacency matrices of the constituent digraphs.
\begin{equation}\label{eq-int-num}
A_iA_j = \sum\limits_{t = 0}^{r-1}p_{i,j}^{t}A_{t} \quad \text{ for all } i,j<r. 
 \end{equation}
 
Hence, $\{A_i: 0\leq i\leq r-1\}$ form a basis of an $r$-dimensional algebra with structure constants $p_{i,j}^t$. In particular, every $A_i$ has minimal polynomial of degree at most $r$.

\begin{definition}
A configuration $\mathfrak{X}$ is \textit{homogeneous} if $c(u,u) = c(v,v)$ for every $u,v\in V$.
\end{definition}

Unless specified otherwise, we will always assume that $0$ is the vertex color of a homogeneous configuration. The constituent which corresponds to the vertex color is also referred as the diagonal constituent. 

In a homogeneous coherent configuration we have $k_{i}^{+} = k_{i}^{-}$ for every color $i$. We denote this common value by $k_i$.

The intersection numbers of a homogeneous coherent configuration satisfy the following relations.

\begin{equation}\label{eq-sum-param}
 \sum\limits_{j = 0}^{r} p_{ij}^{t} = k_i \quad \text{and} \quad p_{i,j}^s k_s = p_{s,j}^{i} k_i.
 \end{equation}

Let $i,j<r$ be colors. Take $u, v\in V$ with $c(u,v) = j$. Define $\dist_i(u,v)$ to be the length $\ell$ of a shortest path $u_0 = u, u_1, ..., u_{\ell} = v$ such that $c(u_{t-1}, u_t) = i$ for $t\in [\ell]$. We claim that $\dist_{i}(j) = \dist_{i}(u,v)$ is well defined, i.e., does not depend on the choice of $u,v$, but only on the colors $j$ and $i$. 

Indeed, let $c(u,v) = c(u', v') = j$ and suppose there exist a path $u_0 = u, u_1, ..., u_{\ell} = v$ of length $\ell$, such that $c(u_{t-1}, u_t) = i$. Denote by $e_{t} = c(u_{t}, v)$. Then we know that $p_{i, e_{t}}^{e_{t-1}}\neq 0$ for $t\in [\ell-1]$. Let $u_0' = u'$. Then, as $p_{i, e_{t}}^{e_{t-1}}\neq 0$, by induction, there exists a $u_{t}'$ such that $c(u_{t-1}', u_{t}') = i$ and $c(u_{t}, v) = e_{t}$ for all $t\in [\ell-1]$. Hence, $\dist(u',v')\leq \dist(u,v)$ and similarly $\dist(u,v)\leq \dist(u', v')$. Therefore, $\dist_i(j)$ is well-defined.

\begin{obs}\label{obs-color-dist}
If $\dist_{i}(j)$ is finite, then $\dist_{i}(j)\leq r-1$.
\end{obs}
\begin{proof}
 Suppose that $\dist_{i}(j)$ is finite, then for $c(u,v) = j$ there exists a shortest path $u_0 = u, u_1, ..., u_{\ell} = v$ with $c(u_{t-1}, u_t) = i$. Denote by $e_{t} = c(u_{t}, v)$ for $0\leq t\leq \ell-1$. Then, all $e_{t}$ are different edge colors, or the path can be shortened. Thus $\ell \leq r-1$.
\end{proof}

\begin{definition}
A coherent configuration is called an \textit{association scheme} if $c(u,v) = c(v,u)$ for any $u,v\in V$. 
\end{definition}

\begin{corollary}
Any association scheme is a homogeneous configuration.
\end{corollary}
\begin{proof}
Since in a coherent configuration color of every edge is aware of the colors of its head and tail vertices, these vertices have the same color for any edge. 
\end{proof}

Note, for an association scheme every constituent digraph is a graph. Thus, for an association scheme and $i\neq 0$ the $i$-th constituent $X_i$  is a $k_i$-regular graph with $\lambda(X_i) = p_{i,i}^{i}$.
Moreover, it is clear that $p_{i,j}^s = p_{j,i}^s$.

\begin{definition}
A homogeneous coherent configuration is called \textit{primitive} if every constituent is strongly connected.
\end{definition}

It is not hard to check that every constituent graph of a homogeneous coherent configuration is stongly connected if and only if it is weakly connected.

Note, that by Observation \ref{obs-color-dist} we have $\dist_{i}(j)\leq r-1$ for any edge colors $i,j$ of a primitive coherent configuration.

The following definition will be useful in Section \ref{sec-coherent}.

\begin{definition}\label{def-assoc-diam} We say that an association scheme has \textit{diameter} $d$ if any non-diagonal constituent has diameter at most $d$ and there exists a non-diagonal constituent of diameter~$d$.
\end{definition}

Note, that if an association scheme has finite diameter, then in particular it is  primitive. Alternatively, every primitive association scheme of rank $r$ has diameter $\leq r-1$. 

\begin{definition}
A regular graph is called \textit{edge-regular} if any pair of adjacent vertices has the same number of common neighbors. A graph is called \textit{co-edge-regular} if its complement is edge-regular. 
\end{definition}

Observe that for every undirected color $i$ constituent $X_i$ is an edge-regular graph.

\subsection{Distance-regular graphs}
\begin{definition}
 A connected graph $X$ of diameter $d$ is called \textit{distance-regular} if for any $0\leq i\leq d$ there exist constants $a_i, b_i, c_i$ such that for any $v\in X$ and any $w\in N_i(v)$ the number of edges between $w$ and $N_i(v)$ is $a_i$, between $w$ and $N_{i+1}(v)$ is $b_i$, and between $w$ and $N_{i-1}(v)$  is $c_i$. The sequence
 \[\iota(X) = \{b_0, b_1,\ldots, b_{d-1}; c_1, c_2,\ldots, c_d\}\]
 is called the \textit{intersection array} of $X$.
\end{definition}

 Note, that for a distance-regular graph $b_d = c_0 = 0$, $b_0 = k$, $c_1 = 1$, $\lambda = a_1$ and $\mu = c_2$. From what is written below it follows that numbers $k_i = |N_i(v)|$ do not depend on vertex $v\in X$. By edge counting, the following straightforward conditions on the parameters of distance-regular graphs hold.
\begin{enumerate}
\item $a_i+b_i+c_i = k$ for every $i$,
\item $k_ib_i = k_{i+1}c_{i+1}$, 
\item $b_{i+1} \leq b_{i}$ and $c_{i+1}\geq c_{i}$ for $0\leq i\leq d-1$.
\end{enumerate}

With any graph of diameter $d$ we can naturally associate matrices $A_{i}\in M_n$, where rows and columns are marked by vertices, with entries $(A_i)_{u,v} = 1$ if and only if $\dist(u,v) = i$. That is, $A_i$ is the adjacency matrix of the distance-$i$ graph $X_i$ of $X$. For a  distance-regular graph they satisfy the relations
\begin{equation}
A_0 = I,\quad A_1 =: A,\quad \sum\limits_{i=0}^d A_i = J,
\end{equation}
\begin{equation}\label{eq-rec}
 AA_i = c_{i+1}A_{i+1}+a_iA_i+b_{i-1}A_{i-1}\quad \text{for } 0\leq i\leq d,
 \end{equation}  
where $c_{d+1} = b_{-1} = 0$ and $A_{-1} = A_{d+1} = 0$. Clearly, Eq. \eqref{eq-rec} implies that for every $0\leq i\leq d$ there exist a polynomial $\nu_i$ of degree exactly $i$, such that $A_i = \nu_i(A)$. Moreover, minimal polynomial of $A$ has degree exactly $d+1$. Hence, since $A$ is symmetric, $A$ has exactly $d+1$ distinct real eigenvalues. Additionally, we conclude that for all $0\leq i,j,s\leq d$ there exist numbers $p_{i,j}^{s}$, such that 
\[A_iA_j = \sum\limits_{s = 0}^{d}p_{i,j}^{s} A_s.\]
Recalling the definition of $A_i$ it implies that for any $u,v\in X$ with $\dist(u,v) = s$ there exist exactly $p_{i,j}^{s}$ vertices at distance $i$ from $u$ and distance $j$ from $v$, i.e., $|N_i(u)\cap N_j(v)| = p_{i,j}^{s}$. 

Therefore, every distance-regular graph $X$ of diameter $d$ induces an association scheme $\mathfrak{X}$ of rank $d+1$, where vertices are connected by an edge of color $i$ in $\mathfrak{X}$ if and only if they are at distance $i$ in $X$ for $0\leq i\leq d$.  Hence, we get the following statement.
\begin{lemma}
If graph $X$ is distance-regular of diameter $d$, then the distance-$i$ graphs $X_i$ form constituents of an association scheme $\mathfrak{X}$ of rank $d+1$ and diameter $d$. In opposite direction, if an association scheme of rank $d+1$ has a constituent of diameter $d$, then this constituent is distance-regular. 
\end{lemma}
Note, that 
\[p_{1,i}^{i-1} = b_{i-1},\quad p_{1,i}^{i} = a_i,\quad p_{1,i}^{i+1} = c_{i+1}\] 
Let $\eta$ be an eigenvalue of $A$, then Eq. \eqref{eq-rec} implies that 
\[\eta \nu_i(\eta) = c_{i+1}\nu_{i+1}(\eta)+a_i\nu_i(\eta)+b_{i-1}\nu_{i-1}(\eta), \text{ so}\]
\[\eta u_i(\eta) = c_{i}u_{i-1}(\eta)+a_i u_i(\eta)+b_{i}u_{i+1}(\eta),\]
where $u_i(\eta) = \frac{\nu_i(\eta)}{k_i}$. Therefore, the eigenvalues of $A$ are precisely the eigenvalues of tridiagonal $(d+1)\times (d+1)$ \textit{intersection matrix} 

\[T(X) = \left(\begin{matrix} 
a_0 & b_0 & 0 & 0 & ... \\ 
c_1 & a_1 & b_1 & 0 &...\\
0& c_2 & a_2 & b_2 & ...\\
...& & \vdots & & ...\\
...& & 0& c_d & a_d  

\end{matrix}\right)\]

In Section \ref{sec-spectral} the exact expressions of $p_{i,j}^s$ for a diameter-3 distance-regular graph will be useful. We summarize the computations in

\begin{proposition}\label{intersection-num-diam-3}
The parameters of a distance-regular graph $X$ of diameter 3 satisfy
\[p_{1,1}^1 = \lambda, \quad p_{1,2}^1 = p_{2,1}^1 = k-\lambda-1, \quad p_{2,2}^1 = \frac{(k-\lambda-1)(k-a-\mu)}{\mu},\]
\[p_{2,3}^1 = p_{3,2}^1 = \frac{a(k-\lambda-1)}{\mu},\quad p_{3,3}^1 = \frac{(k-b)(k-\lambda-1)a}{b\mu},\]
\[p_{1,1}^2 = \mu,\quad p_{1,2}^2 = p_{2,1}^2 = k-a-\mu, \quad p_{1,3}^2 = p_{3,1}^2 = a,\]
\[p_{2,2}^2 = ((k-\lambda-1)+\frac{ba-k+(k-a-\mu)(k-a-\mu-\lambda)}{\mu}),\] 
\[p_{2,3}^2 = p_{3,2}^2 = \frac{((k-b)+(k-a-\lambda)-\mu)a}{\mu},\]
\[p_{3,3}^2 = \left(\frac{(k-b)^2-\lambda(k-b)-k+ab}{\mu}\right)\frac{a}{b},\]
\[p_{1,2}^3 =p_{2,1}^3 = b, \quad p_{1,3}^3 = p_{3,1}^3 = k-b, \quad p_{2,2}^3 = \frac{((k-b)+(k-a-\lambda)-\mu)b}{\mu},\] 
\[p_{2,3}^3 = p_{3,2}^3 = \frac{(k-b)^2-(k-b)\lambda-k+ab}{\mu},\]
where $a = b_2$ and $b = c_3$.
\end{proposition}
\begin{proof}
The parameters could be computed recursively using identity $(AA_i)A_j = A(A_iA_j)$.
\end{proof}

\subsection{Weisfeiler-Leman refinement}

Let $\mathfrak{X} = (V, c)$ be a configuration. A \textit{refinement} of the coloring $c$ is a new coloring $c'$ also defined on $V\times V$ such that if $c'(x) = c'(y)$ for $x,y\in V\times V$, then $c(x) = c(y)$. If coloring $c'$ satisfies condition (ii) of Definition \ref{def-conf}, then $\mathfrak{X}' = (V, c')$ is a refined configuration.

An important example of a refinement was introduced by Weisfeiler and Leman in \cite{Weisfeiler-Leman} in 1968. The \textit{Weisfeiler-Leman refinement} proceeds in rounds. On each round it takes a configuration $\mathfrak{X}$ of rank $r$ and for each pair $(x,y)\in V\times V$ it encodes in a new color $c'(x, y)$ the following information: the color $c(x,y)$,  the number $wl_{i,j}(x, y) = |\{z: c(x,z) = i, c(z,y) = j\}|$ for all $i,j\leq r$. That is, pairs $(x_1,y_1)$ and $(x_2, y_2)$ receive the same color if and only if their colors and all the numbers $wl_{i,j}$ were equal. It is easy to check that for the refined coloring $c'$ structure $\mathfrak{X}' = (V, c')$ is a configuration as well. The refinement process applied to a configuration $\mathfrak{X}$ takes $\mathfrak{X}$ as an input on the first round, and on every subsequent round in takes as an input the output of the previous round . The refinement process stops when it reaches a stable configuration (i.e, $\mathfrak{Y}' = \mathfrak{Y}$). It is easy to see that process will always stop as the number of colors increases after a refinement round applied to a non-stable configuration. One can check that configurations that are stable under this refinement process are precisely coherent configurations. Therefore, the Weisfeiler-Leman refinement process takes any configuration and refine it to a coherent configuration.

The Weisfeiler-Leman refinement is \textit{canonical} in the following sense. Let $\mathfrak{X}$ and $\mathfrak{Y}$ be configurations and $\mathfrak{X}^{wl}, \mathfrak{Y}^{wl}$ be the refined configurations outputted by the process. Then the sets of isomorphisms are equal
\[\Iso(\mathfrak{X}, \mathfrak{Y}) = \Iso(\mathfrak{X}^{wl}, \mathfrak{Y}^{wl}) \subseteq \Sym(V).\]

Other important procedure that can be applied to a configuration is individualization. We say that a configuration $\mathfrak{X}^* = (V, c^*)$ is an \textit{individualization} of $\mathfrak{X}$ at pair $(x,y)\in V\times V$ if $c^*$ is obtained from $c$ by replacing $c(x,y)$ and $c(y,x)$ with the new (possibly equal) colors $c^*(x,y)$ and $c^{*}(y,x)$, which were not in the image of $c$. A configuration individualized at pair $(x,x)$ for $x\in V$ is said to be individualized at vertex $x$. Similarly, it is individualized at the set of vertices $S\subseteq V$ if it is subsequently individualized in each vertex from $S$.

\begin{definition}
Let $\mathfrak{X}$ be a configuration and $S\subseteq V$ be a set of vertices. Denote by $\mathfrak{X}'$ the configuration obtained from $\mathfrak{X}$ by individualization at the set $S$. We say that $S$ \textit{splits $\mathfrak{X}$ completely} with respect to some canonical refinement process $\mathfrak{r}$, if this refinement process applied to $\mathfrak{X}'$ stabilizes only when every vertex from $V$ get a unique color.
\end{definition}

We will use this definition only with respect to the Weisfeiler-Leman refinement process, so we will omit mentioning that in the future.

Observe, that if $S$ splits $\mathfrak{X}$ completely, then the pointwise stabilizer $\Aut(\mathfrak{X})_{(S)}$ is the identity group. Thus, in particular $|\Aut(\mathfrak{X})|\leq n^{|S|}$, where $n = |V|$. Hence, individualization/refinement techniques can be used to bound the order of the automorphism group.

\subsection{Thickness of a group}\label{sec-thickness}

For various applications an important measure of largeness of a group is its thickness.

\begin{definition}[Babai \cite{Babai-str-reg}]\label{def-thickness}
 The \textit{thickness} $\theta(G)$ of a group $G$ is the greatest $t$ such that the alternating group $A_t$ is involved in $G$. A group $H$ is said to be \textit{involved} in $G$ if $H \cong L/N$ for some $N\triangleleft L \leq G$.
\end{definition}

As Babai pointed out, this is an old concept, in \cite{Babai-str-reg} he just coined the term.

Clearly, if $\theta(G) = t$, then $|G|\geq \eee^{t(\ln(t)-1)}$. It turns out that for a primitive permutation group essentially tight upper bound on its order can be given in terms of the thickness by the result of Babai, Cameron and P\'alfy \cite{thickness-primitive}. (The result depends on the Classification of the Finite Simple Groups.)

\begin{theorem}[Babai, Cameron, P\'alfy]\label{BCP-thm}
If $G$ is a primitive permutation group of degree $n$ and thickness $t$, then $|G| = n^{O(t\log(t))}$.
\end{theorem}

\noindent Later \cite{Pyber-improved} Pyber showed improved $n^{O(t)}$ upper bound  in the above theorem.

The minimal degree of a permutation group $G\leq Sym(n)$ can be used to bound the thickness $\theta(G)$. One way to do this is purely combinatorial, through the modified version of Babai-Seress lemma proven in \cite{Babai-Seress}, as it was done in \cite{Babai-str-reg}. There it is shown that if minimal degree of $G$ is at least $\gamma n$, then $\theta(G) = O\left(\frac{\ln(n)^{2}}{\ln\ln(n)}\right)$.  However, as Babai pointed out, more efficient bounds can be achieved using a group theoretic result of Wielandt \cite{Wielandt}.

\begin{theorem}[Wielandt, see {\cite[Theorem 6.1]{Babai-doublytransitive}}]\label{Wielandt-thm}
Let $n>k>\ell$ be positive integers, $k\geq 7$, and let $0<\alpha<1$. Suppose that $G$ is a permutation group of degree $n$ and minimal degree at least $\alpha n$. If
\[\ell(\ell-1)(\ell-2)\geq (1-\alpha)k(k-1)(k-2),\]
and $G$ involves the alternating group $A_k$, then $n\geq \binom{k}{\ell}$.
\end{theorem}
\begin{corollary}\label{Wielandt-cor}
Suppose that $G$ is a permutation group of degree $n$ and minimal degree at least $\alpha n$. Then the thickness of $G$ is at most $\theta(G)\leq 3(1-\alpha)^{-1/3}\ln\left(\frac{1}{1-\alpha}\right)^{-1}\ln(n)$.
\end{corollary}

An analog of Conjectures \ref{conj-1} and \ref{conj-3} in terms of the thickness is the following.

\begin{conjecture}[Babai \cite{Babai-str-reg}]\label{conj-2}
For every $\varepsilon>0$, there is some $N_{\varepsilon}$ such that if $\mathfrak{X}$ is a primitive coherent configuration on $n> N_\varepsilon$ vertices and $\theta(\Aut(\mathfrak{X}))\geq n^{\varepsilon}$, then $\mathfrak{X}$ is a Cameron scheme.
\end{conjecture}

As before, Conjecture \ref{conj-2} for $\varepsilon > \frac{1}{2}$ follows from \cite{Babai-annals} and for $\varepsilon >\frac{1}{3}$ follows from \cite{Sun-Wilmes}. For configurations of rank 3 and 4 and any $\varepsilon>0$ the conjecture follows from Theorem \ref{babai-str-reg-thm} and Theorem~\ref{main-coherent}, respectively.

\subsection{Johnson, Hamming and Cameron schemes}\label{sec-subsec-exceptions}

In this subsection we will define families of graphs and coherent configurations with huge automorphism groups. In particular, for certain range of parameters they have motion sublinear in the number of vertices.

\begin{definition}
Let $t\geq 2$ and $\Omega$ be a set of $m\geq 2t+1$ points. The \textit{Johnson graph} $J(m,t)$ is a graph on a set  of $n = \binom{m}{t}$ vertices  $V(J(m,t)) = \binom{\Omega}{t}$, i.e., vertices are the subsets of size $t$ of $\Omega$. Two vertices of $J(m,t)$ are adjacent if and only if the corresponding subsets $U_1, U_2\subseteq \Omega$ differ by exactly one element, i.e., $|U_1\setminus U_2| = |U_2\setminus U_1| = 1$.
\end{definition}

It is not hard to check that $J(m, t)$ is distance-transitive, and so it is a distance-regular graph of diameter $t$ with the intersection numbers
\[b_i = (t-i)(m-t-i) \quad \text{and} \quad c_{i+1} = i^{2}, \quad \text{for } 0\leq i<t.\] 
In particular, $J(m, t)$ is regular of degree $k = t(m-t)$ with $\lambda = m-2$ and $\mu = 4$.  The eigenvalues of $J(m,t)$ are
\[ \xi_j = (t-j)(m-t-j)-j \quad \text{with multiplicity}\quad \binom{m}{j} - \binom{m}{j-1}, \, \text{for } 0\leq j\leq t.\]

The automorphism group of $J(m,t)$ is the Johnson group $S_m^{(t)}$, which acts on $\binom{\Omega}{t}$ via the induced action of $S_m$ on $\Omega$. Indeed, it is obvious, that $S^{(t)}_m\leq \Aut(J(m,t))$. An opposite inclusion can be derived from the Erd\H{o}s-Ko-Rado theorem.

Thus, for fixed $t$ we get that the order is $|Aut(J(m,t)| = m! = \Omega(\exp(n^{\frac{1}{t}}))$, the thickness satisfy $\theta(\Aut(J(s, m))) = m = \Omega(n^{1/t})$ and  
%$2\binom{m-2}{t-1} 
$$\motion(J(m,t)) = O(n^{1-1/t}).$$

The association scheme $\mathfrak{J}(m,t)$ induced by $J(m,t)$ is called \textit{Johnson scheme}.

\begin{definition}
Let $t\geq 2$ and $\Omega$ be a set of $m\geq 2t+1$ points. The \textit{Johnson scheme} $\mathfrak{J}(m,t)$ is an association scheme of rank $t+1$ on a set  of $n = \binom{m}{t}$ vertices  $V = \binom{\Omega}{t}$. Two vertices $v_1$ and $v_2$ of $\mathfrak{J}(m,t)$ are connected by an edge of color $i$ if and only if the corresponding subsets $U_1, U_2\subseteq \Omega$ differ by exactly $i$ elements, i.e., $|U_1\setminus U_2| = |U_2\setminus U_1| = i$.
\end{definition}

It is not hard to check that $\mathfrak{J}(m,t)$ is primitive and $\Aut(\mathfrak{J}(m,t)) =\Aut(J(m,t))$. 

\begin{definition}
 The Johnson graph $J(s,2)$ is also called the \textit{triangular graph} and is denoted by $T(s)$, where $s\geq 5$.
\end{definition} 

\begin{definition}
Let $\Omega$ be a set of $m\geq 2$ points. The \textit{Hamming graph} $H(s, m)$ is a graph on a set of $n = m^s$ vertices $V = \Omega^{s}$. Two vertices in $H(s,m)$ are adjacent if and only if the corresponding $s$-tuples $v_1, v_2$ differ in precisely one position, or in other words, if the Hamming distance $d_H(v_1, v_2)$ for the corresponding tuples is equal 1.
\end{definition}
Again, it is not hard to check that $H(s,m)$ is distance-transitive, so is a distance-regular graph of diameter $s$ with the intersection numbers
\[ b_i =(s-i)(m-1) \quad \text{and} \quad c_{i+1} = i \quad \text{for } 0\leq i\leq s-1.\]
In particular, $H(s,m)$ is regular of degree $k = s(m-1)$ with $\lambda = m-2$ and $\mu = 2$. The eigenvalues of $H(s, m)$ are
\[ \xi_j = s(m-1) - jm \quad \text{with multiplicity}\quad \binom{s}{j}(m-1)^{j}, \, \text{for } 0\leq j\leq s.\]

The automorphism group of $H(s, m)$ is isomorphic to $S_m\wr S_s$. Hence,  its order is $|\Aut(H(s, m))| = (m!)^s s!$, the thickness satisfies $\theta(H(s,m)) \geq m = n^{\frac{1}{s}}$ and $$\motion(H(s, m))\leq 2m^{s-1} = O(n^{1-1/s}).$$

The association scheme $\mathfrak{H}(s,m)$ induced by the Hamming graph $H(s,m)$ is called \textit{Hamming scheme}.

\begin{definition}
Let $\Omega$ be a set of $m\geq 2$ points. The \textit{Hamming scheme} $\mathfrak{H}(s, m)$ is an association scheme of rank $s+1$ on a set of $n = m^s$ vertices $V = \Omega^{s}$. Two vertices in $\mathfrak{H}(s,m)$ are connected by an edge of color $j$ if and only if the corresponding $s$-tuples $v_1, v_2$ differ in precisely $i$ positions, or in other words, if the Hamming distance $d_H(v_1, v_2)$ for the corresponding tuples is equal $j$.
\end{definition}

Again, one can check that $\mathfrak{H}(s,m)$ is a primitive coherent configuration and $\Aut(H(s,m)) = \Aut(\mathfrak{H}(s, m))$.

As for Johnson graphs, the graph $H(2, m)$ has a special name and notation.
\begin{definition}
The Hamming graph $H(2, m)$ is called the \textit{lattice graph} and is denoted by $L_2(m)$, where $m\geq 2$.
\end{definition}

Observe, both the triangular graph $T(s)$ and the lattice graph $L_{2}(m)$ are strongly regular with the smallest eigenvalue $-2$. Notice also that the  smallest eigenvalue of $J(m, t)$ and $H(t, m)$ is $-t$.

Let $G\leq Sym(\Omega)$ be a permutation group. The orbits of the $G$-action on $\Omega\times\Omega$ are called \textit{orbitals} of $G$. 

\begin{definition}\label{def-orbital}
For $G\leq Sym(\Omega)$ with orbitals $R_1, R_2, ..., R_k$ define \textit{orbital configuration} as $\mathfrak{X}(G) = (\Omega, \{R_1, R_2, ..., R_k\})$. A configuration is called \textit{Schurian} if it is equal to $\mathfrak{X}(G)$ for some permutation group $G$.
\end{definition}

Observe that the configuration $\mathfrak{X}(G)$ is coherent. By the discussion above, Johnson and Hamming schemes are Schurian configurations.

Using the Classification of the Finite Simple Groups, Cameron in \cite{Cameron} classified all primitive groups of degree $n$ of order at least $n^{c\log\log n}$. We state here Mar\'oti's refined version of this result.

\begin{definition}[Cameron groups]\label{def-Cameron-groups}
Let $G$ be a primitive group. Let $m,k,d$ be positive integers and $(A_{m}^{(k)})^{d}\leq G\leq S_{m}^{(k)}\wr S_d$, where $S_{m}^{(k)}\wr S_d$ has the product action on $\binom{m}{k}^d$ elements. Then group $G$ is called a \textit{Cameron group}.

 The corresponding orbital configuration $\mathfrak{X}(G)$ is called a \textit{Cameron scheme}.

\end{definition}

\begin{theorem}[Cameron \cite{Cameron}, Mar\'oti \cite{Maroti}]\label{cameron-maroti}
If $G$ is a primitive permutation group of degree $n>24$, then one of the following is true.
\begin{enumerate}
\item $G$ is a Cameron group.
\item $|G|\leq n^{1+\log(n)}$. 
\end{enumerate}
\end{theorem}

Cameron groups appear as an exceptions in other similar classifications of ``large'' primitive groups. We mention the classification result by Liebeck and Saxl \cite{Liebeck-Saxl} in terms of the minimal degree .

%\begin{definition}
%Set $S$ is called a \textit{base} of a permutation group $G$ if pointwise stabilizer of $S$ in $G$ is trivial. The smallest size of a base of $G$ is denoted by $b(G)$.
%\end{definition}

\begin{theorem}[Liebeck, Saxl \cite{Liebeck-Saxl}]
If $G$ is a primitive permutation group of degree $n$, then one of the following is true.
\begin{enumerate}
\item $G$ is a Cameron group.
\item $\mindeg(G)\geq n/3$. 	 
\end{enumerate}
\end{theorem} 

\begin{lemma}\label{cameron-min-deg}
Let $G$ be a Cameron group with $(A_m^{k})^d\leq G\leq S_m^{(k)}\wr S_d$ which acts on $n = \binom{m}{k}^d$ points, where $k\leq m/2$. Then as $m\rightarrow \infty$, the following holds uniformly in $d$: we have $\mindeg(G) = o(n)$ if and only if $k = o(m)$. 
\end{lemma}
\begin{proof}
It is not hard to see that the minimal degree of $G$ is realized by the induced action of a cycle of length 2 or 3 (in $S_m$ or $A_m$) on $k$-subsets in just one of $d$ coordinates.
If there is a 2-cycle action in a coordinate, then the minimal degree of $G$ is
\[\left(\binom{m}{k} - \binom{m-2}{k} - \binom{m-2}{k-2}\right)\binom{m}{k}^{d-1},\]
otherwise, the minimal degree of $G$ is
\[\left(\binom{m}{k} - \binom{m-3}{k} - \binom{m-3}{k-3}\right)\binom{m}{k}^{d-1}.\]
As $m\rightarrow \infty$ these expressions are equal
\[ n\cdot\left(1-\frac{(m-k)^2+k^2}{m^2}+o(1)\right) \quad \text{and} \quad  n\cdot \left(1-\frac{(m-k)^3+k^3}{m^3}+o(1)\right),\]
respectively. Clearly, each of these expressions is $o(n)$ if and only if $k = o(m)$. 

\end{proof}

\subsection{Approximation tools}
In Sections \ref{sec-drg-approx} and \ref{sec-coh-approx} we will use the following results from the Approximation theory which allow us to estimate the roots of a perturbed polynomial and the eigenvalues of a perturbed matrix.

\begin{theorem}[{\cite[Appendix A]{Ostrowski}}]\label{poly-approx}
Let $n\geq 1$ be an integer. Consider two polynomials of degree $n$
\[ f(x) = a_0x^n+...+a_{n-1}x+a_n, \quad g(x) = b_0x^n+...+b_{n-1}x+b_n,\]
where $a_0 = b_0 = 1$. Denote $M = \max  \{|a_i|^{1/i}, |b_i|^{1/i}:  0\leq i\leq n\}$ and
\[\varepsilon = 2n \left( \sum\limits_{i = 1}^{n} |b_i - a_i| (2M)^{n-i}\right)^{1/n}.\]
Let $x_1, x_2, ..., x_n$ denote the roots of $f$ and $y_1, y_2, ..., y_n$ denote the roots of $g$. Then, there exists a permutation $\sigma \in S_n$ such that for every $1\leq i\leq n$
\[|x_i - y_{\sigma(i)}|\leq \varepsilon.\]

\end{theorem}

\begin{theorem}[Ostrowski {\cite[Appendix K]{Ostrowski}}]\label{matrix-approximation}
Let $A, B \in M_n(\mathbb{C})$. Let $\lambda_1, \lambda_2, ..., \lambda_n$ be the roots of the characteristic polynomial of $A$ and $\mu_1, \mu_2, ..., \mu_n$ be the roots of the characteristic polynomial of $B$. Consider 
\[M = \max\{|(A)_{ij}|, |(B)_{ij}|: 1\leq i, j\leq n\}, \quad \delta = \frac{1}{nM}\sum\limits_{i = 1}^n\sum\limits_{j=1}^n|(A)_{ij}-(B)_{ij}|.\]

Then, there exists a permutation $\sigma \in S_n$ such that 
\[|\lambda_i - \mu_{\sigma(i)}|\leq 2(n+1)^2M\delta^{1/n}, \quad \text{for all } 1\leq i\leq n.\]
\end{theorem}

\section{Distance-regular graphs}\label{sec-dist-reg}

\subsection{Case of very large vertex degree}\label{sec-large-deg}

In this section we consider the case of a very large vertex degree of a distance-regular graph~$X$. More precisely, we assume that the vertex degree is linear in the number of vertices, that is, $k>\gamma n$ for some fixed $\gamma>0$.

\begin{lemma}\label{max-min}
Let $X$ be a distance-regular graph of diameter $d\geq 2$. Then
\begin{enumerate}
\item The parameters of $X$ satisfy $k-\mu\leq 2(k-\lambda)$.
\item If $a_2\neq 0$, then they also satisfy $k-\lambda\leq 2(k-\mu)$.
\end{enumerate}
\end{lemma}
\begin{proof}
For any two distinct vertices $v,w$ in $X$ we have $N(v)\cap N(w)\in \{0,\lambda,\mu\}$. At the same time, for any vertex $u$
\begin{equation}\label{eq1}
N(v)\setminus N(w)\subseteq (N(v)\setminus N(u))\cup (N(u)\setminus N(w)).
\end{equation}
Now, if $v,w$ are at distance 2 and $u$ is their common neighbor, we get $k-\mu\leq 2(k-\lambda)$. 

Suppose, that $a_2\neq 0$, then for vertex $u$ there are two adjacent vertices $v,w$ at distance 2 from $u$. So, the inclusion for neighborhoods above gives $k-\lambda\leq 2(k-\mu)$. 
\end{proof}

We will also need a straightforward generalization of part 1 of the previous lemma.

\begin{lemma}\label{lambda-mu-ineq}
Let $X$ be a regular graph of degree $k$, such that every pair of adjacent vertices has at least $\lambda$ common  neighbors, while every pair of non-adjacent vertices has at most $\mu$ common neighbors. Then $X$ is a disjoint union of cliques, or $k+\mu\geq 2\lambda$.
\end{lemma}
\begin{proof}
If $X$ is not a union of cliques, then there are two vertices $v,w$ at distance 2 in $X$, and let $u$ be their common neighbor. Then inequality follows from the inclusion in Eq.~\eqref{eq1}. 
\end{proof}

\begin{lemma}
Let $X$ be a non-trivial distance-regular graph. Then every pair of distinct vertices is distinguished by at least $k-\mu$ vertices.
\end{lemma}
\begin{proof} Any two vertices $u,v\in X$ are distinguished by at least $|N(u)\bigtriangleup N(v)| = 2(k-|N(u)\cap N(v)|)$ vertices. Now, by Lemma \ref{max-min}, we get $2(k-\max(\lambda,\mu))\geq k-\mu$.
\end{proof}

\begin{corollary}\label{linear-mu-bound}
Let $X$ be a non-trivial distance-regular graph. Fix some constants $\gamma, \delta>0$. If $k>n\gamma$ and $\mu\leq (1-\delta)k$, then $\motion(X)\geq \gamma\delta n$. 
\end{corollary}
\begin{proof}
The result follows from the lemma above and Observation \ref{obs1}. 
\end{proof}

Hence, we are going to bound $\mu$ and $\lambda$ away from $k$.

\begin{lemma}\label{min-est}
The parameters of a distance-regular graph of diameter $d\geq 2$ satisfy 
\begin{enumerate}
\item $\min(\lambda, \mu)<k(1+\min(\frac{r-1}{d-1}, (\frac{r}{d})^{\frac{1}{d-1}}))^{-1}$,
\item $\mu <k(\min(\frac{r-1}{d-1}, (\frac{r}{d})^{\frac{1}{d-1}}))^{-1}$,
\end{enumerate}
where $r = \frac{n-1}{k}$.
\end{lemma}
\begin{proof}
Recall that the sequences $(b_i)$ and $(c_i)$ are monotone, so $b_i\leq b_1 = k-\lambda-1$ and $c_{i+1}\geq \mu$ for $1\leq i\leq d-1$.
Thus, $k_{i+1}\leq k_i\frac{k-\lambda-1}{\mu}$. Hence, we have $$n = \sum\limits_{i=0}^{d}k_i\leq 1+\sum\limits_{i=0}^{d-1}k\left(\frac{k-\lambda-1}{\mu}\right)^i$$.

If $k-\lambda-1\leq \mu$, then $n\leq 1+k+(d-1)k\frac{k-\lambda-1}{\mu}$, i.e., $\left(\frac{r-1}{d-1}\right)\mu  +\lambda+1\leq k$. Otherwise, we have $n\leq 1+dk\left(\frac{k-\lambda-1}{\mu}\right)^{d-1}$, i.e., $\left(\frac{r}{d}\right)^{\frac{1}{d-1}}\mu+\lambda+1\leq k$.
\end{proof}

\begin{corollary}\label{min-lambda-d3}
Let $X$ be a distance-regular graph of diameter $d\geq 3$. Then $\min(\lambda, \mu)\leq \frac{d-1}{d}k$.
\end{corollary}
\begin{proof} Since diameter $d\geq 3$, we have $2\leq r = \frac{n-1}{k}$, so the result follows from the previous lemma. 
\end{proof}

\begin{lemma}\label{a2-lambda}
Let $X$ be a distance-regular graph of diameter $d\geq 3$ and $a_2 = 0$. Then $\lambda = 0$.
\end{lemma}
\begin{proof}
Fix any vertex $x\in X$. Take $v\in N_2(x)$. Denote the set of neighbors of $v$ in $N_1(x)$ by $S(v)$. Then $|S(v)| = \mu$. Observe, that common neighbors of $v$ and $w\in S(v)$ are in $N_1(x)$ or in $N_2(x)$. Moreover, the condition $a_2 = 0$ force them to be in $S(v)$ and there are exactly $\lambda$ such neighbors. At the same time, $w$ has exactly $\lambda$ neighbors in $N_1(x)$, i.e., all neighbors of $w\in S(v)$ within $N_1(x)$ lie inside $S(v)$.

Suppose that $\lambda\neq 0$, then there are two adjacent vertices $w_1, w_2\in S(v)$. Moreover, for any $w_3\in N_1(x)$ we have $\dist(w_i,w_3)\leq 2$ for $i = 1,2$. And distance could not be 2 for both, as $a_2 = 0$. Therefore, $w_3$ is adjacent to at least one of $w_1, w_2$. However, by the argument above, any such $w_3$ is inside $S(v)$, i.e., $S(v) = N_1(x)$, or in another words $\mu = k$. But it means that there is no vertex at distance $3$ from $x$, so we get a contradiction with the assumtion that diameter is at least $3$.
\end{proof}

\begin{proposition}\label{prop-k-big}
Let $X$ be a distance-regular graph of diameter $d\geq 3$. Suppose $k>n\gamma>2$ for some $\gamma>0$. If $X$ is not a bipartite graph, then $\motion(X)$ is at least $\frac{\gamma}{3} n$.
\end{proposition}
\begin{proof}
Suppose, that $\mu\leq\frac{2}{3}k$, then result follows from the Corollary \ref{linear-mu-bound}. 

If diameter $d\geq 4$, then $\mu\leq \frac{k}{2}$. Indeed, let $v,w$ be vertices at distance $4$ and let $y$ be a vertex at distance 2 from each of them. Then $y$ and $v$ have $\mu$ common neighbors and $y$ and $w$ have $\mu$ common neighbors, and they are all distinct. At the same time, they all are neighbors of $y$, so $\mu\leq k/2$.

If $d = 3$ and $a_2>0$, then by Lemma \ref{max-min} the inequality $k-\min(\lambda, \mu)\leq 2(k-\max(\lambda,\mu))$ holds. Hence, any two vertices $u,v\in X$ are distinguished by at least $$|N(u)\bigtriangleup N(v)| = 2(k-|N(u)\cap N(v)|)\geq 2(k-\max(\lambda,\mu))\geq k-\min(\lambda, \mu)$$ vertices. Moreover, Corollary \ref{min-lambda-d3} for $d=3$ gives $\min(\lambda, \mu)\leq \frac{2}{3}k$. 

Finally, assume $d=3$, $a_2 = 0$ and $\mu > \frac{2}{3}k>1$. Lemma 5.4.1 in \cite{BCN}, states that for a distance-regular graph of diameter $d\geq 3$, if $\mu>1$, then either $c_3\geq 3\mu/2$, or $c_3\geq \mu+b_2 = k-a_2$. Thus, we get that $c_3\geq k$, i.e., $a_3 = 0$. Hence, by Lemma \ref{a2-lambda} graph $X$ is bipartite.
\end{proof}

\subsection{Case of bipartite graphs of diameter 3}\label{sec-bipartite}

In this section we consider the case of a bipartite distance-regular  graph of diameter $3$.

\begin{definition}
Let $X$ be a finite graph with adjacency matrix $A$. Suppose that the set of vertices is partitioned into non-empty sets $X_1, X_2, ..., X_m$. The \textit{quotient matrix} $Q\in M_m(\mathbb{R})$ with respect to this partition is defined as $(Q)_{i,j} = \frac{1}{|X_i|}\mathds{1}_{X_i}^{T}A\mathds{1}_{X_j}$.
\end{definition}
\begin{definition}
We say that a non-increasing sequence $b_1\geq b_1\geq ... \geq b_m$ interlace a sequence $a_1\geq a_2\geq ... \geq a_n$ for $n\geq m$, if for any $i\in [m]$ inequality $a_{i}\geq b_{i} \geq  a_{n-m+i}$ holds.
\end{definition}

To treat the bipartite case we will need the version of the Expander Mixing Lemma for regular bipartite graphs. We include the proof for completeness.

\begin{theorem}[Expander Mixing Lemma: bipartite version, Haemers {\cite[Theorem 5.1]{Haemers}}] Let $X$ be a biregular bipartite graph with parts $U$ and $W$ of sizes $n_U$ and  $n_W$. Denote, by $d_U$ and $d_W$ the degrees of the vertices in parts $U$ and $W$, respectively. Let $\lambda_2$ be the second largest eigenvalue of the adjacency matrix $A$  of $X$. Then for any $S\subseteq U$, $T \subseteq W$
\[\left(E(S,T)\frac{n_U}{|S|} - |T|d_{W}\right)\left(E(S,T)\frac{n_W}{|T|} - |S|d_U\right)\leq \lambda_2^{2}(n_U - |S|)(n_W-|T|),\]
which, using $d_Un_U = d_Wn_W = E(U, W)$, implies

\[\left||E(S,T)| - \frac{d_W |S||T|}{n_U}\right|\leq |\lambda_2|\sqrt{|S||T|},\]
where $E(S, T)$ is the set of edges between $S$ and $T$ in $X$.
\end{theorem} 
\begin{proof}
Consider the quotient matrix with respect to the partition $S, (U\setminus S), T, (W\setminus T)$. 
\[
Q = \left(
\begin{matrix}
0& 0& \frac{E(S,T)}{|S|}& d_{U} - \frac{E(S,T)}{|S|} \\
0& 0& \frac{|T|d_W - E(S,T)}{n_U - |S|}& d_U - \frac{|T|d_W - E(S,T)}{n_U - |S|} \\
\frac{E(S,T)}{|T|}& d_{W} - \frac{E(S,T)}{|T|}& 0& 0 \\
\frac{|S|d_U - E(S,T)}{n_W - |T|}& d_W - \frac{|S|d_U - E(S,T)}{n_W - |T|}& 0& 0 
\end{matrix} \right)
\]
Denote the eigenvalues of $X$ by $\lambda_1\geq \lambda_2\geq ...\geq \lambda_{n}$, where $n =n_U+n_W$. The eigenvalues of bipartite graph are symmetric with respect to 0, that is, $\lambda_{n+1-i} = -\lambda_i$ for all $i\in [n]$. Denote the eigenvalues of $Q$ by $\mu_1\geq \mu_2\geq \mu_3\geq \mu_4$. It is easy to see that 
\[\lambda_1 = \mu_1 = -\lambda_n = -\mu_4 = \sqrt{d_Ud_W}.\]
It is known, that the eigenvalues of a quotient matrix with respect to some partition of a graph interlace the eigenvalues of a graph itself. Interlacing implies
\begin{equation}\label{interlace}
\frac{det(Q)}{d_Ud_W} = -\frac{det(Q)}{\mu_1\mu_4} = -\mu_2\mu_3 \leq -\lambda_2\lambda_{n-1} = (\lambda_2)^{2}.
\end{equation}

Therefore, statement of the theorem follows from \eqref{interlace}, as 
\[\det(Q) = \frac{d_Ud_W}{(n_U - |S|)(n_W-|T|)}\left(E(S,T)\frac{n_U}{|S|} - |T|d_{W}\right)\left(E(S,T)\frac{n_W}{|T|} - |S|d_U\right).\]

\end{proof}

Thus, for bipartite graphs we get the analog of Lemma \ref{mixing-lemma-tool} proven by Babai in~\cite[Proposition 12]{Babai-str-reg}. 

\begin{lemma}\label{bip-mixing-lemma-tool}
Let $X$ be a $k$-regular bipartite graph with parts $U$ and $W$ of size $n/2$ each. Let $\lambda_2$ be the second largest eigenvalue of $A$. Moreover, suppose that any two distinct vertices in $X$ have at most $q$ common neighbors. Then any non-trivial automorphism $\sigma$ of $X$ has at most $\frac{k+|\lambda_2|+q}{2k}n$ fixed points. 
\end{lemma}
\begin{proof}
Take any non-trivial automorphism $\sigma$ of $X$. Consider $S_1\subseteq U$ and $S_2\subseteq W$, such that  $S_1\cup S_2 = \supp(\sigma) = \{x\in X | x^{\sigma} \neq x\}$ be the support of $\sigma$. Without lost of generality, we may assume that $|S_1|\geq |S_2|$. Denote $S = S_1$ and let $T\subset W$ satisfies $S_2\subseteq T$ and $|T| = |S|$. By Expander Mixing Lemma we get 
\[ \frac{|E(S, T)|}{|S|}\leq |\lambda_2|+k\frac{2|S|}{n}.\]
Hence, there exist a vertex $x$ in $S$ which has at most $|\lambda_2|+k\frac{2|S|}{n}$ neighbors in $T$. Thus, $x$ has at least $k - (|\lambda_2|+k\frac{2|S|}{n})$ neighbors in $W\setminus T$, and they all are fixed by $\sigma$. Therefore, they all are common neighbors of $x$ and $x^{\sigma} \neq x$. We get the inequality $q\geq k - (|\lambda_2|+k\frac{2|S|}{n})$, which is equivalent to 
$\frac{|\lambda_2|+q}{k}\frac{n}{2}\geq \frac{n}{2}-|S|$. By definition of $S$ and $T$ the number of fixed points of $\sigma$ is at most
\[n-|S_1|-|S_2|\leq n-|S| \leq \left(\frac{1}{2}+\frac{|\lambda_2|+q}{2k}\right)n.\]
\end{proof}

\begin{fact}[\cite{BCN}, p. 432]\label{fact1}
Let $X$ be a bipartite distance-regular graph of diameter $3$. Then $X$ has the intersection array
$$\iota(X) = \{k, k-1, k-\mu; 1, \mu, k\}.$$
The eigenvalues of $X$ are $k$, $-k$ with multiplicity 1, and $\pm \sqrt{k-\mu}$ with multiplicity $\frac{n}{2}-1$ each. 
\end{fact}

\begin{definition}
Graph $X$ is called a \textit{cocktail-party graph} if it is obtained from a regular complete bipartite graph by deletion of one perfect matching. 
\end{definition}

\begin{theorem}\label{main-bipartite}
Let $X$ be a $k$-regular bipartite distance-regular graph of diameter $3$ for $k\geq 400$. Then $X$ is a cocktail-party graph, or $\motion(X)\geq \frac{1}{100}n$. 
\end{theorem}
\begin{proof}
Denote the parts of the bipartite graph $X$ by $U$ and $W$. We consider 2 cases. Let $Y = X_3$ be a distance-3 graph of $X$.

Case 1. Suppose that $Y$ is disconnected. Then there exist a pair of vertices $u,v$ in one of the parts, say $F$, so that $u$ and $v$ lie in different connected components of $Y$. Clearly, $\dist(u,v) = 2$, so $p^{2}_{3,3} = 0$. Hence, $k_3 = 1$ and the pairs of vertices at distance $3$ form a perfect matching. Thus $X$ is a regular complete bipartite graph with one perfect matching deleted. 

Case 2. $Y$ is connected and so is itself distance-regular of diameter $3$. Note, that $k+k_3 = n/2$, so if necessary by passing to $Y$, we may assume that $k\leq n/4$. Lemma~\ref{min-est} implies $\mu<\frac{\sqrt{3}}{2}k$. By Fact \ref{fact1} the second largest eigenvalue is $\lambda_2 = \sqrt{k-\mu}$. So the statement of the theorem follows from the Lemma~\ref{bip-mixing-lemma-tool} and inequality
\[\frac{99}{100}k\geq \frac{(\sqrt{3}+2)}{4}k+\sqrt{k} \quad \text{for } k\geq 400.\]

\end{proof}

\subsection{Spectral gap analysis for the diameter-3 case}\label{sec-spectral}

Assume that a distance-regular graph does not satisfy the assumption of Section \ref{sec-large-deg}, that is, we assume here $k\leq \gamma n$ for some fixed $\gamma>0$. Then, by Lemma \ref{min-est}, we get upper bound on $\mu$. If $\gamma$ is small enough, we get that $\mu$ is small comparable to $k$, which allows us to analyze the eigenvalues of $X$ effectively and use Lemma \ref{mixing-lemma-tool}.

For the convenience of computations denote $a = b_2$ and $b = c_3$. Let $\xi_1, \xi_2, \xi_3$ be the non-trivial eigenvalues of $X$ (i.e., that do not correspond to the all-ones vector). By computing the characteristic polynomial of intersection matrix $T(X)$, introduced in Section \ref{sec-prelim}, we get
\begin{equation}\label{eq-f-p}
((a+b)-k)k+(k-b)\mu = \xi_1\xi_2\xi_3
\end{equation}
\begin{equation}
(k-(a+b))\lambda - (k-b)\mu -k+\mu = \xi_1\xi_2+\xi_2\xi_3+\xi_3\xi_1
\end{equation}
\begin{equation}
(k-(a+b))+(\lambda-\mu) = \xi_1+\xi_2+\xi_3
\end{equation}

Let us order the eigenvalues so that $|\xi_1|\geq|\xi_2|\geq|\xi_3|$.

\begin{obs}
Eq. \eqref{eq-f-p} tells us that $|\xi_1\xi_2\xi_3|\leq  2k^2$,  thus $|\xi_3|\leq (2k)^{2/3}$.
\end{obs}

In what follows, we denote by $\varepsilon$ some small positive number, say $0<\varepsilon\leq \frac{1}{100}$.

\begin{proposition}\label{eigen-approxim-d-3}
Suppose that $\mu\leq\frac{\varepsilon^2}{2}k$. Denote $x = k-(a+b)$ and $y = \lambda-\mu$. Then $\min\{|x-\xi_1|, |y-\xi_1|\}\leq \varepsilon k$ and $\min\{|x-\xi_2|, |y-\xi_2|\}\leq \varepsilon k$ for $k\geq (10\varepsilon^{-2})^3$.
\end{proposition}
\begin{proof}
Recall that $xy-a\mu-k+\mu = \xi_1\xi_2+\xi_2\xi_3+\xi_3\xi_1$ and 
$x+y = \xi_1+\xi_2+\xi_3$. In other words, $|xy-\xi_1\xi_2| \leq \frac{\varepsilon^2}{2}k^2+(2k)^{5/3}+k$ and $|x+y-\xi_1-\xi_2| \leq (2k)^{2/3}$. Therefore, since $|\xi_2|\leq k$ we get $|xy-(x+y)\xi_2+\xi_2^2|-(2k)^{2/3}k\leq \frac{\varepsilon^2}{2}k^2+(2k)^{5/3}+k$. Hence, $|x-\xi_2||y-\xi_2|\leq \varepsilon^2k^2$ for $k\geq (10\varepsilon^{-2})^3$. Thus, $\min\{|x-\xi_2|, |y-\xi_2|\}\leq \varepsilon k$. But the initial equations were symmetric in $\xi_1, \xi_2$, so similarly we get  $\min\{|x-\xi_1|, |y-\xi_1|\}\leq \varepsilon k$.
\end{proof}
\begin{remark}
Similar result can be also derived from Theorem \ref{poly-approx}.
\end{remark}
\begin{corollary}
Suppose that $\mu\leq\frac{\varepsilon^2}{2}k$, then $|\xi_1|\leq \max\{|k-(a+b)|+\varepsilon k, (\lambda +\varepsilon k)\}$ for $k\geq(10\varepsilon^{-2})^3$.
\end{corollary}

\begin{obs}\label{obs2}
Suppose that $\lambda<(1/2-\varepsilon)k$, $\mu\leq\frac{\varepsilon^2}{2}k$ and  $k\geq(10\varepsilon^{-2})^3$, then with the notation of Lemma \ref{mixing-lemma-tool} we get  $q+\xi<\max(k-\varepsilon k, |k-(a+b)|+\varepsilon k+\max(\lambda,\mu))$. 
\end{obs}
Thus, our goal for now will be to bound the second term out from $k$.

\begin{proposition}\label{Prop-a-b-big}
Let $X$ be a distance-regular graph of diameter $3$ with $\mu\leq \frac{\varepsilon^2}{2} k$. Assume that $\delta = 2\varepsilon+\varepsilon^2/2\leq \frac{1}{100}$ and $k\geq 100$,  then $X$ is bipartite or $$(a+b)-k+\varepsilon k+\max(\lambda,\mu)\leq k-\varepsilon k.$$
\end{proposition}
\begin{proof}
Suppose $(a+b)-k+\varepsilon k+\max(\lambda,\mu)>k-\varepsilon k$. Then, after replacing maximum by the sum we get $((k-\lambda)-a)+(k-b)<(2\varepsilon+\varepsilon^2/2) k=\delta k$. Thus, as $a\leq k-\lambda-1$ and $b\leq k$, we get
\begin{equation}\label{eq-prop321}
b>k(1-\delta),\quad  a>(k-\lambda)-\delta k\quad \text{ and}\quad \mu<\delta k.
\end{equation}

Let $X_2$ be a distance-2 graph of $X$, recall that $V(X_2) = V(X)$ and two vertices in $X_2$ are adjacent if and only if they are at distance 2 in $X$. Additionally, recall that $\lambda(X_2)$ is the number of common neighbors of two adjacent vertices in $X_2$ and $\mu(X_2)$ is the maximal number of common neighbors of two non-adjacent vertices in $X_2$. We want to show for $\delta\leq \frac{1}{100}$, if $X_2$ is not a union of cliques, then our assumptions lead us to a contradiction with the inequality $2\lambda(X_2)\leq k_2+\mu(X_2)$ from Lemma \ref{lambda-mu-ineq}. If $X_2$ is a union of cliques, then being at distance $2$ in $X$ is an equivalence relation, which implies that $X$ is bipartite.

Recall that by Lemma \ref{max-min} we have $\lambda\leq \frac{k+\mu}{2}\leq \frac{2}{3}k$. In particular, it implies  $k-\lambda-\delta k\geq (1-3\delta)(k-\lambda)$. Additionally, observe that, by Eq. \ref{eq-prop321}, 
$$(k-a-\mu)(k-a-\mu-\lambda)\geq k(-2\delta k).$$ 
Thus, by Proposition \ref{intersection-num-diam-3}, we have the following estimates
\[k_2 = \frac{k(k-\lambda-1)}{\mu}\leq \frac{k(k-\lambda)}{\mu},\]
\[\lambda(X_2) = p_{2,2}^2 >\frac{k(1-\delta)(k-\lambda)(1-3\delta)-2\delta k^2-k}{\mu}, \]
\[p_{2,2}^1<\frac{(\lambda+\delta k)(k-\lambda)}{\mu}\leq \left(\frac{2}{3}+\delta\right)\frac{k(k-\lambda)}{\mu}, \quad p_{2,2}^3<\frac{2\delta k^2}{\mu}\leq \frac{6\delta k(k-\lambda)}{\mu}.\]
Hence, as $k>100$ implies $k\leq \frac{3}{100}k(k-\lambda)$, we get
\[\lambda(X_2)\frac{\mu}{k(k-\lambda)}\geq (1-\delta)(1-3\delta)-6\delta-\frac{3}{100} \geq \frac{97}{100}-10\delta,\]
\[(k_2+\mu(X_2))\frac{\mu}{k(k-\lambda)}\leq 1+\frac{2}{3}+\delta.\]
Therefore, for $\delta \leq \frac{1}{100}$ we have $2\lambda(X_2)>k_2+\mu(X_2)$, so if $X$ is not bipartite we get a contradiction.

\end{proof}

 For the next proposition, the following simple observation will be useful.
\begin{obs}\label{aplusb}
We have $a+b\geq \lambda+2$.
\end{obs}
\begin{proof}
Fix some $x\in X$ and take adjacent $v\in N_2(x)$ and $w\in N_3(x)$. Then $v$ and $w$ have exactly $\lambda$ common neighbors and clearly all of them lie either in $N_3(x)$ or in $N_2(x)$. At the same time, $w$ has $b-1$ neighbors in $N_2(x)$ different from $v$ and $v$ has $a-1$ neighbors in $N_3(x)$ different from $w$. Thus, $a+b\geq \lambda+2$.
\end{proof}

\begin{proposition}\label{Prop-a-b-small}
Let $X$ be a distance-regular graph of diameter $3$ with $\mu\leq \frac{\varepsilon^2}{2} k$. Suppose $\delta = 2\varepsilon+\varepsilon^2/2\leq \frac{1}{100}$ and $k\geq 100$. Then $\lambda+1>\frac{4}{15}k$ or 
\[ k-(a+b)+\varepsilon k+\max(\lambda, \mu)\leq k- \varepsilon k.\]
\end{proposition}
\begin{proof}
Suppose $k-(a+b)+\varepsilon k+\max(\lambda, \mu)>k- \varepsilon k$,  then $\lambda+2\varepsilon k+\mu>a+b$. Thus, 
\begin{equation}\label{eq-lambda-a-b-delta}
\lambda+\delta k \geq \lambda+(2\varepsilon)k+\mu>a+b.
\end{equation}
 
Consider the graph $Y = X_{1,2}$ for which $V(Y) = V(X)$ and two vertices $x,y\in V(Y)$ are adjacent if and only if they are at distance 1 or 2 in $X$. Denote by $k_Y$ the uniform degree of  $Y$, by $\mu(Y)$ the number of common neighbors of two vertices at distance 2 in $Y$. Additionally, let $\lambda_i(Y)$ denote the number of common neighbors of any pair of vertices $x,y\in V(Y)$, which are at distance $i$ from each other in $X$ for $i = 1, 2$. Then, by Proposition~\ref{intersection-num-diam-3},

\[k_Y = k+k_2 = k+\frac{(k-\lambda-1)k}{\mu},\]
\[\lambda_1(Y) = p_{2,2}^1+p_{1,2}^1+p_{2,1}^1+p_{1,1}^1 = \frac{(k-\lambda-1)(k-a-\mu)}{\mu}+2(k-\lambda-1)+\lambda,\]
\[\lambda_2(Y) = \left(k-\lambda-1+\frac{ba-k+(k-a-\mu)(k-a-\mu-\lambda)}{\mu}\right)+2(k-a-\mu)+\mu,\]
\[\mu(Y) = p_{2,2}^3+p_{1,2}^3+p_{2,1}^3+p_{1,1}^3 = 
2b+\frac{(2k-\lambda-\mu-a-b)b}{\mu} = b \frac{2k-\lambda-a+(\mu-b)}{\mu}.\]

Observe, since $X$ is a distance-regular graph of diameter $3$, there exist vertices $u,v,w$ in $X$, such that $\dist(u,v) = 2$, $\dist(u,w) = 1$ and $\dist(v,w) = 3$. Applying the inclusion from Eq.~\eqref{eq1} to these $u,v,w$ we get that the parameters of $Y$ satisfy 
\begin{equation}\label{eq-tr-1}
\lambda_1(Y)+\lambda_2(Y)\leq k_Y+\mu(Y).
\end{equation}

Suppose that $\lambda+1< \alpha k$. Then, by Eq. \eqref{eq-lambda-a-b-delta}, both $a< (\alpha+\delta) k$ and $b< (\alpha+\delta) k$. Recall also, that by Lemma \ref{max-min} we have $\lambda+1\leq \frac{k+\mu}{2}+1\leq \frac{2k}{3}$, so in particular $k\leq \frac{1}{10}k(k-\lambda-1)$ for $k\geq \frac{1}{100}$. 
Thus, we have the following estimates:
\[\lambda_1(Y)\geq \frac{(k-\lambda-1)(1-\alpha-2\delta) k}{\mu},\]
\[\lambda_2(Y)\geq \frac{(k-\lambda-1-2\delta k)k(1-2\alpha+2\delta)-k}{\mu}\geq \frac{k(k-\lambda-1)[(1-2\alpha-2\delta)(1-6\delta)-\frac{1}{10}]}{\mu},\]
\[\mu(Y)\leq (\alpha+\delta) \frac{k(2k-\lambda)}{\mu}\leq (\alpha+\delta) \frac{4k(k-\lambda-1)}{\mu}, \quad k_Y\leq \frac{(k-\lambda-1)k}{\mu}(1+3\delta).\]
Therefore, by Eq. \eqref{eq-tr-1},
\[1-\alpha-2\delta+(1-2\alpha-2\delta)(1-6\delta)-\frac{1}{10}\leq 3\delta+1+4\alpha+4\delta.\]
Hence, $\alpha\geq \frac{1-17\delta+12\delta^2-1/10}{7-12\delta}> \frac{1}{10}$ for $\delta\leq \frac{1}{100}$. In other words, $\lambda+1\leq \frac{1}{10} k$ is impossible under the assumptions we made.

Now, we will do a more careful analysis. Suppose, $\lambda+1 = \alpha k$, where $\frac{2}{3}\geq \alpha>\frac{1}{10}$.  Note also, that by Observation \ref{aplusb} we have $a+b\geq \lambda+2>\alpha k$. Thus, we have the following estimates
\[\lambda_1(Y)\geq \frac{(1-\alpha)(k-a-\mu)k}{\mu},\]
\[ \lambda_2(Y) \geq \frac{(b-\mu)a-k+(k-a)(k-a-\mu-\lambda)}{\mu}+(k-1)\geq \frac{(k-a)(k-a-\mu-\lambda)-\mu}{\mu},\]
\[\lambda_2(Y)\geq \frac{(1-\alpha-\delta)(k-a-\alpha k-\delta k)k-\delta k}{\mu}\geq \frac{[(1-\alpha)(k-a-\alpha k-\delta k)-3\delta(1-\alpha) k]k}{\mu},\]
\[\mu(Y)\leq \frac{(2-2\alpha)bk}{\mu}, \quad k_Y\leq \frac{(1-\alpha+\delta)k^2}{\mu}\leq \frac{(1-\alpha)(1+3\delta)k^2}{\mu}.\]
Therefore, by Eq. \eqref{eq-tr-1},
\[k-a-\delta k+k-a-\alpha k -4\delta k\leq 2b+(1+3\delta)k,\]
\[(1-8\delta)\leq 2\frac{a+b}{k}+\alpha.\]
Suppose, that $a+b>\frac{6}{5}\alpha k$. Then, since $\alpha>\frac{1}{10}$, we have 
$$k-(a+b)+\lambda+(\mu+\varepsilon k)\leq k-\frac{1}{5}\alpha k+\delta k\leq \frac{99}{100}k\leq k-\varepsilon k.$$ 
Finally, if $a+b\leq \frac{6}{5}\alpha k$, we get $\alpha\geq (1-8\delta)\frac{5}{17}> \frac{4}{15}$.
\end{proof}

\subsection{Geometric distance-regular graphs}\label{sec-geom}

In this section we discuss geometric distance-regular graphs. More on geometric distance-regular graphs can be found in~\cite{Koolen-survey}.

Let $X$ be a distance-regular graph, and $\theta_{\min}$ be its smallest eigenvalue. Delsarte proved in \cite{Delsarte} that any clique $C$ in $X$ satisfies $|C|\leq 1-k/\theta_{\min}$. A clique in $X$ of size $1-k/\theta_{\min}$ is called a \textit{Delsarte clique}.
\begin{definition}\label{def-geom}
A distance-regular graph $X$ is called \textit{geometric} if there exists a collection $\mathcal{C}$ of Delsarte cliques such that every edge of $X$ lies in precisely one clique of $\mathcal{C}$.
\end{definition}

More generally, we say that any graph $X$ contains a \textit{clique geometry}, if there exists a collection $\mathcal{C}_0$ of (not necessarily Delsarte) maximal cliques, such that every edge is contained in exactly one clique of $\mathcal{C}_0$. We will sometimes call the cliques of $C_0$ by \textit{lines}. Metsch proved that any graph $X$ under rather modest assumptions contains a clique geometry.

\begin{theorem}[Metsch {\cite[Result 2.2]{Metsch}}]\label{Metsch}
Let $\mu\geq 1$, $\lambda^{(1)}, \lambda^{(2)}$ and $m$ be integers. Assume that $X$ is a connected graph with the following properties.
\begin{enumerate}
\item Every pair of adjacent vertices has at least $\lambda^{(1)}$ and at most $\lambda^{(2)}$ common neighbors;
\item Every pair of non-adjacent vertices has at most $\mu$ common neighbors;
\item $2\lambda^{(1)}-\lambda^{(2)}>(2m-1)(\mu-1)-1$;
\item Every vertex has fewer than $(m+1)(\lambda^{(1)}+1)-\frac{1}{2}m(m+1)(\mu-1)$ neighbors.
\end{enumerate}

Define a line to be a maximal clique $C$ satisfying $|C|\geq \lambda^{(1)}+2-(m-1)(\mu-1)$. Then every vertex is on at most $m$ lines, and every pair of adjacent vertices lies in a unique line. 

\end{theorem}

\begin{remark}\label{geom-smallest-eigenvalue} Suppose that $X$ satisfies the conditions of the previous theorem. Then the smallest eigenvalue of $X$ is at least $-m$.
\end{remark}
\begin{proof}
Let $\mathcal{C}$ be the collection of lines of $X$. Consider $|V|\times |\mathcal{C}|$ vertex-clique incidence matrix $N$. That is, $(N)_{v,C} = 1$ if and only if $v\in C$ for $v\in X$ and $C\in \mathcal{C}$. Since every edge belongs to exactly one line, we get $NN^{T} = A+D$, where $A$ is the adjacency matrix of $X$ and $D$ is a diagonal matrix. Moreover, $(D)_{v,v}$ equals to the number of lines that contain $v$. By the previous theorem, $D_{v,v} \leq m$ for every $v\in X$. Thus, 
\[A+mI = NN^{T}+(mI-D)\] 
is positive semidefinite, so all eigenvalues of $A$ are at least $-m$.
\end{proof}
\begin{remark}\label{m2-line}
Suppose $X$ is a non-complete regular graph that satisfies the conditions of Theorem~\ref{Metsch} with $m=2$. Then $X$ is a line graph $L(Y)$ for some graph $Y$.
\end{remark}
\begin{proof}
Let $\mathcal{C}$ be the collection of lines of $X$. Observe that since $X$ is regular and non-complete, every vertex is in at least two elements of $\mathcal{C}$. Indeed, if $v$ belongs to the unique $C\in \mathcal{C}$, then $N(v) = C$, as every edge of $X$ lies in a unique element of $\mathcal{C}$. Thus, regularity of $X$ implies that $N(v)$ is a connected component of $X$, which by assumptions is connected and non-complete. Therefore, by Theorem \ref{Metsch} every vertex of $X$ is in exactly two lines.

Define a graph $Y$ with the set of vertices $V(Y) = \mathcal{C}$. Two vertices $C_1, C_2 \in \mathcal{C}$ in $Y$ are adjacent if and only if $C_1\cap C_2 \neq \emptyset$. We claim that $L(Y) \cong X$. Indeed, since every edge of $X$ is in exactly one line, $|C_1\cap C_2| \leq 1$, so there is a well-defined map $f:E(Y)\rightarrow V(X)$. Moreover, since every vertex of $X$ is in exactly two lines, map $f$ is bijective. Additionally, two edges in $Y$ share the same vertex if and only if there is an edge between corresponding vertices in $X$. Hence, $L(Y) \cong X$. 
\end{proof}

Another useful tool is the following sufficient condition of geometricity.
\begin{proposition}[{\cite[Proposition 9.8]{Koolen-survey}}]\label{suff-cond}
Let $X$ be a distance-regular graph of diameter $d$. Assume there exist a positive integer $m$ and a set $\mathcal{C}$ of cliques of $X$ such that every edge in $X$ is contained in precisely one clique of $\mathcal{C}$ and every vertex $u$ is contained in exactly $m$ cliques of $\mathcal{C}$. If $|\mathcal{C}|<|V(X)| = n$, then $X$ is geometric with smallest eigenvalue $-m$. Moreover, $|\mathcal{C}|<n$ holds if $\min\{|C| : C\in \mathcal{C}\}>m$.  
\end{proposition} 

\begin{corollary}[{\cite[Proposition 9.9]{Koolen-survey}}]\label{geom-corol}
Let $m\geq 2$ be an integer, and let $X$ be a distance-regular graph  with $(m-1)(\lambda+1)<k\leq m(\lambda+1)$ and diameter $d\geq 2$. If $\lambda\geq \frac{1}{2}m(m+1)\mu$, then $X$ is geometric with smallest eigenvalue $-m$.
\end{corollary}
\begin{proof}
Directly follows from Theorem \ref{Metsch} and Proposition \ref{suff-cond}.
\end{proof}

Next lemma is an analog of Lemma 8.3 in \cite{Sun-Wilmes}.
\begin{lemma}\label{clique-mu-bound}
Let $X$ be a graph. Let $\mathcal{C}$ be a collection of cliques in $X$, such that every edge lies in a unique clique from $\mathcal{C}$ and every vertex is in at most $m$ cliques from $\mathcal{C}$. Then $\mu(X)\leq m^2$. If $X$ is distance-regular of diameter $d\geq 3$, then  $\mu = \mu(X)\leq (m-1)^2$.
\end{lemma}
\begin{proof}
Let $u,v\in X$ be a pair of vertices at distance $2$. By the assumptions of the lemma we can write $N(u) = \bigcup\limits_{i = 1}^{m_u} C_i^u$ and $N(v) = \bigcup\limits_{i = 1}^{m_v} C_i^v$, where $C_i^u, C_j^v\in \mathcal{C}$. Since $\dist(u,v) = 2$, all cliques are distinct. Observe, that any two different cliques in $\mathcal{C}$ intersect each other in at most one vertex. Hence, $N(u)\cap N(v)\leq m_u m_v$, so $\mu(X)\leq m^2$.

If $X$ is distance-regular of $d\geq 3$, there exist vertices $x,y\in X$ such that the distances in $X$ are $\dist(x,u) = 1$, $\dist(x,v) = 3$, $\dist(y,v) = 1$, $\dist(y,u) = 3$.  Let $C_1^u$ be a clique in $\mathcal{C}$ that contains vertices $x$ and $u$. Then since $\dist(x,v) = 3$ we get $C_1^u\cap N(v) = \emptyset$. Similarly, let $\{y,v\}\in C_1^v$, then $C_1^v\cap N(u) = \emptyset$. Again, any two cliques in $\mathcal{C}$ intersect each other in at most one vertex. Hence, $N(u)\cap N(v)\leq (m_u-1)(m_v-1)$, which implies $\mu\leq (m-1)^2$. 
\end{proof}

Our strategy is to show that graphs, for which we could not apply Lemma~\ref{mixing-lemma-tool} and Observation~\ref{obs1}, satisfy Metsch's theorem for small $m$. Thus, in particular, by Remark~\ref{geom-smallest-eigenvalue}, smallest eigenvalue of $X$ is at least $-m$. In this sense, the problem of classifying graphs with the linear motion is strongly related with the problem of classifying graphs with certain level of regularity and smallest eigenvalue at least $-m$. The first important result for the latter problem is due to Seidel, who characterized the strongly regular graphs with smallest eigenvalue $-2$. We state here a generalization of Seidel's result given by Brouwer, Cohen and Neumaier.

\begin{definition}
 Graph $X$ is called \textit{$m\times n$-grid} if it is the line graph of a complete bipartite graph $K_{m,n}$.
\end{definition}

\begin{theorem}[{\cite[Corollary 3.12.3]{BCN}}]\label{geom-eig-more2}
Let $X$ be a connected regular graph with smallest eigenvalue $>-2$. Then $X$ is a complete graph, or $X$ is an odd polygon.  
\end{theorem}

\begin{theorem}[Seidel \cite{Siedel}; Brouwer-Cohen-Neumaier {\cite[Theorem 3.12.4]{BCN}}]\label{geom-eig-2}
Let $X$ be a connected regular graph on $n$ vertices with smallest eigenvalue $-2$.
\begin{enumerate}[(i)]
\item If $X$ is strongly regular, then $X = L_{2}(s)$, or $X = T(s)$ for some $s$, or $n\leq 28$.
\item If $X$ is edge-regular, then $X$ is strongly regular or the line graph of a regular triangle-free graph.
\item If $X$ is co-edge-regular, then $X$ is strongly regular, an $m_1\times m_2$-grid, or one of the two regular subgraphs of the Clebsh graph with $8$ and $12$ vertices, respectively.
\end{enumerate}
\end{theorem}

For the case of distance-regular graphs of diameter $3$ we use two recent classification results for geometric distance-regular graphs with smallest eigenvalue $-3$ proven by Bang~\cite{Bang-diam-3} and Bang, Koolen~\cite{Bang-Koolen-diam-3}. Since Bang's classification result includes a lot of cases we state here the short corollary of her result that fits our purposes.

\begin{theorem}[Bang, Koolen \cite{Bang-Koolen-diam-3}]\label{bang-koolen} 
A geometric distance-regular graph with smallest eigenvalue $-3$, diameter $d\geq 3$ and $\mu\geq 2$ is one of the following.
\begin{enumerate}
\item The Hamming graph $H(3, s)$, where $s\geq 3$, or Johnson graph $J(s,3)$, where $s\geq 6$;
\item The collinearity graph of a generalized quadrangle of order $(s,3)$ deleting the edges in a spread, where $s\in \{3, 5\}$.
\end{enumerate}
\end{theorem}  

\begin{theorem}[Bang \cite{Bang-diam-3}, corollary to Theorem 4.3]\label{bang}
Let $X$ be a geometric distance-regular graph with diameter $d = 3$, smallest eigenvalue $-3$, $\mu = 1$ and $k>24$. Then $\lambda\geq \beta\geq 2$ and $b_2 = 2\lambda-2\beta+4$, $c_3 = 3\beta$.

\end{theorem}

\subsection{Proof of the diameter-3 case}\label{sec-diam-3-proof}

In this section we finally provide the classification of all distance-regular graphs of diameter 3 with sublinear motion. 

As before, we assume that $\varepsilon$ is some small positive number, say $0< \varepsilon \leq \frac{1}{100}$. We start with stating two corollaries to Theorem \ref{Metsch}.

\begin{corollary}\label{lem-3-lines}
Let $X$ be a distance-regular graph. Suppose that $\frac{4}{15}k<\lambda+1\leq(\frac{1}{2}-\varepsilon)k$ and $\mu<\varepsilon k$, then $X$ is a  geometric distance-regular graph and every vertex lies in exactly $3$ lines.
\end{corollary} 
\begin{proof}
Applying Theorem \ref{Metsch} with $m=3$ we get that every vertex lies on at most 3 lines. At the same time the assumed upper bound on $\lambda$ assures that no vertex can lie in just 2 cliques. The geometricity of $X$ directly follows from Proposition \ref{suff-cond} with $m = 3$.
\end{proof}

\begin{theorem}[Corollary to Theorem 4.2.16 \cite{BCN}]\label{Theor4-2-16}
 Let $Y$ be a connected graph, such that its line graph $X = L(Y)$ is distance-regular. Then $X$ is a polygon, a bipartite graph or has diameter $\leq 2$.
\end{theorem}

\begin{corollary}\label{lambda-half}
Let $X$ be a distance-regular graph. Suppose that $\lambda+1>(\frac{1}{2}-\varepsilon)k$ and $\mu<\varepsilon k$, then $X$ is a geometric distance-regular graph and every vertex lies in exactly $2$ lines. Moreover, $X$ is either a polygon, or a bipartite graph, or has diameter 2.
\end{corollary}
\begin{proof}
By Metsch's Theorem \ref{Metsch}, $X$ is a geometric distance-regular graph and every vertex is in exactly 2 lines. Thus, by Remark \ref{m2-line}, $X$ is a line graph. So statement follows from Theorem \ref{Theor4-2-16}.
\end{proof}

We will need the following inequality from \cite{BCN}.

\begin{proposition}[Prop. 4.3.2(i), Prop. 4.3.3 \cite{BCN}]\label{Prop433}
Let $X$ be a distance-regular graph and $\mu=1$, then
\[\frac{nk}{(\lambda+1)(\lambda+2)}\geq 1+\frac{\lambda+2}{\lambda+1}b_1+\frac{\lambda+2}{\lambda+1}b_1^2.\]
\end{proposition}

\begin{proposition}\label{geom-case}
Let $X$ be a geometric distance-regular graph of diameter $d=3$, such that every vertex lies in exactly $3$ lines. Then one of the following is true:
\begin{enumerate}
\item $X$ is a Hamming graph or a Johnson graph,
\item $b_2\geq \frac{11}{10} (\lambda+1)$,
\item $k\leq 100$.
\end{enumerate}
\end{proposition}
\begin{proof}
 The definition of a geometric graph and the conditions on $X$ imply that the smallest eigenvalue of $X$ is $-3$. Now we consider two cases. If $\mu>1$, by Theorem \ref{bang-koolen} we get that $X$ is a Johnson graph, a Hamming graph or $k\leq 20$. Otherwise, if $\mu = 1$, by Theorem \ref{bang}, we get that for $X$ either $k\leq 24$, or $b = c_3 = 3\beta\geq 6$.
 By Proposition \ref{Prop433}, 
 \[\frac{3n}{\lambda+2} = \frac{nk}{(\lambda+1)(\lambda+2)}\geq 1+\frac{\lambda+2}{\lambda+1}b_1+\frac{\lambda+2}{\lambda+1}b_1^2>b_1^2.\]
 Hence, as $k=3(\lambda+1)$ and $b_1 = 2(\lambda+1)$ in this case, the inequality 
  $3n>4(\lambda+1)^3$ holds.
  At the same time, $n = 6(\lambda+1)^2\frac{b_2}{c_3}+6(\lambda+1)^2+3(\lambda+1)+1$. Suppose, that $b_2\leq \frac{11}{10}(\lambda+1)$ and $k>100$. Then, $\frac{198}{10c_3}+\frac{3}{5}>4$. Therefore, $c_3<6$, contradiction.
\end{proof}

\begin{theorem}\label{main-thm}
Let $X$ be a distance-regular graph of diameter $3$ on $n$ vertices. Then one of the following is true.
\begin{enumerate}
\item We have $\motion(X)\geq \gamma n$ for some absolute constant $\gamma>0$.
\item $X$ is the Johnson graph $J(s,3)$ for $s\geq 7$, or the Hamming graph $H(3,s)$ for $s\geq 3$. 
\item $X$ is the cocktail-party graph.
\end{enumerate}
\end{theorem}
\begin{proof}
 Fix $\varepsilon = \frac{1}{300}$, so that $\delta = 2\varepsilon+\frac{\varepsilon^2}{2}<\frac{1}{100}$ and let $\gamma_0 = \frac{1}{12}\varepsilon^{4}$. Let $C = \max(100, (10\varepsilon^{-2})^3)$. Observe that for any group its minimal degree is at least 1. Hence, if $k\leq C$, then $n\leq k^3\leq C^3$ and so for $\gamma_1\leq C^{-3}$ the motion of $X$ is at least $\gamma_1 n$. Take $\gamma = \frac{2}{3}\min(\gamma_0, \gamma_1)$.
 
 If the parameters of $X$ satisfy $k>\frac{3}{2}\gamma
 (n-1)>\gamma n$, then the result follows from Proposition~\ref{prop-k-big}. If $k\leq\frac{3}{2}\gamma (n-1)$, then by Lemma \ref{min-est}, $\mu<\frac{\varepsilon^2}{2}k$. In this case, for $k\geq(10\varepsilon^{-2})^3$ we get bounds on the eigenvalues of $X$ in Theorem \ref{eigen-approxim-d-3}. By Observation~\ref{obs2}, we either get a linear lower bound $\gamma n$ on the motion of $X$, or Propositions~\ref{Prop-a-b-big} and \ref{Prop-a-b-small} show that $X$ is bipartite or $\lambda\geq \frac{4}{15}k$. Corollary \ref{lem-3-lines}, Corollary \ref{lambda-half}, Theorem \ref{main-bipartite} and Proposition \ref{geom-case} finish the proof.
\end{proof}

By Corollary \ref{Wielandt-cor}, Theorem \ref{main-thm} immediately implies

\begin{theorem}\label{main-thickness}
Let $X$ be a distance-regular graph of diameter $3$. Then one of the following is true.
\begin{enumerate}
\item $\theta(\Aut(X)) = O(\log(n))$.
\item $X$ is the Johnson graph $J(s,3)$ for $s\geq 7$, or the Hamming graph $H(3,s)$ for $s\geq 3$. Hence, $\theta(\Aut(X))$ has order $n^{1/3}$ in this case.
\item $X$ is the cocktail-party graph and $\theta(\Aut(X)) = \frac{n}{2}$ in this case.
\end{enumerate}
\end{theorem}

\subsection{A tradeoff for intersection numbers}\label{sec-tradeoff}

We start our discussion of the case of distance-regular graphs of arbitrary diameter from the following inequality, which is the main ingredient of the reduction to geometric graphs in this case.
 It essentially states that if $b_j$ is large (so are $b_i$ for $i\leq j$, by monotonicity) and $c_{j+1}$ is small, then both $b_{j+1}$ and $c_{j+2}$ could not be small simultaneously.

\begin{lemma}\label{b-c-ineq}
Let $X$ be a primitive distance-regular graph of diameter $d\geq 3$. Fix $1 \leq j\leq d-2$ and $\alpha>\varepsilon>0$. Suppose that $c_{j+1}\leq \varepsilon k$ and $b_j\geq \alpha k$. Denote $C = \frac{\alpha}{\varepsilon}$, then for any $1\leq s\leq j+1$
\begin{equation}
b_{j+1}\left(\sum\limits_{t=1}^{s}\frac{1}{b_{t-1}}+\sum\limits_{t=1}^{j+2-s}\frac{1}{b_{t-1}}\right)+c_{j+2}\sum\limits_{t=1}^{j+1}\frac{1}{b_{t-1}}\geq 1-\frac{4}{C-1}.
\end{equation}

\end{lemma}
\begin{proof}
Consider the graph $Y$ with the set of vertices $V(Y) = V(X)$, where two vertices $u,v$ are adjacent if they are at distance $\dist(u,v)\leq j+1$ in $X$. We want to find the restriction on the parameters of $X$ implied by inclusion (\ref{eq1}) applied to graph $Y$ and $w,v$ at distance $j+2$ in $X$. Let $\lambda_i^Y$ denote the number of common neighbors in $Y$ for two vertices $u,v$ at distance $i$ in $X$ for $i\leq j+1$. Let $\mu_{j+2}^{Y}$ denote the number of common neighbors in $Y$ for two vertices $u,v$ at distance $j+2$ in $X$. For $C = \frac{\alpha}{\varepsilon}$ we have $k_{i+1}\geq C k_i$ for $i\leq j$. Thus, the degree of every vertex in $Y$ satisfies
\[k^Y = \sum\limits_{i=1}^{j+1}k_{i}\leq k_{j+1}\sum\limits_{i = 0}^{j} C^{-i}\leq k_{j+1}\frac{C}{C-1}.\]
Note, that $\sum\limits_{t = 0}^{d}p^i_{s, t} = k_{s}$. Hence, we have
\[\mu_{j+2}^Y = \sum\limits_{1\leq s,t\leq j+1} p^{j+2}_{s,t}\leq 2\sum\limits_{i = 1}^{j} k_i +p^{j+2}_{j+1, j+1}\leq \frac{2}{C-1}k_{j+1}+p^{j+2}_{j+1, j+1},\] 
\[\lambda_i^Y = \sum\limits_{1\leq s, r\leq j+1} p^{i}_{r,s}\geq \sum\limits_{1\leq s\leq j+1} p^{i}_{j+1, s}= k_{j+1}-\sum\limits_{j+2\leq s\leq d} p^{i}_{j+1, s} - p_{j+1, 0}^{i}.\]
Now we are going to get some bounds on $\sum\limits_{j+2\leq s\leq d} p^{i}_{j+1, s}$. We use the following observation. Suppose, that $x, y$ are two vertices at distance $i$. Then there are exactly $\prod\limits_{t = 1}^{i} c_t$ paths of length $i$ between $x$ and $y$. Thus, $(\prod\limits_{t = 1}^{i} c_t)\sum\limits_{s = j+2}^{d} p^{j+1}_{i, s}$ equals to the number of paths of length $i$ starting at $v$ and ending at distance at least $j+2$ from $u$ and at distance $i$ from $v$, where $\dist(u,v) = j+1$ in $X$. We count such paths by considering possible choices of edges for a path at every step. At step $t$ every such path should go from $N_{t-1}(v)$ to $N_{t}(v)$, hence there are at most $b_{t-1}$ possible choices for a path at step $t$ for $1\leq t\leq i$. Moreover, since path should end up at distance at least $j+2$ from $u$, then for some $1\leq t\leq i$ path should go from $N_{j+1}(u)$ to $N_{j+2}(u)$. Therefore, the number of paths that go from $N_{j+1}(u)$ to $N_{j+2}(u)$ at step $t$ is at most $(\prod\limits_{s = 1}^{i}b_{s-1})\frac{b_{j+1}}{b_{t-1}}$. Hence,
\[ \sum\limits_{s = j+2}^{d} p^{i}_{j+1, s} = \frac{k_{j+1}}{k_i}\sum\limits_{s = j+2}^{d} p^{j+1}_{i, s}\leq \frac{k_{j+1}}{k_i} \sum\limits_{t = 1}^{i} \left(\prod\limits_{s = 1}^{i}b_{s-1}\right)\frac{b_{j+1}}{b_{t-1}} \left(\prod\limits_{t = 1}^{i} c_t\right)^{-1} = k_{j+1}\sum\limits_{t = 1}^{i}\frac{b_{j+1}}{b_{t-1}}. \]
Thus, in particular,
\[\lambda_i^Y\geq k_{j+1}\left(1-\sum\limits_{t=1}^{i}\frac{b_{j+1}}{b_{t-1}}\right) - p_{j+1, 0}^{i}.\]
 Similarly, 
 \[ p_{j+1, j+1}^{j+2} \leq k_{j+1}\sum\limits_{t=1}^{j+1} \frac{c_{j+2}}{b_{t-1}}. \]
Hence, 
\[\mu_{j+2}^Y \leq k_{j+1} \left(\frac{2}{C-1}+\sum\limits_{t=1}^{j+1}\frac{c_{j+2}}{b_{t-1}}\right).\]
Therefore, by applying inclusion from Eq. \eqref{eq1} to vertices $u,v,w$ in  $Y$, that satisfy $\dist(v,w) = j+2$, $\dist(u,v) = s$ and $\dist(w,u) = j+2-s$  in $X$, we get the desired inequality.
\end{proof}

\subsection{Minimal distinguishing number of primitive distance-regular graphs}\label{sec-dist-number-gen-drg}
Recall that a distance-regular graph $X$ is \textit{primitive} if the distance-$i$ graph $X_i$ is connected for every $1\leq i\leq d$. A distance-regular graph $X$ is \textit{antipodal} if being at distance $d$ in $X$ is an equivalence relation. It is known that every distance-regular graph is primitive, antipodal or bipartite.

As mentioned in Section \ref{sec-tools} via the corresponding distance configuration (see Def.~\ref{def-dist-conf}) a distance-regular graph can be seen as a special case of an association scheme. An edge $(u, v)$ receives color $i$ if distance between $u$ and $v$ in $X$ is precisely $i$. In any configuration define $D(u,v)$ to be the number of vertices that distinguish $u$ and $v$ (see Definition \ref{def-disting}). Define the \textit{minimal distinguishing number} $D_{\min}(\mathfrak{X})$ of the configuration $\mathfrak{X}$ to be 
\[D_{\min}(\mathfrak{X}) = \min\limits_{u\neq v\in V} D(u, v).\]
It is easy to see that for a homogeneous coherent configuration $\mathfrak{X}$ number $D(u,v)$ depends only on the color between $u$ and $v$, so one can correctly define $D(i) = D(u,v)$. We will need the following observation by Babai from \cite{Babai-annals}.

\begin{lemma}[Babai {\cite[Proposition 6.4]{Babai-annals}}]\label{babai-dist}
Let $\mathfrak{X}$ be a homogeneous coherent configuration of rank $r$. Then for any colors $1\leq i, j\leq r-1$ the inequality $D(j)\leq \dist_{i}(j) D(i)$ holds.
\end{lemma}
\begin{proof}
Since the constituent $X_i$ is connected, statement follows from the triangle inequality  $D(u,v)\leq D(v,w)+D(w, u)$ for any  vertices $u,v,w$.
\end{proof}

\begin{proposition}\label{primitive-distinguish} Let $X$ be a primitive distance-regular graph of diameter $d\geq 2$. Fix some positive real numbers $\alpha, \beta >0$. Suppose that for some $1\leq j \leq d-1$ inequalities $b_j\geq \alpha k$ and $c_{j+1}\geq \beta k$ hold. Then any two distinct vertices in $X$ are distinguished by at least $\frac{2\min(\alpha, \beta)}{d^2}(n-1)$ vertices.
\end{proposition}
\begin{proof}
Clearly, $n-1 = \sum\limits_{i = 1}^{d} k_i$. Thus, for some $1\leq t\leq d$ we have $k_t\geq \frac{n-1}{d}$. Since the sequence $(b_i)$ is non-increasing, if $t\leq j$, then $a_t = k-b_t-c_t\leq (1-\alpha)k$. Similarly, the sequence $(c_i)$ is non-decreasing, so if $t>j$, then $a_t = k-b_t-c_t\leq (1-\beta)k$.

Consider any pair of adjacent vertices $u,v\in X$. Then they are distinguished by at least $|N_t(u)\bigtriangleup N_t(v)| = 2(k_t-p^1_{t,t})$ vertices. Note, that $$p^1_{t,t} = p^t_{t,1}\frac{k_t}{k} = k_t\frac{a_t}{k}\leq (1-\min(\alpha, \beta))k_t.$$
Hence, any two adjacent vertices are distinguished by at least $2\min(\alpha,\beta)k_t\geq \frac{2\min(\alpha,\beta)(n-1)}{d}$ vertices. Finally, the result follows from the Lemma \ref{babai-dist}.
\end{proof}

\subsection{Bound on the spectral gap of a distance-regular graph}\label{sec-drg-approx}

In this section we give a simple estimate on the spectral gap of a distance-regular graph in terms of its intersection numbers. We will use it to apply Lemma~\ref{mixing-lemma-tool} (the spectral tool).

Note, that if $b_i$ and $c_i$ are simultaneously small, then by monotonicity, $b_j$ for $j\geq i$ and $c_t$ for $t\leq i$ are small. Hence, the intersection matrix $T(X)$ is a small perturbation of a block diagonal matrix $N$, where one block is upper triangular and the second block is lower triangular. So the eigenvalues of $N$ are just the diagonal entries.

\begin{lemma}\label{eigenvalues-approximation}
Let $X$ be a distance-regular graph of diameter $d$. Suppose that $b_i\leq \varepsilon k$ and $c_i\leq \varepsilon k$ for some $i$ and some $\varepsilon>0$. Suppose furthermore that $b_{i-1}\geq \alpha k$ and $c_{i+1}\geq \beta k$, for some $\alpha, \beta >0$. Consider the matrix
\[T = T(X) = \left(\begin{matrix} 
a_0 & b_0 & 0 & 0 & ... \\ 
c_1 & a_1 & b_1 & 0 &...\\
0& c_2 & a_2 & b_2 & ...\\
...& & \vdots & & ...\\
...& & 0& c_d & a_d  

\end{matrix}\right)\]
Then the zero-weight spectral radius $\xi$ of $X$ satisfies 
\begin{equation}\label{eq-spectral-gap-dist}
\xi\leq k (1-\min(\alpha, \beta)+2(d+2)^2\varepsilon^{\frac{1}{d+1}}).
\end{equation}

\end{lemma}
\begin{proof}
Let $N$ be a matrix obtained from $T$ by replacing all $b_s$ and $c_t$ with 0 for $s\geq i$ and $t\leq i$. As in Theorem \ref{matrix-approximation}, consider
\[ M = \max\{|(T)_{sj}|, |(N)_{sj}|: 1\leq s, j\leq d\} = k, \]
\[ \delta = \frac{1}{(d+1)M}\sum\limits_{s = 1}^{d+1}\sum\limits_{j=1}^{d+1}|(T)_{sj}-(N)_{sj}| \leq \frac{(d+1)\varepsilon k}{(d+1)M} = \varepsilon\]

 Observe, that the diagonal entry $a_i$ is the only non-zero entry in the $i$-th row of $N$. Furthermore, the upper-left $i\times i$ submatrix is upper triangular and $(d-i)\times (d-i)$ lower-right submatrix is lower triangular. Then the eigenvalues of $N$ are equal to $a_j$ for $0\leq j\leq d$.

Conditions on $X$ imply that $a_i\geq (1-2\varepsilon) k$, while $a_j\leq k(1-\min(\alpha, \beta))$ for $j\neq i$. Hence, since $k$ is an eigenvalue of multiplicity 1 of $X$, by Theorem \ref{matrix-approximation}, the zero-weight spectral radius of $X$ satisfies Eq.~\eqref{eq-spectral-gap-dist}.
\end{proof}

\subsection{Reduction to geometric case for primitive distance-regular graphs with fixed diameter}\label{sec-general}

Let $X$ be a primitive distance-regular graph. In this section we finally show that either the motion of $X$ is linear in the number of vertices, or $X$ is geometric (Theorem \ref{main-general-case}).

\begin{definition} Let $0\leq\delta<1$. We say that $(\alpha_i)_{i=0}^{\infty}$ is the $FE(\delta)$ sequence, if $\alpha_0 = 1$ and for $j\geq 1$ the element $\alpha_j$ is defined by the recurrence
\[\alpha_{j+1} = (1-\delta)\left( \sum\limits_{t=1}^{\lceil \frac{j+2}{2} \rceil}\frac{1}{\alpha_{t-1}}+\sum\limits_{t=1}^{\lfloor \frac{j+2}{2}\rfloor}\frac{1}{\alpha_{t-1}} \right)^{-1}.\] 

Let $\widehat{\alpha} = (\alpha_i)_{i=0}^{s}$ be a sequence. We say that $\widehat{\beta} = (\beta_i)_{i=2}^{s+2}$ is the $BE(\delta, \widehat{\alpha})$ sequence, if for $j\geq 2$ the element $\beta_j$ is defined as
\[\beta_j = (1-\delta)\left(\sum\limits_{t = 0}^{j-2}\frac{1}{\alpha_t}\right)^{-1}.\]
If additionally, $\widehat{\alpha}$ is the $FE(\delta)$ sequence, then we will say that $\widehat{\beta}$ is just the $BE(\delta)$ sequence.
\end{definition}
\begin{remark}
FE stands for ``forward expansion'' and BE stands for ``backward expansion''.
\end{remark}

\begin{definition}
Let $\widehat{\alpha} = (\alpha)_{i=0}^{s}$ be a decreasing sequence of positive real numbers with $\alpha_0 = 1$. Let $\widehat{\beta} = (\beta_i)_{i=2}^{s
+2}$ be the corresponding $BE(\delta, \widehat{\alpha})$ sequence. Let $0< \delta< 1$, we say that $\varepsilon>0$ is \textit{$(\delta, j, \widehat{\alpha})$-compatible} for $j\leq s$, if $\varepsilon$ satisfies
\[  \left(\frac{\alpha_j-5\varepsilon}{\alpha_j-\varepsilon} - 2\varepsilon \sum\limits_{t=1}^{j+1}\frac{1}{\alpha_{t-1}}\right)>(1-\delta)\quad  \text{and} \quad
2(d+2)^2\varepsilon^{\frac{1}{d+1}} \leq \min(\alpha_j, \beta_{j+2})\delta. \] 
\end{definition}

Note that if $\varepsilon$ is $(\delta, j, \widehat{\alpha})$-compatible for $j\geq 1$, then it is $(\delta, (j-1) , \widehat{\alpha})$-compatible as well. Note also that the second condition on $\varepsilon$ implies that $\delta>\varepsilon$ and $\beta_{j+2}>\varepsilon$, $\alpha_j>\varepsilon$. 

The proposition below states that for a distance-regular graph $X$ either we are in a favorable situation, or all parameters $b_i$ of $X$ are relatively large. 

\begin{proposition}\label{bound-on-b}
Let $X$ be a primitive distance-regular graph of diameter $d\geq 3$. Fix any $0<\delta<1$. Let $1 \leq j\leq d-2$ and $\widehat{\alpha} = (\alpha_i)_{i=0}^{j}$ be a decreasing sequence of positive real numbers. Consider corresponding $BE(\delta, \widehat{\alpha})$ sequence $\widehat{\beta}$ and $(\delta, j, \widehat{\alpha})$-compatible $\varepsilon>0$.  Assume that parameters of $X$ satisfy $c_{j+1}\leq \varepsilon k$ and $b_i\geq \alpha_i k$ for all $1\leq i\leq j$.
 Then one of the following is true.
\begin{enumerate}
\item $\motion(X)\geq \frac{2\varepsilon}{d^2}(n-1)$.
\item The zero-weight spectral radius $\xi$ for $X$ satisfies $\xi\leq k(1-(1-\delta)\beta_{j+2})$.
\item Let $\alpha_{j+1} = (1-\delta)\left( \sum\limits_{t=1}^{\lceil \frac{j+2}{2} \rceil}\frac{1}{\alpha_{t-1}}+\sum\limits_{t=1}^{\lfloor \frac{j+2}{2}\rfloor}\frac{1}{\alpha_{t-1}} \right)^{-1}$, then $b_{j+1}\geq\alpha_{j+1} k$ and $c_{j+2}\leq \varepsilon k$.
\end{enumerate}
\end{proposition}
\begin{proof}
 Assume that $c_{j+2}\geq \beta_{j+2} k$. If $b_{j+1}\geq \varepsilon k$, then by Proposition \ref{primitive-distinguish} any two distinct vertices in $X$ are distinguished by at least $\frac{2\varepsilon}{d^2}(n-1)$ vertices, which implies that the motion of $X$ is at least $\frac{2\varepsilon}{d^2}(n-1)$.

Thus, suppose that $b_{j+1}\leq \varepsilon k$. Then we fall into the assumptions of Lemma \ref{eigenvalues-approximation} with $i = j+1$. Hence, zero-weight spectral radius $\xi$ for $X$ satisfies $$\xi\leq k (1-\min(\alpha_j, \beta_{j+2})+2(d+2)^2\varepsilon^{\frac{1}{d+1}})\leq k(1-(1-\delta)\beta_{j+2}).$$
Note, that by definition $\beta_{j+2}<\alpha_j$, so $\min(\alpha_{j}, \beta_{j+2}) = \beta_{j+2}$.

Assume now that $\varepsilon k\leq c_{j+2}\leq \beta_{j+2} k$. the conditions on $\varepsilon$ imply $$\beta_{j+2} \leq \left(\frac{\alpha_j-5\varepsilon}{\alpha_j-\varepsilon}\right)\left(\sum\limits_{t=1}^{j+1}\frac{1}{\alpha_{t-1}}\right)^{-1}-2\varepsilon.$$  Then, by Lemma \ref{b-c-ineq}, we get $b_{j+1}\geq \varepsilon k$. Hence, by Proposition \ref{primitive-distinguish} any two distinct vertices in $X$ are distinguished by at least $\frac{2\varepsilon}{d^2}(n-1)$ vertices, which implies that the motion of $X$ is at least $\frac{2\varepsilon}{d^2}(n-1)$.

Finally, assume that $c_{j+2}\leq \varepsilon k$. Then, by Lemma \ref{b-c-ineq}, for 
\[0<\alpha_{j+1} \leq \left(\frac{\alpha_j-5\varepsilon}{\alpha_j-\varepsilon} - \varepsilon \sum\limits_{t=1}^{j+1}\frac{1}{\alpha_{t-1}}\right)\left( \sum\limits_{t=1}^{\lceil \frac{j+2}{2} \rceil}\frac{1}{\alpha_{t-1}}+\sum\limits_{t=1}^{\lfloor \frac{j+2}{2}\rfloor}\frac{1}{\alpha_{t-1}} \right)^{-1}
\]
we have $b_{j+1}\geq \alpha_{j+1}k$.
\end{proof}

\begin{theorem}\label{main-general-case}
For any $d\geq 3$ there exist constants $\gamma_d>0$ and $m_d\in \mathbb{N}$, such that for any primitive distance-regular graph $X$ with $n$ vertices and diameter $d$ one of the following is true. 
\begin{enumerate}
\item We have $\motion(X)\geq \gamma_d n$.
\item $X$ is a geometric distance-regular graph with smallest eigenvalue $-m$, where $m\leq m_d$. 
\end{enumerate}
Furthermore, one can set $m_d = \lceil 2(d-1)(d-2)^{\log_2(d-2)} \rceil$.
\end{theorem}
\begin{proof}
Fix some small $1>\delta>0$. Let $\widehat{\alpha} = (\alpha)_{i=0}^{\infty}$ be the $FE(\delta)$ sequence and $\widehat{\beta} = (\beta)_{i=2}^{\infty}$ be the $BE(\delta)$ sequence.  Let $\varepsilon>0$ be $(\delta, d-2, \widehat{\alpha})$-compatible. Additionally, assume that 
\begin{equation}\label{eq-eps-dist-reg-bound}
\varepsilon \leq \left(\frac{1}{(1-2\delta)\min(\alpha_{d-1}, \beta_d)}+1\right)^{-3}.
\end{equation}
Then, by Proposition \ref{bound-on-b} applied to $j = 0,..., d-2$ one by one, one of the following is true.
\begin{enumerate}
 \item The motion of $X$ is at least $\frac{2\varepsilon}{d^2}(n-1)$;
 \item the zero-weight spectral radius satisfies $\xi\leq (1-(1-\delta)\beta_i) k\leq (1-(1-\delta)\beta_d)k$ for some $1\leq i\leq d$;
 \item $b_i\geq \alpha_i k$ and $c_{i+1}\leq \varepsilon k$ for all $0\leq i\leq d-1$.
 \end{enumerate}  
 
 In the third case, we apply Lemma \ref{eigenvalues-approximation} for $i = d$. It implies that the zero-weight spectral radius $\xi$ of $X$ satisfies $\xi\leq k(1-(1-\delta)\alpha_{d-1})$. 
 
 Suppose, that $\lambda<(1-2\delta)\min(\alpha_{d-1}, \beta_{d})k=: rk$, then by Lemma \ref{mixing-lemma-tool}, we obtain $\motion(X)\geq \delta \min(\alpha_{d-1}, \beta_{d})n$. Hence, for $\gamma_d = \frac{1}{2}\min(\frac{2\varepsilon}{d^2}, \delta\min(\alpha_{d-1}, \beta_{d}))$ the motion of $X$ is at least $\gamma_d n$.
 
 Otherwise, since Eq.~\eqref{eq-eps-dist-reg-bound} holds, the conditions $\lambda\geq rk$, and $\mu = c_2\leq \varepsilon k$, by Metsch's criteria (see Theorem \ref{Metsch} and Corollary \ref{geom-corol}), imply that $X$ is geometric distance-regular graph, such that any vertex is contained in exactly $m$ lines, where $m$ is taken to satisfy 
 $$\frac{k}{m}\leq(\lambda+1)< \frac{k}{m-1}.$$
Clearly, condition $\lambda\geq rk$ implies 
$m<\frac{1}{r}+1$. Hence, $m_d$ can be defined as $$m_d = \lceil \frac{1}{(1-2\delta)\min(\alpha_{d-1} , \beta_{d})} \rceil.$$

Finally, consider the $FE(\delta)$ sequence $(\alpha_i)_{i=0}^{\infty}$ and the corresponding $BE(\delta)$ sequence $(\beta_i)_{i=2}^{\infty}$ for $0<\delta\leq\frac{1}{9}$.
We show by induction that $\alpha_j\geq \frac{(1-\delta)^2}{2} j^{-\log_2(j)}$  for all $j\geq 1$. Indeed, for $j=1,2$ we have $\alpha_1 = \frac{1-\delta}{2}$ and $\alpha_2 \geq \frac{(1-\delta)^2}{4}$, so inequality is true. For $j\geq 2$, we have
\[\alpha_{j+1} = (1-\delta)\left( \sum\limits_{t=1}^{\lceil \frac{j+2}{2} \rceil}\frac{1}{\alpha_{t-1}}+\sum\limits_{t=1}^{\lfloor \frac{j+2}{2}\rfloor}\frac{1}{\alpha_{t-1}} \right)^{-1}\geq \frac{(1-\delta)}{j+2}\alpha_{\lceil \frac{j}{2} \rceil}\geq \frac{(1-\delta)^3}{2(j+2)}  \left(\frac{j+1}{2}\right)^{-\log_2(\frac{j+1}{2})} = \]
\[ = \frac{(1-\delta)(j+1)2^{\log_2(\frac{j+1}{2})}}{j+2} \frac{(1-\delta)^2}{2} (j+1)^{-\log_2(j+1)} \geq \frac{(1-\delta)^3}{2} (j+1)^{-\log_2(j+1)}.
 \]
 Thus, 
 \[\beta_{j+2} = (1-\delta)\left(\sum\limits_{t = 0}^{j}\frac{1}{\alpha_t}\right)^{-1} \geq \frac{1-\delta}{j+1}\alpha_j\geq \frac{(1-\delta)^3}{2(j+1)} j^{-\log_2(j)}\] 
Hence, taking $\delta$ small enough we get $$m_d< \lceil\max\left(2(d-1)(d-2)^{\log_2(d-2)}, 2 (d-1)^{\log_2(d-1)}\right)\rceil+1.$$

\end{proof}
\vspace{30pt}

\section{Primitive coherent configurations of rank 4}\label{sec-coherent}

In this section we prove that a primitive coherent configuration of rank 4 on $n$ vertices is a Johnson or a Hamming scheme, or its motion is at least $\gamma n$ for some absolute constant $\gamma>0$ (Theorem \ref{main-coherent}). 

Primitive coherent configurations of rank 4 naturally split into three classes: configurations induced by a primitive distance-regular graph of diameter 3,  association schemes of diameter $2$ (see Definition \ref{def-assoc-diam}), and primitive coherent configurations with one undirected color and two oriented colors. The case of distance-regular graphs of diameter 3 was settled in Sections \ref{sec-large-deg} - \ref{sec-diam-3-proof}. So we need to deal with the other two classes. 

It is not hard to see that in the case with two oriented colors, the undirected constituent is strongly regular. Thus, by Babai's result (Theorem \ref{babai-str-reg-thm}), if $n\geq 29$, the only possibility for $\mathfrak{X}$ to have motion less than $\frac{n}{8}$ is when the undirected constituent is $T(s)$ and $L_{2}(s)$ or their complement. In the latter case we prove that the motion is linear using one of the arguments of Sun and Wilmes, appearing in the proof of {\cite[Lemma 3.5]{Sun-Wilmes}} (see Lemma~\ref{sun-wilmes-tool} in this paper).

Hence, we need to concentrate on the case of primitive association schemes of rank 4 and diameter 2. The general strategy is similar to the one used in the case of distance-regular graphs. As a first step we show that either we have a constituent with very large degree, or every pair of vertices can be distinguished by $\varepsilon n$ vertices (see Lemma~\ref{assoc-k2-large}). This directly implies that the motion is also at least $\varepsilon n$. On the other hand, the fact that one of the constituents, say $X_3$, has very large degree implies that some of the intersection numbers are quite small (see Proposition \ref{assoc-param-ineq}). This allows us to approximate the eigenvalues of the constituents $X_1$ and $X_2$, and so to approximate their zero-weight spectral radii with simple expressions involving the intersection numbers (see Lemma~\ref{assoc-spectral-radius}). Therefore, we try to apply Lemma \ref{mixing-lemma-tool} to the constituents $X_1$ and $X_2$. Considering the cases how degrees $k_1$ and $k_2$ can differ from each other, we obtain that either the motion of $\mathfrak{X}$ is linear, or one of the graphs $X_J$ is a line graph, where $J\in\{1, 2, \{1,2\}\}$. Recall, by definition, $X_{1,2}$ is a complement of $X_3$. Since a constituent graph is always edge-regular, we use the classification of edge-regular and co-edge-regular graphs with smallest eigenvalue $-2$ proved by Seidel and Brouwer, Cohen, Neumaier (see Theorem \ref{geom-eig-2}). The classification tells us that $X_i$ is strongly regular with smallest eigenvalue $-2$ or it is the line graph of a triangle-free regular graph (see Proposition~\ref{assoc-diam2} for the more precise statement). Moreover, if $X_J$ is strongly regular with smallest eigenvalue $-2$, then $X_J$ is a triangular graph $T(s)$, or a lattice graph $L_2(s)$,  or has at most 28 vertices.

If one of the constituents is a line graph, this allows us to obtain more precise bounds on the intersection numbers. In particular, we  approximate the zero-weight spectral radius of the graph $X_{1,2}$ with a relatively simple expression as well. From this point, our main goal becomes to get constraints on the intersection numbers, that will allow us to apply Lemma \ref{mixing-lemma-tool} effectively to one of the graphs $X_1$, $X_2$ or $X_{1,2}$. We consider four cases. Three of them are defined by which of the graphs $X_1$, $X_2$ or $X_{1,2}$ is strongly regular. In the fourth case, one of the constituents is the line graph of a triangle-free regular graph.  For the ranges of the parameters when Lemma \ref{mixing-lemma-tool} cannot be used effectively, we use a generalization of the argument due to Sun and Wilmes, which is proven in Lemma \ref{sun-wilmes-tool}. Roughly speaking, it says that if a triangular graph $T(s)$ is a union of several constituents of a coherent configuration $\mathfrak{X}$, then $\Aut(\mathfrak{X})$ is small if the following holds for every Delsarte clique: if we look on the configuration induced on the Delsarte clique, then each pair of vertices is distinguished by a constant fraction of the vertices of the clique. The hardest case in the analysis is the case when the constituent of the smallest degree, $X_1$, is strongly regular. This case is settled in Proposition \ref{assoc-x1-strongly-reg} and it requires preparatory work with several new ideas. In particular, we use an analog of the argument from the proof of Metsch's criteria to get a constant upper bound on the fraction $k_2/k_1$ in certain range of parameters. 

In the final subsection (Sec.~\ref{sec-coherent-thm-subsec}) we combine all the cases to prove our main result for the coherent configurations, Theorem \ref{main-coherent}.

The following definition will be useful.

\begin{definition}
We say that a homogeneous coherent configuration $\mathfrak{X}$ of rank $r$ \textit{has constituents ordered by degree}, if color $0$ corresponds to the diagonal constituent and the degrees of non-diagonal constituents satisfy $k_1\leq k_2\leq ...\leq k_{r-1}$. 
\end{definition}

\subsection{Approximation of eigenvalues}\label{sec-coh-approx}

In this subsection we provide technical lemmas that allow us to approximate the zero-weight spectral radius of $X_1$, $X_2$ and $X_{1,2}$ under quite modest assumptions.

\begin{lemma}\label{rank-4-roots}
Let $\mathfrak{X}$ be an association scheme of rank 4. Let $\eta$ be a non-trivial eigenvalue of $A_1$. Then $\eta$ satisfies the following cubic polynomial
\[\eta^{3} - (p_{1,1}^{1}+p_{1,2}^{2}-p_{1,1}^{3}-p_{1,2}^{3})\eta^2+\]
\[+\left( (p_{1,2}^{2}-p_{1,2}^{3})(p^{1}_{1,1} - p_{1,1}^{3}) - (p^{2}_{1,1} - p_{1,1}^{3})(p_{1,2}^{1}-p_{1,2}^{3})- (k_1-p^{3}_{1,1})\right)\eta+\]
\[+\left( (p_{1,2}^{2}-p_{1,2}^{3})(k_1-p_{1,1}^{3})+(p^{2}_{1,1} - p_{1,1}^{3})p_{1,2}^{3}\right) = 0. \]
\end{lemma}
\begin{proof}
By Eq. \eqref{eq-int-num} for intersection numbers we have
\[A_1^{2} = p^{1}_{1,1}A_1+p_{1,1}^{2}A_2+p_{1,1}^{3}A_3+k_1I.\]
We can eliminate $A_3$ using Eq. \eqref{eq-constit-sum}.
\begin{equation}\label{eq-eig-comp-1}
A_1^{2}= (p^{1}_{1,1} - p_{1,1}^{3})A_1+(p_{1,1}^{2}-p_{1,1}^{3})A_2+(k_1-p_{1,1}^{3})I+p_{1,1}^{3}J.
\end{equation}
Multiply previous equation by $A_1$ and use Eq. \eqref{eq-int-num}.
\begin{equation}\label{eq-eig-comp-2}
\begin{multlined}
A_1^{3} = (p^{1}_{1,1} - p_{1,1}^{3})A_1^{2}+(k_1-p^{3}_{1,1})A_1+p_{1,1}^{3}k_1J+\\
+(p^{2}_{1,1} - p_{1,1}^{3})((p_{1,2}^{1}-p_{1,2}^{3})A_1+(p_{1,2}^{2}-p_{1,2}^{3})A_2+p_{1,2}^{3}J-p_{1,2}^{3}I).
\end{multlined}
\end{equation}
Combining Eq. \eqref{eq-eig-comp-1} and \eqref{eq-eig-comp-2} we eliminate $A_2$ as well.

\[A_1^{3} - (p_{1,2}^{2}-p_{1,2}^{3})A_1^2 = (p^{1}_{1,1} - p_{1,1}^{3})A_1^{2}+(k_1-p^{3}_{1,1})A_1+p_{1,1}^{3}k_1J -  \]
\[ - (p_{1,2}^{2}-p_{1,2}^{3})(p^{1}_{1,1} - p_{1,1}^{3})A_1 - (p_{1,2}^{2}-p_{1,2}^{3})((k_1-p_{1,1}^{3})I+p_{1,1}^{3}J)+\]
\[+(p^{2}_{1,1} - p_{1,1}^{3})(p_{1,2}^{1}-p_{1,2}^{3})A_1+(p^{2}_{1,1} - p_{1,1}^{3})(p_{1,2}^{3}J-p_{1,2}^{3}I).\]
Suppose that $v$ is an eigenvector of $A_1$, which is different from the all-ones vector, and $\eta$ be the corresponding eigenvalue. Then $Jv = 0$ and $A_1v = \eta v$, so non-trivial eigenvalue $\eta$ is a root of the polynomial
\[ \eta^{3} - (p_{1,1}^{1}+p_{1,2}^{2}-p_{1,1}^{3}-p_{1,2}^{3})\eta^2+\]
\[+\left( (p_{1,2}^{2}-p_{1,2}^{3})(p^{1}_{1,1} - p_{1,1}^{3}) - (p^{2}_{1,1} - p_{1,1}^{3})(p_{1,2}^{1}-p_{1,2}^{3})- (k_1-p^{3}_{1,1})\right)\eta+\]
\[+\left( (p_{1,2}^{2}-p_{1,2}^{3})(k_1-p_{1,1}^{3})+(p^{2}_{1,1} - p_{1,1}^{3})p_{1,2}^{3}\right). \]

\end{proof}

\begin{proposition}\label{assoc-spectral-radius}
Fix $\varepsilon>0$. Let $\mathfrak{X}$ be an association scheme of rank 4. Suppose that the parameters of $\mathfrak{X}$ satisfy $\frac{1}{\varepsilon}\leq k_1$ and $p^{3}_{1,i}\leq \varepsilon k_1$ for $i = 1,2$. Then the zero-weight spectral radius $\xi(X_1)$ of $X_1$ satisfies
\[\xi(X_1)\leq \frac{p_{1,1}^{1}+p_{1,2}^{2}+\sqrt{(p_{1,1}^{1} - p_{1,2}^{2})^2+4p_{1,1}^{2}p_{1,2}^{1}}}{2}+25\varepsilon^{1/3} k_1. \]
\end{proposition}
\begin{proof}
By Lemma \ref{rank-4-roots}, any non-trivial eigenvalue is a root of the polynomial
\[\eta^3+a_1\eta^2+a_2\eta+a_3 := \eta^{3} - (p_{1,1}^{1}+p_{1,2}^{2}-p_{1,1}^{3}-p_{1,2}^{3})\eta^2+\]
\[+\left( (p_{1,2}^{2}-p_{1,2}^{3})(p^{1}_{1,1} - p_{1,1}^{3}) - (p^{2}_{1,1} - p_{1,1}^{3})(p_{1,2}^{1}-p_{1,2}^{3})- (k_1-p^{3}_{1,1})\right)\eta+\]
\[+\left( (p_{1,2}^{2}-p_{1,2}^{3})(k_1-p_{1,1}^{3})+(p^{2}_{1,1} - p_{1,1}^{3})p_{1,2}^{3}\right). \]
Observe, that for
\[ b_1  = -(p_{1,1}^{1}+p_{1,2}^{2}), \quad b_2 = p_{1,2}^2 p_{1,1}^{1} - p_{1,1}^{2} p_{1,2}^{1}, \quad b_3 = 0,\]
the following inequalities are true
\[ |a_1 - b_1|\leq 2\varepsilon k_1, \quad
|a_2 - b_2| \leq \left(4\varepsilon+2\varepsilon^2+\frac{1}{k_1}\right)k_1^2, \quad   |a_3 - b_3|\leq 2k_1^2\leq 2\varepsilon k_1^{3}.\]

Denote by $\nu_1, \nu_2, \nu_3$ non-trivial eigenvalues of $A_1$. Therefore, by Theorem \ref{poly-approx}, we can enumerate the roots $x_1, x_2, x_3$ of  $x^3+b_1x^2+b_2x+b_3 = 0$ so that $|\nu_i - x_i|\leq \delta$, where
\[\delta = 6 \left(2\varepsilon k_1(4k_1)^2+6\varepsilon k_1^2 (4k_1)+2\varepsilon k_1^3\right)^{1/3}\leq 25\varepsilon^{1/3} k_1.\]

\end{proof}

\begin{proposition}\label{assoc-x12-approx}
Fix $\varepsilon>0$. Let $\mathfrak{X}$ be an association scheme of rank 4. Suppose that the parameters of $\mathfrak{X}$ satisfy $\frac{1}{\varepsilon}\leq k_1$, $p_{1,1}^{2}\leq \varepsilon k_1$ and $p^{3}_{i,j}\leq \varepsilon \min(k_{i}, k_{j})$ for $\{i,j\} = \{1,2\}$. Then the zero-weight spectral radius $\xi(X_{1,2})$ of $X_{1,2}$ satisfies
\[\xi(X_{1, 2})\leq \frac{p_{1,1}^{1}+p_{1,2}^{2}+p_{2,2}^{2}+\sqrt{(p_{2,2}^{2}+ p_{1,2}^{2}- p_{1,1}^{1})^2+4p_{1,2}^{2}p_{2,2}^{1}}}{2}+25\varepsilon^{1/3} (k_1+k_2). \]
\end{proposition}
\begin{notation} We will use the non-asymptotic notation $y = \square(x)$ to say that $|y|\leq x$.
\end{notation}
\begin{proof}
The proof is similar to the proofs of Proposition \ref{assoc-spectral-radius} and Lemma \ref{rank-4-roots}.
Denote $k = k_1+k_2$.  By Eq.~\eqref{eq-int-num} we have
\begin{equation}
(A_1+A_2)^{2} = (p_{1,1}^{1}+2p_{1,2}^{1}+p_{2,2}^{1})A_1+(p_{1,1}^{2}+2p_{1,2}^{2}+p_{2,2}^{2})A_2+(p_{1,1}^{3}+2p_{1,2}^{3}+p_{2,2}^{3})A_3+kI.
\end{equation}
Note that, by assumptions of this proposition $0\leq p_{1,1}^{3}+2p_{1,2}^{3}+p_{2,2}^{3} \leq 2\varepsilon k$ and $0\leq p_{1,2}^{1} = \frac{k_2}{k_1}p_{1,1}^2\leq \varepsilon k$. Using Eq. \eqref{eq-constit-sum} we eliminate $A_3$.
\begin{equation}\label{eq-eigx12-sq}
\begin{multlined}
(A_1+A_2)^{2} = (p_{1,1}^{1}+p_{2,2}^{1}+2\square(\varepsilon k))A_1+ (2p_{1,2}^{2}+p_{2,2}^{2}+2\square(\varepsilon k))A_2+
\\
+(k+2\square(\varepsilon k))I+2\square(\varepsilon k)J = \\
 = (p_{1,1}^{1}+p_{2,2}^{1}+2\square(\varepsilon k))(A_1+A_2)+2\square(\varepsilon k^2)I+2\square(\varepsilon k)J+\\
 +\left(2p_{1,2}^{2}+p_{2,2}^{2}-p_{1,1}^{1}-p_{2,2}^{1}+4\square(\varepsilon k)\right)A_2.
\end{multlined}
\end{equation}
Denote by $R = 2p_{1,2}^{2}+p_{2,2}^{2}-p_{1,1}^{1}-p_{2,2}^{1}+4\square(\varepsilon k)$ the last coefficient in Eq.~\eqref{eq-eigx12-sq}. Multiplying Eq. \eqref{eq-eigx12-sq} by $(A_1+A_2)$ we get
\begin{equation}\label{eq-eigx12-cub}
\begin{multlined}
(A_1+A_2)^{3} = (p_{1,1}^{1}+p_{2,2}^{1}+2\square(\varepsilon k))(A_1+A_2)^{2}+2\square(\varepsilon k^2)(A_1+A_2)+2\square(\varepsilon k^2)J+\\
+R\left(  (p_{1,2}^{1}+p_{2,2}^{1}+\square(\varepsilon k))(A_1+A_2)+(k_2+\square(\varepsilon k))I+\square(\varepsilon k)J \right)+\\
+R(p_{1,2}^{2}+p_{2,2}^{2}-p_{1,2}^{1}+p_{2,2}^{1})A_2 = \\
 = (p_{1,1}^{1}+p_{2,2}^{1}+2\square(\varepsilon k))(A_1+A_2)^{2}+9\square(\varepsilon k^2)(A_1+A_2)+5\square(\varepsilon k^2)J+\\
 + (2p_{1,2}^{2}+p_{2,2}^{2}-p_{1,1}^{1}-p_{2,2}^{1})p_{2,2}^{1}(A_1+A_2)+3\square(\varepsilon k^{2})I\\
 +R(p_{1,2}^{2}+p_{2,2}^{2} - p_{1,2}^{1} - p_{2,2}^{1})A_2.
\end{multlined}
\end{equation}

Multiply Eq. \eqref{eq-eigx12-sq} by $p_{1,2}^{2}+p_{2,2}^{2} - p_{1,2}^{1} - p_{2,2}^{1} = p_{1,2}^{2}+p_{2,2}^{2} - p_{2,2}^{1}+\square(\varepsilon k)$ to eliminate $A_2$ from Eq.~\eqref{eq-eigx12-cub}. Observe first 
\[(2p_{1,2}^{2}+p_{2,2}^{2} - p_{1,1}^{1} - p_{2,2}^{1})p_{2,2}^{1} - (p_{1,1}^{1}+p_{2,2}^{1})(p_{1,2}^{2}+p_{2,2}^{2} - p_{2,2}^{1}) = p_{1,2}^{2}p_{2,2}^{1} - p_{1,1}^{1}p_{1,2}^{2} - p_{1,1}^{1}p_{2,2}^{2}.\]
Thus, 
\begin{equation}\label{eq-poly-for-x12}
\begin{multlined}
(A_1+A_2)^{3} - (p_{1,1}^{1}+p_{2,2}^{2} +p_{1,2}^{2})(A_1+A_2)^{2}-\\-(p_{1,2}^{2}p_{2,2}^{1} - p_{1,1}^{1}p_{1,2}^{2} - p_{1,1}^{1}p_{2,2}^{2})(A_1+A_2)+\\
+3\square(\varepsilon k)(A_1+A_2)^{2}+13\square(\varepsilon k^{2})(A_1+A_2)+5\square(\varepsilon k^{3})I+8\square(\varepsilon k^{2})J = 0.
\end{multlined}
\end{equation}
Consider
\[ b_1 = -(p_{1,1}^{1}+p_{2,2}^{2} +p_{1,2}^{2}), \quad b_2 =  p_{1,1}^{1}(p_{1,2}^{2} + p_{2,2}^{2}) - p_{1,2}^{2}p_{2,2}^{1}, \quad b_3 = 0.\]
Then, Eq. \eqref{eq-poly-for-x12} implies that any non-trivial eigenvalue $\eta$ of $X_{1,2}$ satisfies the polynomial $\eta^{3}+a_1\eta^{2}+a_2\eta+a_3 = 0$, where
\[ |a_1 - b_1|\leq 3\varepsilon k, \quad |a_2 - b_2|\leq 13\varepsilon k^2, \quad |a_3 - b_3|\leq 5\varepsilon k^{3}.\]

Denote by $\nu_1, \nu_2, \nu_3$ non-trivial eigenvalues of $A_1+A_2$. Therefore, by Theorem \ref{poly-approx} we can enumerate the roots $x_1, x_2, x_3$ of  $x^3+b_1x^2+b_2x+b_3 = 0$ so that $|\nu_i - x_i|\leq \delta$, where
\[\delta = 6 \left(3\varepsilon k(2k)^2+13\varepsilon k^2 (2k)+5\varepsilon k^3\right)^{1/3}\leq 25\varepsilon^{1/3} k.\]
Here we use that, by Eq. \eqref{eq-sum-param}, inequalities $b_{1} = p_{1,1}^{1}+(p_{2,1}^{2}+p_{2,2}^{2})\leq k$ and $b_2\leq k^{2}$ hold.

\end{proof}

\subsection{Reduction to the case of a constituent with a clique geometry}\label{sec-coh-reduction} 

\begin{lemma}\label{assoc-k2-large}
Let $\mathfrak{X}$ be an association scheme of rank 4 and diameter $2$ with the constituents ordered by degree. If $k_2\geq \gamma k_3$, then every pair of distinct vertices is distinguished by at least $\frac{\gamma(n-1)}{6}$ vertices.
\end{lemma}
\begin{proof}
Since $\mathfrak{X}$ has diameter 2 it is enough to show that some pair of vertices is distinguished by at least $\frac{\gamma(n-1)}{3}$ vertices, as then result follows by Lemma \ref{babai-dist}. Observe that vertices $u,v$ connected by an edge of color $i$ are distiguished by at least $|N_2(u)\bigtriangleup N_2(v)| = 2(k_2 - p_{2,2}^i)$ vertices. At the same time we have
\[k_2 (k_2 - 1) = \sum\limits_{i = 1}^{3} k_i p_{2, 2}^{i}.\]
Thus, $k_3\geq k_2$ implies $k_2\geq p_{2,2}^{2}+p_{2,2}^{3}$, so $\min(p_{2,2}^{2}, p_{2,2}^{3})\leq \frac{k_2}{2}$. Hence, vertices connected by an edge of color $i$, which minimize $p_{2,2}^{i}$, are distinguished by at least $k_2\geq \gamma k_3\geq \frac{\gamma (n-1)}{3}$ vertices.
\end{proof}
\begin{remark}
Note that the result of the lemma can be also derived directly from Proposition~6.3 proven by Babai in \cite{Babai-annals}.
\end{remark}

\begin{lemma}\label{assoc-param-ineq}
Let $\mathfrak{X}$ be an association scheme of rank 4 and diameter $2$ with the largest degree equal $k_3$. Fix some $\varepsilon>0$. 
Assume $\max(k_1, k_2)\leq \frac{\varepsilon}{2} k_3$. Then
\begin{equation}\label{eq-statement-2} 
\quad p_{1,2}^{3}\leq \varepsilon k_1, \quad p_{1,1}^{3}\leq \varepsilon k_1, \quad p_{2,2}^{3}\leq \varepsilon k_2, \quad \text{and}
\end{equation}
\begin{equation}
p^{1}_{3,3}\geq k_3(1-\varepsilon), \quad p^{2}_{3,3}\geq k_3(1-2\varepsilon).
\end{equation}

\end{lemma}
\begin{proof}
 Note that $k_i(k_i-p_{i,i}^{i}-1)\geq k_3p^{3}_{i,i}$, so $p^{3}_{i,i}\leq \varepsilon k_i/2$ for $i=1,2$. Additionally, $p^{1}_{2,3}\leq k_2\leq \varepsilon k_3/2$. Thus, by Eq.~\eqref{eq-sum-param}, $p_{1,2}^{3} = \frac{k_1}{k_3}p_{2,3}^{1} \leq \varepsilon k_1/2$.

Finally, by Eq.~\eqref{eq-sum-param} we have $p_{3,3}^i+p_{3, 2}^i+p_{3, 1}^i = k_3$ and $p_{3,j}^{i}\leq k_j\leq \varepsilon k_3/2$ for $j\in \{1, 2\}$. Therefore, $k_{3,3}^{i}\geq (1-\varepsilon)k_3$.

\end{proof}
\begin{remark} Note that in inequalities from Eq. \eqref{eq-statement-2} are still true if we replace $\varepsilon$  by $\varepsilon/2$.
\end{remark}

We need another generalization of Lemma \ref{max-min}, corollaries of which will be used several times.

\begin{lemma}\label{assoc-tr} Let $\mathfrak{X}$ be an association scheme. Suppose that there exists a triangle with sides of colors $(s,r,t)$. Then  
\[  p_{i,j}^{s}+p_{j,l}^{r}\leq k_j+p_{i,l}^{t}. \]
\end{lemma}
\begin{proof}
Apply the inclusion
\[N_i(u)\setminus N_l(w) \subseteq (N_i (u)\setminus N_j(v))\cup (N_j(v)\setminus N_l(w)),\] to vertices $u,v,w$, where $c(u,v) = s$, $c(v,w) = r$, and $c(u,w) = t$.
\end{proof}
\begin{corollary}\label{cor1} Suppose an association scheme $\mathfrak{X}$ of rank 4 satisfies $\max(k_1, k_2)\leq \frac{\varepsilon}{2} k_3$. Suppose also that there exists a triangle with sides $(s,t,3)$.
Then,  $p_{i,j}^{s}\leq p^{t}_{i,3}+\varepsilon k_j$, where $i,j\in \{1,2\}$.
\end{corollary}
\begin{proof}
Take $r = l = 3$ in Lemma \ref{assoc-tr} and note that by Lemma~\ref{assoc-param-ineq} we have $p_{j,3}^{3} = \frac{k_j}{k_3}p_{3,3}^{j}\geq (1-\varepsilon)k_j$.
\end{proof}

\begin{corollary}\label{cor2} Suppose an association scheme $\mathfrak{X}$ has rank 4 and diameter~2. Moreover, assume $\max(k_1, k_2)\leq \frac{\varepsilon}{2} k_3$.
Then,  $p_{i,j}^{s}\leq p^{s}_{i,3}+\varepsilon k_j$, where $i,j,s \in \{1,2\}$.
\end{corollary}
\begin{proof}
Take $r = l = 3$ and $s = t$ in Lemma \ref{assoc-tr}. Observe, that a triangle with sides of colors $(s,s,3)$ exists, as diameter of $\mathfrak{X}$ is 2 and $s\neq 3$.
\end{proof}

\begin{corollary}\label{cor3} Suppose an association scheme $\mathfrak{X}$ has rank 4 and diameter~2. Assume $\max(k_1, k_2)\leq \frac{\varepsilon}{2} k_3$.
Then, $2p_{i,j}^{s}\leq k_j+\varepsilon k_i$, where $i,j,s\in \{1,2\}$. Moreover, if $k_1\leq k_2$, then $2p_{1,2}^{2}\leq (1+\varepsilon)k_1$.
\end{corollary}
\begin{proof}
 Take $t = 3$, $s = r$ and $l = i$ in Lemma \ref{assoc-tr} and we use Lemma \ref{assoc-param-ineq}. Take $s=1$ and $i=j=2$, then $2p_{2,2}^{1}\leq (1+\varepsilon)k_2$, so $2p_{1,2}^{2} \leq \frac{k_1}{k_2} (1+\varepsilon)k_2 = (1+\varepsilon)k_1$.
\end{proof}

We state the following simple corollary of Metsch's criteria (Theorem \ref{Metsch}) and of the classification of graphs with the smallest eigenvalue $\geq -2$ (Theorems \ref{geom-eig-more2} and \ref{geom-eig-2}), which will be used in the proof of Proposition \ref{assoc-diam2}. We would like to note, that the lemma below is considerably easier then the classification of graphs with the smallest eigenvalue $\geq -2$, as the fact that the given graph is a line graph comes for free.

\begin{lemma}\label{assoc-clique-geom}
Let $\mathfrak{X}$ be an association scheme of rank $r\geq 4$ and diameter 2 on $n$ vertices.
\begin{enumerate}
 \item Assume that for some $i$ the constituent $X_i$ satisfies the assumptions of Theorem \ref{Metsch} for $m=2$. Then $X_i$ is a strongly regular graph with smallest eigenvalue $-2$, or is the line graph of a regular triangle-free graph.
 
 In particular, $X_i$ satisfies the assumptions of Theorem \ref{Metsch} for $m=2$, if we have one of the following
 \begin{enumerate}
 \item $\lambda(X_i) = p_{i,i}^{i}\geq\frac{2}{5}k_i$ and $\mu(X_i) = \max\{p_{i,i}^{j}: 0<j\neq i\} \leq \frac{1}{30}k_i$, or
 \item $\lambda(X_i) = p_{i,i}^{i}\geq(\frac{1}{2}-\frac{1}{20})k_i$ and $\mu(X_i) = \max\{p_{i,i}^{j}: 0<j\neq i\} \leq \frac{1}{11}(1+\frac{1}{100})k_i$. 
\end{enumerate}

\item Assume for some $i$ the complement $\overline{X_{i}}$ of $X_i$ satisfies the conditions of Theorem \ref{Metsch} for $m=2$. If $n\geq 12$, then graph $\overline{X_{i}}$ is strongly regular with smallest eigenvalue $-2$.  
\end{enumerate} 
\end{lemma}
\begin{proof}
\begin{enumerate}
\item First, we check that $X_i$ satisfies the conditions of Theorem \ref{Metsch} for $m=2$.
\begin{enumerate}
\item Compute 
$$\lambda(X_i) \geq \frac{2}{5}k_i > \frac{3}{30}k_i\geq 3\mu(X_i)\quad \text{and}\quad 3\lambda(X_i)-3\mu(X_i)\geq \left(\frac{6}{5}-\frac{1}{10}\right)k_i>k_i.$$
\item Compute 
\[\lambda(X_i) - 3\mu(X_i) \geq \left(\frac{1}{2} - \frac{1}{20} - \frac{3}{11}\left(1+\frac{1}{100}\right)\right)k_i = \frac{48}{275}k_i>0,\]
\[3\lambda(X_i)-3\mu(X_i)\geq k_i+\left(\frac{1}{2} - \frac{3}{20} - \frac{3}{11}\left(1+\frac{1}{100}\right)\right)k_i = k_i+\frac{41}{550}k_i>k_i.\]
\end{enumerate}
Now, if $X_i$ satisfies the conditions of Theorem \ref{Metsch} for $m=2$, by Remark \ref{m2-line}, it is a line graph and by Remark \ref{geom-smallest-eigenvalue} the smallest eigenvalue of $X_i$ is at least $-2$. Moreover, recall that $X_i$ is edge-regular. If the smallest eigenvalue is $>-2$, by Theorem~\ref{geom-eig-more2}, $X_i$ is a complete graph or an odd polygon. It is impossible, since $\mathfrak{X}$ has diameter $2$ and at least three non-empty constituents.  If the smallest eigenvalue is $-2$, then by Theorem~\ref{geom-eig-2}, we get that $X_i$ is a strongly regular graph, or is the line graph of a regular triangle-free graph.

\item Since $\overline{X_{i}}$ satisfies the conditions of Theorem \ref{Metsch}, by Remarks \ref{m2-line} and \ref{geom-smallest-eigenvalue}, graph $\overline{X_{i}}$ is a line graph and its smallest eigenvalue is at least $-2$. Note also that $\overline{X_{i}}$ is co-edge-regular. If the smallest eigenvalue is $>-2$, by Theorem \ref{geom-eig-more2}, $\overline{X_{i}}$ is complete graph or an odd polygon. It is impossible, since $\mathfrak{X}$ has diameter $2$ and at least three non-empty constituents. If the smallest eigenvalue is $-2$, then by Theorem \ref{geom-eig-2}, we get that $\overline{X_{i}}$ is strongly regular, an $m_1\times m_2$-grid or one of the two regular subgraphs of the Clebsh graph with 8 or 12 vertices.

Assume $\overline{X_{i}}$ is a $m_1\times m_2$-grid with $m_1\neq m_2$ and $m_1, m_2>1$. That is, $\overline{X_{i}}$ is the line graph of $K_{m_1, m_2}$. Denote the parts of $K_{m_1, m_2}$ by $U_1$ and $U_2$ with $|U_i| = m_i$. By symmetry we can assume $m_1< m_2$. Compute $n = m_1m_2$, $k_1+k_2 = m_1+m_2 - 2$ and $\mu = 2$. Observe that, two edges of $K_{m_1, m_2}$ that share a vertex in $U_i$ have $m_i - 2$ common neighbors. Since $m_1 \neq m_2$, two pairs of edges in $K_{m_1, m_2}$ that share a vertex in $U_1$ and $U_2$, respectively, cannot be connected by the same color in  $\mathfrak{X}$. Thus, in particular, $\mathfrak{X}$ is not primitive.  

Therefore, $\overline{X_{i}}$ is strongly regular.  
\end{enumerate}
\end{proof}

\begin{proposition}\label{assoc-diam2}
Let $\mathfrak{X}$ be an association scheme of rank 4 on $n$ vertices  with diameter~2 and with constituents ordered by degree. Recall that $q(X_i) = \max\{p_{i,i}^{j}: j\in [3]\}$ is the maximal number of common neighbours of two distinct vertices in $X_i$. Fix $\varepsilon = 10^{-16}$. Then one of the following is true.
\begin{enumerate}
  \item Every pair of distinct vertices is distinguished by at least $\frac{\varepsilon k_3}{4} \geq \frac{\varepsilon(n-1)}{12}$ vertices.
  \item The zero-weight spectral radius $\xi(X_i)$ of $X_i$ satisfies $q(X_i)+\xi(X_i)\leq (1-\varepsilon)k_i$ for $i = 1$ or $i = 2$.  
  \item The graph $X_1$ is either strongly regular with smallest eigenvalue $-2$, or the line graph of a connected regular triangle-free graph. 
  \item The graph $X_2$ is either, strongly regular with smallest eigenvalue $-2$ or the line graph of a connected regular triangle-free graph. Moreover, $k_2\leq \frac{101}{100}k_1$.
  \item If $n\geq 12$, then the graph $X_{1,2}$ is strongly regular with smallest eigenvalue $-2$ and $k_2\leq \frac{101}{100}k_1$.
  \end{enumerate}  
\end{proposition}
\begin{proof}
We may assume that parameters of $\mathfrak{X}$ satisfy $\max(k_1, k_2)\leq \frac{\varepsilon}{2} k_3$, otherwise, by Lemma~\ref{assoc-k2-large} every pair of distinct vertices is distinguished by at least $\frac{\varepsilon k_3}{4}$ vertices. Therefore, all the inequalities provided by Lemma \ref{assoc-param-ineq} hold. 

Thus, by Proposition \ref{assoc-spectral-radius},
\[\xi(X_1)\leq \frac{p_{1,1}^{1}+p_{1,2}^{2}+\sqrt{(p_{1,1}^{1}-p_{1,2}^{2})^{2}+4p_{1,1}^{2}p_{1,2}^{1}}}{2}+\varepsilon_1 k_1, \text{ so}\]
\begin{equation}\label{x1-eq-bound}
\xi(X_1)\leq \max(p_{1,1}^{1}, p_{1,2}^{2})+\sqrt{p_{1,1}^{2} p_{1,2}^{1}}+\varepsilon_1 k_1,
\end{equation}
where $\varepsilon_1  = 25\varepsilon^{1/3}$. We note, by Eq. \eqref{eq-statement-2}, $k_1 \geq \frac{p_{1,1}^3}{\varepsilon}\geq \frac{1}{\varepsilon}$.
Similarly,
\begin{equation}\label{x2-eq-bound}
\xi(X_2) \leq \max(p_{2,2}^{2}, p_{2,1}^{1})+\sqrt{p_{2,2}^{1}p_{1,2}^{2}}+\varepsilon_1 k_2.
\end{equation}

\noindent \textbf{Case 1.} Assume $\gamma k_2> k_1$, where $\gamma = \frac{1}{900}$.
 
\noindent Then $p_{1,1}^2  = \frac{k_1}{k_2} p_{1,2}^{1}\leq \gamma p_{1,2}^{1}$, so $\mu(X_1)\leq \max(\gamma, \varepsilon) k_1  = \frac{1}{900}k_1$. Note, by Corollary \ref{cor3}, inequality $\max(p_{1,1}^{1}, p_{1,2}^{2})\leq \frac{1+\varepsilon}{2}k_1$ holds. Hence, by Eq. \eqref{x1-eq-bound},
\begin{equation}\label{eq-case-1}
q(X_1)+\xi(X_1)\leq ((\varepsilon+ \gamma)k_1+p_{1,1}^{1})+\left(\frac{1+\varepsilon}{2}k_1+\sqrt{\gamma}k_1+\varepsilon_1 k_1\right).
\end{equation}

If $p_{1,1}^{1}<\frac{2}{5}k_1$, then Eq. \eqref{eq-case-1} implies $q(X_1)+\xi(X_1)\leq (1-\varepsilon)k_1$. Otherwise, if $p_{1,1}^{1}\geq\frac{2}{5}k_1$, by Lemma \ref{assoc-clique-geom}, graph $X_1$ satisfies the statement 3 of this proposition.

\noindent \textbf{Case 2.} Assume $k_1 = \gamma k_2$, where $(1+\varepsilon_3)^{-1} \geq \gamma\geq \frac{1}{900}$, for $\varepsilon_3 = \frac{1}{100}$. 
We consider two subcases.

\noindent\textbf{Case 2.1.}
Suppose $p_{2,3}^{1} = 0$, so by Corollary \ref{cor2},
\begin{equation}
p_{2,2}^{1}\leq \varepsilon k_2+p_{2,3}^{1} = \frac{\varepsilon}{\gamma}k_1, \quad p_{1,2}^{2} = \frac{k_1}{k_2} p_{2,2}^{1} \leq \varepsilon k_1,
\end{equation}
\begin{equation}\label{eq212}
\quad p_{1,2}^{1}\leq \varepsilon k_1 + p_{2,3}^{1} = \varepsilon k_1, \quad p_{1,1}^{2} = \frac{k_1}{k_2} p_{1,2}^{1} \leq \varepsilon k_1.
\end{equation} 
 Then,
\begin{equation}
 \max_{i\in [3]}(p_{1,1}^{i})+\max(p_{1,1}^{1}, p_{1,2}^{2})+\sqrt{p_{1,1}^{2} p_{1,2}^{1}} \leq 3\varepsilon k_1 +2p_{1,1}^{1}.
\end{equation} 
 Thus, by Eq.~\eqref{x1-eq-bound}, $q(X_1)+\xi(X_1)<(1-\varepsilon)k_1$  if $p_{1,1}^{1}<\frac{2}{5}k_1$. Alternatively, if $p_{1,1}^{1}\geq \frac{2}{5}k_1$, by Lemma \ref{assoc-clique-geom}, graph $X_1$ satisfies the statement 3 of this proposition, as $\mu(X_1)\leq \varepsilon k_1$ by Eq.~\eqref{eq212}.

\noindent\textbf{Case 2.2.}
 Suppose $p_{2,3}^{1} \neq 0$. 
\begin{enumerate}
\item[]\textbf{Case 2.2.1.} Assume that $p_{1,1}^{1}\geq p_{1,1}^{2}$. 

By Corollary \ref{cor1},
 $$p_{1,2}^{2}\leq p_{1,3}^{1}+\varepsilon k_2 = p_{1,3}^{1}+\frac{\varepsilon}{\gamma}k_1 \quad \text{and} \quad p_{1,1}^{1}\leq p_{1,3}^{1}+\varepsilon k_1.$$
Then 
\begin{equation}\label{eq-case-2211}
\max_{i\in [3]}(p_{1,1}^{i})+p_{1,2}^{2}+\sqrt{\gamma}p_{1,2}^{1}\leq 
p_{1,1}^{1}+p_{1,3}^{1}+p_{1,2}^{1} - (1-\sqrt{\gamma})p_{1,2}^{1}+\left(\varepsilon+\frac{\varepsilon}{\gamma}\right)k_1,
\end{equation}
\begin{equation}
\max_{i\in [3]}(p_{1,1}^{i})+p_{1,1}^{1}+\sqrt{\gamma}p_{1,2}^{1}\leq 
p_{1,3}^{1}+p_{1,1}^{1}+p_{1,2}^{1} - (1-\sqrt{\gamma})p_{1,2}^{1}+\left(\varepsilon+\frac{\varepsilon}{\gamma}\right)k_1 .
\end{equation}

\item[]\textbf{Case 2.2.2.} Assume that $p_{1,1}^{2}\geq p_{1,1}^{1}$. 

By Corollary \ref{cor2}, 
\[p_{1,1}^{2} = \gamma p_{1,2}^{1}\leq \gamma (p_{1,3}^{1}+\varepsilon k_2) \leq \gamma p_{1,3}^{1}+\varepsilon k_1. \]
This implies
\begin{equation}
 \max_{i\in [3]}(p_{1,1}^{i})+p_{1,1}^{1}+\sqrt{\gamma}p_{1,2}^{1}\leq \gamma p_{1,3}^{1}+p_{1,1}^{1}+\sqrt{\gamma}p_{1,2}^{1}+\varepsilon k_1,
\end{equation}
\begin{equation}\label{eq-case-2222}
 \max_{i\in [3]}(p_{1,1}^{i})+p_{1,2}^{2}+\sqrt{\gamma}p_{1,2}^{1}\leq p^{2}_{1, 1}+p_{1,2}^{2}+p_{1,3}^{2} - (1-\sqrt{\gamma})p_{1,2}^{1} +\left(\varepsilon +\frac{\varepsilon}{\gamma}\right)k_1,
\end{equation}
where in Eq.~\eqref{eq-case-2222} we use the inequality $p_{1,2}^{1}\leq p_{1,3}^{2}+\varepsilon k_2$ given by Corollary \ref{cor1}.
\end{enumerate}
That is, using Eq. \eqref{eq-sum-param}, in both subcases by Eq.~\eqref{eq-case-2211}~-~\eqref{eq-case-2222} we get
\[ \max_{i\in [3]}(p_{1,1}^{i})+\left(\max(p_{1,1}^{1}, p_{1,2}^{2})+\sqrt{p_{1,1}^{2} p_{1,2}^{1}}\right) \leq k_1 - (1-\sqrt{\gamma})p_{1,2}^{1} +\left(\varepsilon +\frac{\varepsilon}{\gamma}\right)k_1,\]
\begin{equation}\label{eq-two-case-bound}
q(X_1)+\xi(X_1)\leq k_1 - (1-\sqrt{\gamma})p_{1,2}^{1} +\left(\varepsilon +\frac{\varepsilon}{\gamma}\right)k_1 +\varepsilon_1 k_1.
\end{equation}

\begin{enumerate}
\item[] \textbf{Case 2.2.a.} Suppose $p_{1,2}^{1}>\varepsilon_2 k_1$ for $\varepsilon_2 = \frac{1}{30}$.
\[\varepsilon_2(1-\sqrt{\gamma}) - \varepsilon\left(2+\frac{1}{\gamma}\right)-\varepsilon_1\geq 10^{-4} - 902\varepsilon - 25\varepsilon^{1/3}>0,\]
 so by Eq.~\eqref{eq-two-case-bound},
\begin{equation}\label{eq-e2-small}
q(X_1)+\xi(X_1) \leq \max_{i\in [3]}(p_{1,1}^{i})+\max(p_{1,1}^{1}, p_{1,2}^{2})+\sqrt{p_{1,1}^{2} p_{1,2}^{1}}+\varepsilon_1 k_1 \leq (1-\varepsilon)k_1.
\end{equation}

\item[]\textbf{Case 2.2.b.} Suppose $p_{1,2}^{1}\leq \varepsilon_2 k_1$. 

This implies $p_{1,1}^{2}\leq \varepsilon_2 k_1$, so $\mu(X_1)\leq \frac{1}{30}k_1$. Recall, that by Corollary~\ref{cor3}, the  inequality $\max(p_{2,1}^{2}, p_{1,1}^{1})\leq \frac{1+\varepsilon}{2} k_1 $ holds. Then we have 
\begin{equation}\label{eq-e2-large}
q(X_1)+\xi(X_1)\leq (\varepsilon+\varepsilon_2)k_1+p_{1,1}^{1}+\frac{1+\varepsilon}{2}k_1+ \varepsilon_2k_1+\varepsilon_1 k_1.
\end{equation} 
 Thus, either Eq.~\eqref{eq-e2-large} implies the inequality $q(X_1)+\xi(X_1)<(1-\varepsilon)k_1$, or $p_{1,1}^{1}\geq \frac{2}{5} k_1$ and, by Lemma \ref{assoc-clique-geom}, the graph $X_1$ satisfies the statement 3 of this proposition.
 \end{enumerate}

\noindent \textbf{Case 3.} Suppose that $k_2\leq (1+\varepsilon_3)k_1$, where $\varepsilon_3 = \frac{1}{100}$. 

\noindent In this case we will work with both $X_1$ and $X_2$ in the same way. Additionally, we will need to consider the graph $X_{1,2}$ with the set of vertices $V(X_{1,2}) = V(X_1) = V(X_2)$ and set of edges $E(X_{1,2}) = E(X_1)\cup E(X_2)$. Observe that $X_{1,2}$ is the complement of the constituent graph $X_3$. Graph $X_{1,2}$ is regular of degree $k_1+k_2$, and any two non-adjacent vertices have 
\begin{equation}\label{eq-mux12-above}
\mu(X_{1,2}) = p_{1,1}^{3}+2 p_{1,2}^{3}+p_{2,2}^{3} \leq 2\varepsilon (k_1+k_2) \leq 4\varepsilon(1+\varepsilon_3)k_1 
\end{equation} 
common neighbors. Any two vertices connected by an edge of color $i$ have $$\lambda_i = p_{1,1}^{i}+2 p_{1,2}^{i}+p_{2,2}^{i}$$ common neighbors, for $i = 1,2$. Applying inclusion from Eq. \eqref{eq1} to the triangle with the side colors $(i,i,3)$, where $i\in \{1,2\}$, we get $2\lambda_i \leq k_1+k_2+\mu(X_{1,2})$, i.e., 
\begin{equation}\label{eq-lambda-above}
\lambda_i = p_{1,1}^{i}+2 p_{1,2}^{i}+p_{2,2}^{i}\leq \frac{1+2\varepsilon}{2}(k_1+k_2) \leq k_1(1+\varepsilon_3+2\varepsilon).
\end{equation} 
Let $\{i, j\} = \{1,2\}$, then by Eq.~\eqref{x1-eq-bound}-\eqref{x2-eq-bound},
\begin{equation}\label{eq-case-3}
\begin{multlined}
q(X_i)+\xi(X_i) \leq q(X_i)+\max( p_{i,i}^{i}, p_{i,j}^{j})+\sqrt{p_{i,j}^{i}p_{i,i}^{j}} +\varepsilon_1 k_i \leq  \\ \leq \max( p_{i,i}^{i}, p_{i,i}^{j})+\varepsilon k_i+ \max( p_{i,i}^{i}, p_{i,j}^{j})+p_{i,j}^{i}(1+\varepsilon_3)+\varepsilon_1 k_i.
\end{multlined}
\end{equation}
Consider all possible choices of opening the maximums in \eqref{eq-case-3} (we write terms without epsilons).
\begin{enumerate}
\item $2p_{i,i}^{i}+p_{i,j}^{i}$,
\item $p_{i,i}^{i}+p_{i,i}^{j}+p_{i,j}^{i}\leq (1+\varepsilon_3)(p_{i,i}^{i}+2p_{i,j}^{i}) = (1+\varepsilon_3)(\lambda_i - p_{j,j}^{i})\leq (1+\varepsilon_3)\lambda_i,$
\item $p_{i,j}^{j}+p_{i,i}^{i}+p_{i,j}^{i}\leq (1+\varepsilon_3)(p_{j,j}^{i}+p_{i,i}^{i}+p_{i,j}^{i}) = (1+\varepsilon_3)(\lambda_i - p_{i,j}^{i})\leq (1+\varepsilon_3)\lambda_i,$
\item $p_{i,j}^{j}+p_{i,i}^{j}+p_{i,j}^{i}\leq (1+\varepsilon_3)(p_{j,j}^{i}+2p_{i,j}^{i}) = (1+\varepsilon_3)(\lambda_i - p_{i,i}^{i})\leq (1+\varepsilon_3)\lambda_i.$
\end{enumerate}
Hence, by Corollary \ref{cor3}, Eq.~\eqref{eq-case-3} implies
\begin{equation}\label{eq-max-open}
q(X_i)+\xi(X_i)\leq \max(2p_{i,i}^{i}+p_{i,j}^{i}, (1+\varepsilon_3)\lambda_i)+k_i(\varepsilon_1+\frac{2}{3}\varepsilon_3+\varepsilon).
\end{equation}

If $\lambda_t\geq (\frac{2}{3}+\frac{1}{300})k_t$ for both $t =1,2$, then in notations of Theorem \ref{Metsch} 
\[\lambda^{(1)} \geq \left(\frac{2}{3}+\frac{1}{300}\right)k_1,\quad \text{and by Eq. \eqref{eq-lambda-above},}\quad \lambda^{(2)} \leq \frac{11}{10}k_1.\]
Check,
\[2\lambda^{(1)} - \lambda^{(2)}\geq 20\varepsilon(1+\varepsilon_3)k_1\geq 5\mu, \quad \text{and} \quad 3\lambda^{(1)} - 3\mu \geq 2k_1+\frac{1}{100}k_1\geq k_1+k_2.\] Thus, $X_{1,2}$ satisfies conditions of Theorem \ref{Metsch}, so  the statement 5 of this proposition holds by Lemma \ref{assoc-clique-geom}.

 Suppose that $\lambda_i\leq (\frac{2}{3}+\frac{1}{300})k_i$. If $2p_{i,i}^{i}+p_{i,j}^{i}\leq k_i - k_i(\varepsilon_3 +2\varepsilon+\varepsilon_1)$, then Eq.~\eqref{eq-max-open} implies
 \[q(X_i)+\xi(X_i)\leq (1-\varepsilon)k_i.\]
Hence, we can assume $2p_{i,i}^{i}+p_{i,j}^{i}\geq k_i - k_i(\varepsilon_3 +2\varepsilon+\varepsilon_1)$. Recall, that 
\[ 2p_{i,i}^{i}+p_{i,j}^{i} = \lambda_i+(p_{i,i}^{i} - p_{i,j}^{i}- p_{j,j}^{i})\leq \lambda_i+\frac{1+\varepsilon}{2}k_i - (p_{i,j}^{i}+ p_{j,j}^{i}) .\]
Thus,
\[ p_{i,j}^{i}+ p_{j,j}^{i}\leq \lambda_i+\frac{1+\varepsilon}{2}k_i-k_i(1-\varepsilon_3 - 2\varepsilon - \varepsilon_1)\leq k_i \frac{51}{300} + k_i (\varepsilon_3+3\varepsilon+\varepsilon_1)\leq \frac{2}{11}k_i.\]
It implies, 
\begin{equation}
\min(p_{j,j}^{i}, p_{i,i}^{j}) \leq (1+\varepsilon_3)\min(p_{i,j}^{i}, p_{j,j}^{i})\leq \frac{1+\varepsilon_3}{11}k_1.
\end{equation}

Take $\{s,t\} = \{1,2\}$, so that $p_{s,s}^{t}\leq \frac{1+\varepsilon_3}{11}k_1$. Then $\mu(X_s)\leq \max(\varepsilon k_s, \frac{1+\varepsilon_3}{11}k_1)=\frac{1+\varepsilon_3}{11}k_1$. We consider two possibilities. 

First, assume that $p_{s,s}^{s}\geq (\frac{1}{2} - \frac{1}{20})k_s$. 
Then, by Lemma \ref{assoc-clique-geom} graph $X_s$ satisfies the statement 3 or 4 of this proposition.

Assume now that $p_{s,s}^{s}\leq (\frac{1}{2}-\frac{1}{20})k_s$, then
\begin{equation}\label{eq31}
2p_{s,s}^{s}+p_{s,t}^{s}\leq k_s - \frac{1}{10}k_s +\frac{(1+\varepsilon_3)^2}{11}k_s \leq (1-2\varepsilon-\frac{2}{3}\varepsilon_3-\varepsilon_1)k_s,
\end{equation}
\begin{equation}\label{eq32}
\begin{multlined}
(1+\varepsilon_3)\lambda_s\leq p_{s,s}^{s}+2p_{s,t}^{s}+2p_{t,t}^{s}+2\varepsilon_3k_s \leq \\
\\ \leq \left(\frac{1}{2}-\frac{1}{20}+\frac{4+4\varepsilon_3}{11}+2\varepsilon_3\right)k_s<(1-2\varepsilon-\frac{2}{3}\varepsilon_3-\varepsilon_1)k_s.
\end{multlined}
\end{equation}

Thus, by Eq. \eqref{eq-max-open}, equations \eqref{eq31} and \eqref{eq32} imply
\[q(X_s)+\xi(X_s)\leq  \max(2p_{s,s}^{s}+p_{s,t}^{s}, (1+\varepsilon_3)\lambda_s)+(\varepsilon+\frac{2}{3}\varepsilon_3+\varepsilon_1)k_s\leq (1-\varepsilon)k_s.\]

\end{proof}

\subsection{Case of a constituent with a clique geometry}\label{sec-subsec-str-reg}

In the previous subsection Proposition \ref{assoc-diam2} reduces the diameter 2 case of Theorem \ref{main-coherent} to the case when one of the constituents is a strongly regular graph with smallest eigenvalue $-2$, or is the line graph of a triangle-free regular graph. In this subsection we resolve the remaining cases.
 
In the case when the dominant constituent $X_3$ is strongly regular we will introduce an additional tool (Prop. \ref{sun-wilmes-tool}), which allows us to bound the order of the group and its minimal degree, when vertices inside a clique are well-distinguished.

In the cases when constituent $X_1$ or $X_2$ is strongly regular, we will prove upper bounds on the quantity $q(X_J)+\xi(X_J)$ for $J\in \{1, 2, \{1, 2\}\}$ with the consequence that the spectral tool (Lemma \ref{mixing-lemma-tool}) can be applied effectively. 

The hardest case to analyze is the case when the constituent  of the smallest degree, $X_1\,$, is strongly regular. This case is settled in Proposition \ref{assoc-x1-strongly-reg} (Sec. \ref{sec-x1}) and it requires considerable preparatory work to establish a constant upper bound on the quotient $k_2/k_1$ in certain range of parameters.

\subsubsection{Triangular graphs with well-distinguished cliques}

To treat the case of strongly regular constituent we will need an additional tool: the argument similar in flavor to the one used by Sun and Wilmes (see Lemma 3.5 in \cite{Sun-Wilmes}).

\begin{proposition}\label{sun-wilmes-tool}
Let $\mathfrak{X}$ be a homogeneous coherent configuration on $n$ vertices. Let $I$ be a set of colors, such that if $i\in I$, then $i^{*}\in I$. Suppose that graph $X_I$ is the triangular graph $T(s)$ for some $s$. Denote a Delsarte clique geometry in $X_I$ by $\mathcal{C}$. Additionally, assume that there exist a constant $0<\alpha<\frac{1}{2}$, such that for every clique $C\in \mathcal{C}$ and every two distinct vertices $x, y \in C$ there exist at least $\alpha |C|$ elements $z\in C$ which distinguish $x$ and $y$, i.e., $c(z,x) \neq c(z,y)$. Then
 \begin{enumerate}
 \item There exists a set of vertices of size $O(\log(n))$ that completely splits $\mathfrak{X}$. Hence, $|\Aut(\mathfrak{X})| = n^{O(\log(n))}$,
 \item $\motion(\mathfrak{X})\geq \frac{\alpha}{2} n$.
 \end{enumerate}

\end{proposition}
\begin{proof} Consider any clique $C\in \mathcal{C}$. Since any two distinct vertices $x,y\in C$ are distinguished by at least $\alpha |C|$ vertices of $C$, by Lemma \ref{Babai-disting-order-group}, there is a set of size at most $\frac{2}{\alpha}\log(|C|)+1$ that splits $C$ completely.

Take any vertex $x\in \mathfrak{X}$. By the assumptions of this proposition $\{x\}\cup N_I(x) = C_1\cup C_2$ for some $C_1, C_2\in \mathcal{C}$. Then there exists a set $S$ of size $\frac{4}{\alpha}\log(|C|)+2\leq \frac{4}{\alpha}\log(n)+2$, that splits both $C_1$ and $C_2$ completely. Note that every clique $C\in \mathcal{C}$, distinct from $C_1$ and $C_2$, intersects each of them in exactly one vertex, and is uniquely determined by $C\cap C_1$ and $C\cap C_2$. Therefore, pointwise stabilizer $\Aut(\mathfrak{X})_{(S)}$ fixes every clique $C\in \mathcal{C}$ as a set. At the same time, every vertex $v$ is uniquely defined by the collection of cliques in $\mathcal{C}$ that contain~$v$. Therefore, $S$ splits $X$ completely. 

Suppose $\sigma \in \Aut(\mathfrak{X})$ and $|\supp(\sigma)|< \frac{\alpha}{2} n$. Then, by pigeonhole principle there exists a vertex $x$, such that $\sigma$ fixes at least $\lceil(1 - \frac{\alpha}{2})(|N_{I}(x)|+1)\rceil$ vertices in $N_I(x)\cup \{x\}$. Since $X_I = T(s)$, we have $\{x\}\cup N_I(x) = C_1\cup C_2$ for some $C_1, C_2\in \mathcal{C}$ and $|C_1| = |C_2| = 1+\frac{|N_I|}{2}$. Thus $\sigma$ fixes more than $(1-\alpha)|C_i|$ vertices in $C_i$ for every $i\in \{1,2\}$. This means that any pair of vertices  $x, y\in C_i$ is distinguished by at least one vertex fixed by $\sigma$. Hence, $\sigma(x)\neq y$. At the same time, since $(1-\alpha)|C_i|>1$, $\sigma(x)\in C_i$ for every $x\in C_i$. Therefore, $\sigma$ fixes poinwise both $C_1$ and $C_2$. Finally, by argument in the previous paragraph, we get that $\sigma$ fixes every vertex, so $\sigma$ is identity.

\end{proof}

\subsubsection{Constituent $X_3$ is strongly regular} 

In the following proposition we consider the case when constituent $X_3$ is a strongly regular graph (case of statement 5 in Proposition \ref{assoc-diam2}).
 
\begin{proposition}\label{assoc-x3-strongly-reg}
Let $\mathfrak{X}$ be an association scheme of rank 4 and diameter~2 on $n\geq29$ vertices. Assume additionally, that the constituents of $\mathfrak{X}$ are ordered by degree and $k_2\leq \frac{\varepsilon}{2}k_3$ holds for $\varepsilon < \frac{1}{100} $. Suppose that $k_2\leq \frac{11}{10}k_1$ and $X_{1,2}$ is strongly regular with smallest eigenvalue~$ -2$. Then neither $X_1$, nor $X_2$ are not strongly regular with smallest eigenvalue $-2$, and one of the following is true.
\begin{enumerate}
\item The association scheme $\mathfrak{X}$ satisfies the assumptions of Proposition \ref{sun-wilmes-tool} for $I = \{1, 2\}$ and  $\alpha  = \frac{1}{16}$.
\item $X_1$ or $X_2$ is the line graph of a regular triangle-free graph.
\end{enumerate} 
\end{proposition}
\begin{proof}
Note that by assumptions of the proposition all the inequalities from Lemma \ref{assoc-param-ineq} hold. 

Since $X_{1,2}$ is strongly regular with smallest eigenvalue $-2$, by Seidel's classification (see Theorem \ref{geom-eig-2}), $X_{1,2} = T(s)$ or $X_{1,2} = L_{2}(s)$ for some $s$.  Suppose that $X_{1,2}$ is $L_{2}(s)$, then $n = s^2$, $k_1+k_2 = 2(s-1)$, so $k_1\leq (s-1)$. At the same time, since $X_1$ has diameter $2$, degree $k_1$ should satisfy $k_1^{2}\geq n-1$, which gives us a contradiction. Therefore, $X_{1,2} = T(s)$. Consider 2 cases.

\noindent \textbf{Case 1.} Assume $p_{1,1}^{2}\geq \frac{k_1}{30}$ and $p_{2,2}^{1}\geq \frac{k_2}{30}$.

We can rewrite the assumptions of this case in the form $p_{1,2}^{1} = p_{1,1}^{2}\frac{k_2}{k_1}\geq \frac{k_2}{30}$ and $p_{1,2}^{2}\geq \frac{k_1}{30}$. We know that $X_{1,2} = T(s)$ for some $s$. Let $\mathcal{C}$ be a collection of its Delsarte cliques. Then every clique $C\in \mathcal{C}$ has size $1+p_{1,1}^{i}+p_{2,2}^{i}+2p_{1,2}^{i} = 1+\lambda_i(X_{1,2})= \frac{k_1+k_2}{2}+1\leq \frac{21}{20}k_1+1$ for any $i\in \{1, 2\}$. Every two distinct vertices $x, y\in C$ with $c(x,y) = i\in \{1, 2\}$ are distinguished by at least $|C| - p_{1,1}^{i} - p_{2,2}^{i} = 2p_{1,2}^{i}+1\geq \frac{1}{15}k_1+1\geq \frac{1}{16}|C|$ vertices in $C$.

\noindent \textbf{Case 2.} Assume $p_{i,i}^{j}<\frac{k_i}{30}$ for $\{i,j\} = \{1,2\}$.

\noindent Using Corollary \ref{cor3} and the inequality $k_1\leq k_2 \leq \frac{11}{10}k_1$, we get 
\[\frac{k_i+k_j}{2} = \lambda_i(X_{1,2}) = p_{i,i}^{i}+2p_{i,j}^{i}+p_{j,j}^{i}\leq p_{i,i}^{i} +2\frac{11}{10}\frac{k_i}{30}+\frac{(1+\varepsilon)}{2}k_j.\]
Thus, 
\[\frac{2}{5}k_i\leq \left(\frac{1}{2} - \frac{11}{150} - \frac{11\varepsilon}{20}\right)k_i \leq p_{i,i}^{i}.\]

Therefore, by Lemma \ref{assoc-clique-geom}, graph $X_i$ is strongly regular with smallest eigenvalue $-2$ or is the line graph of a regular triangle-free graph. 

Assume that for some $i\in \{1,2\}$ graph $X_i$ is strongly regular with smallest eigenvalue~$-2$, then $X_i$, as well as $X_{1,2}$, is either $T(s)$ or $L_{2}(s)$. Since $X_i$ and $X_{1,2}$ have the same number of vertices, the only possibility is $X_{1,2} = T(s_1)$ and $X_{i} = L_2(s_2)$. Then $\frac{s_1(s_1 - 1)}{2} = s_2^{2}$, so $s_2> \frac{1}{\sqrt{2}}(s_1-1)$. Which implies $k_i+k_j = 2(s_1 - 2) \leq 2\sqrt{2}(s_2 - 1)+1  = \sqrt{2}k_i+1$, so $k_j \leq (\sqrt{2} - 1)k_i+1$ and we get contradiction with $k_i\leq \frac{11}{10}k_j$ and $k_i\geq \sqrt{n-1}>3$.

\begin{remark}\label{remark-x1-x12}
Observe that argument in the last paragraph of the proof shows that $X_1$ and $X_{1,2}$ could not be simultaneously strongly regular with smallest eigenvalue $-2$ even if assumption that $k_2\leq \frac{11}{10}k_1$ does not hold (we assume that all other assumptions of the proposition are satisfied).
\end{remark}
 
\end{proof}

\subsubsection{Constituent $X_2$ is strongly regular}
Next we consider the case when $X_2$ is strongly regular, i.e., assume that we are in assumptions of statement 4 of Proposition \ref{assoc-diam2}.

\begin{proposition}\label{assoc-x2-strongly-reg}
Let $\mathfrak{X}$ be an association scheme  of rank 4 and diameter~2 on $n\geq 29$ vertices. Assume additionally, that the constituents of $\mathfrak{X}$ are ordered by degree and the inequality $ k_2\leq \frac{\varepsilon}{2} k_3$ for some $\varepsilon < 10^{-11}$ holds. Suppose that $k_2\leq \frac{101}{100}k_1$ and $X_2$ is strongly regular with smallest eigenvalue $-2$. Then 
 \[q(X_{1,2})+\xi(X_{1,2})\leq \frac{99}{100}(k_1+k_2).\]
\end{proposition}
\begin{proof} The assumptions of the proposition allow us to assume the inequalities from Lemma~\ref{assoc-param-ineq}.

Since $X_{2}$ is strongly regular with smallest eigenvalue $-2$ and $n\geq 29$, by Seidel's classification (Theorem \ref{geom-eig-2}) and Lemma~\ref{assoc-param-ineq}, we get that $\frac{k_2}{2}\geq p_{2,2}^{2} \geq \frac{k_2}{2} - 1$ and $p_{2,2}^{1} = p_{2,2}^{3} \leq \varepsilon k_2$.
By Proposition \ref{assoc-x12-approx} for $\varepsilon_1 = 25\varepsilon^{1/3}$, we have 
\begin{equation}\label{eq-x2-xi}
\xi(X_{1,2}) \leq \frac{p_{2,2}^{2}+p_{2,1}^{1}+p_{1,1}^{1}+\sqrt{(p_{2,1}^{1}+p_{1,1}^{1} - p_{2,2}^{2})^2+4p_{1,1}^{2}p_{2,1}^{1}}}{2}+\varepsilon_{1}(k_1+k_2), \, \text{so}
\end{equation}
\begin{equation}\label{eq-x2-xi-2}
\begin{multlined}
\xi(X_{1,2}) \leq \max(p_{2,2}^{2}, p_{2,1}^{1}+p_{1,1}^{1})+\sqrt{p_{1,1}^{2}p_{1,2}^{1}}+\varepsilon_{1}(k_1+k_2)\leq \\
\leq \max(p_{2,2}^{2}+p_{2,1}^{1}, 2p_{2,1}^{1}+p_{1,1}^{1})+\varepsilon_{1}(k_1+k_2)\leq \\
\leq \max(p_{2,2}^{2}+p_{2,1}^{1}, \lambda_1(X_{1,2}))+\varepsilon_{1}(k_1+k_2).
\end{multlined} 
\end{equation}
Recall also, that as in Eq. \eqref{eq-mux12-above} and Eq. \eqref{eq-lambda-above} we have
\begin{equation}\label{eq-x12-lambda-mu-resrt}
\mu(X_{1,2})\leq 2\varepsilon(k_1+k_2)\quad  \text{and} \quad \max(\lambda_1(X_{1, 2}), \lambda_2(X_{1,2}))\leq \frac{1+2\varepsilon}{2}(k_1+k_2).
\end{equation}

\noindent \textbf{Case 1.} Assume $p_{1,2}^{1}>\frac{2}{5} k_1$.

\noindent Then, using that $k_2\leq \frac{101}{100}k_1\leq \frac{11}{10}k_1$, 
\[\lambda_1(X_{1,2}) = p_{1,1}^{1}+2p_{1,2}^{1}+p_{2,2}^{1}\geq \frac{4}{5}k_1\geq \frac{4}{11}(k_1+k_2),\text{ and}\]  
\[\lambda_2(X_{1,2}) = p_{2,2}^{2}+2p_{1,2}^{2}+p_{1,1}^{2}\geq \frac{k_2}{2}+\frac{2}{5}\frac{10}{11} k_1\geq \frac{4}{11}(k_1+k_2).\] 
Since Eq. \eqref{eq-x12-lambda-mu-resrt} holds, graph $X_{1,2}$ satisfies the assumptions of Theorem \ref{Metsch} for $m=2$, so by Lemma \ref{assoc-clique-geom} it is strongly regular with smallest eigenvalue $-2$. However, by Proposition~\ref{assoc-x3-strongly-reg}, it is impossible that $X_{1,2}$ and $X_2$ are simultaneously strongly regular with smallest eigenvalue $-2$ under the assumptions of this proposition.

\noindent \textbf{Case 2.} Assume $\frac{1}{8}k_1\leq p_{1,2}^{1}\leq\frac{2}{5} k_1$.

\noindent \textbf{Case 2.a} Suppose $\lambda_2(X_{1,2}) \geq \lambda_1(X_{1,2})$.

\noindent Then as $k_2\leq \frac{11}{10} k_1$, inequality $q(X_{1,2}) \leq \lambda_2(X_{1,2}) \leq 2\varepsilon k_1 + \frac{k_2}{2}+\frac{2}{5}k_1 \leq \frac{49}{100}(k_1+k_2)$ holds. At the same time, by Eq. \eqref{eq-x2-xi-2}, we get 
\[\xi(X_{1,2}) \leq \max\left(\frac{1}{2}k_2+\frac{2}{5}k_1, \lambda_1(X_{1,2})\right)+\varepsilon_1(k_1+k_2)\leq \]
\[\leq \max\left(\frac{49}{100}(k_1+k_2), \lambda_2(X_{1,2})\right)+\varepsilon_1(k_1+k_2) \leq \frac{1}{2}(k_1+k_2),\] 
so we get $q(X_{1,2})+\xi(X_{1,2})\leq \frac{99}{100}(k_1+k_2)$.

\noindent \textbf{Case 2.b} Suppose $\lambda_1(X_{1,2}) \geq \lambda_2(X_{1,2})$.

\noindent We may assume that $\lambda_1(X_{1,2})\geq \frac{49}{100}(k_1+k_2)$. Otherwise, by Eq. \eqref{eq-x2-xi-2},
 \[q(X_{1,2})+\xi(X_{1,2})\leq \lambda_1(X_{1,2})+\max\left(\frac{1}{2}k_2+\frac{2}{5}k_1, \lambda_1(X_{1,2})\right)+\varepsilon_1(k_1+k_2)\leq \frac{99}{100}(k_1+k_2).\]
 
 Let $p_{1,2}^{1} = \alpha k_1$, then  $\frac{1}{8}\leq \alpha\leq \frac{2}{5}$. The condition $p_{1,1}^{1}+p_{2,2}^{1}+2p_{1,2}^{1} = \lambda_1(X_{1,2})\geq \frac{49}{100}(k_1+k_2)$ implies that 
 \begin{equation}\label{eq-x2-p111-bel}
  p_{1,1}^{1}+ p_{1,2}^{1} \geq \frac{49}{100}(k_1+k_2)-\varepsilon k_2 - \alpha k_1\geq \left(\frac{49}{50} - 2\varepsilon-\alpha\right) k_1\geq \frac{28}{50}k_1\geq \frac{28}{55}k_2>p_{2,2}^{2}.
 \end{equation}
On the other hand, Eq. \eqref{eq-x12-lambda-mu-resrt} implies
\begin{equation}\label{eq-x2-p111-ab}
 p_{1,1}^{1}+p_{1,2}^{1} \leq \lambda_1(X_{1,2}) - p_{1,2}^{1}\leq  \left(\frac{1}{2}+\varepsilon\right)(k_1+k_2) - \alpha k_1\leq \left(\frac{101}{100} - \alpha \right)k_1 . 
\end{equation}
Hence, as $p_{2,2}^{2}\geq (\frac{1}{2}-\varepsilon)k_2\geq (\frac{1}{2}-\varepsilon)k_1$,  Eq. \eqref{eq-x2-p111-bel} and \eqref{eq-x2-p111-ab} imply that
\[|p_{1,1}^{1} +p_{1,2}^{1}-p_{2,2}^{2}|\leq \left(\frac{101}{100} - \alpha \right)k_1 - \left(\frac{1}{2}-\varepsilon\right)k_1 \leq \left(\frac{52}{100} - \alpha\right)k_1.\]
Therefore, using Eq. \eqref{eq-x2-xi}, Eq. \eqref{eq-x2-p111-ab} and $p_{2,2}^{2}\leq \frac{k_2}{2}\leq \frac{101}{200}k_1$, we get $\text{for } \frac{1}{8}\leq \alpha \leq \frac{2}{5}$
\begin{equation}
\xi(X_{1,2}) \leq \frac{\frac{101}{200}+(\frac{101}{100} - \alpha)+\sqrt{(\frac{52}{100} - \alpha)^{2} +4\alpha^{2}}}{2}k_1+\frac{201}{100}\varepsilon_1 k_1\leq \frac{195}{200}k_1, 
\end{equation}
Thus, $q(X_{1,2})+\xi(X_{1,2})\leq \frac{1+2\varepsilon}{2}(k_1+k_2)+\frac{195}{200}k_1\leq \frac{99}{100}(k_1+k_2)$.

\noindent \textbf{Case 3.} Assume $p_{1,2}^{1}< \frac{1}{8}k_1$.

\noindent Then, using Eq. \eqref{eq-x2-xi-2}, Corollary \ref{cor3} and inequality $k_2\leq \frac{101}{100}k_1$,  
\[\xi(X_{1,2})\leq \max(p_{2,2}^{2}+p_{2,1}^{1}, 2p_{2,1}^{1}+p_{1,1}^{1})+\varepsilon_{1}(k_1+k_2)\leq \]
\[ \leq \max\left(\frac{1}{2}k_2+\frac{1}{8}k_1, \frac{1}{4}k_1+\frac{1+\varepsilon}{2}k_1 \right)+\varepsilon_1(k_1+k_2) \leq \frac{2}{5}(k_1+k_2).\]
Combining this with Eq. \eqref{eq-x12-lambda-mu-resrt} we get
$q(X_{1,2})+\xi(X_{1,2})\leq \frac{99}{100}(k_1+k_2)$.

\end{proof}

\subsubsection{Constituent $X_1$ is strongly regular}\label{sec-x1}

The common strategy of our proofs is to get good spectral gap for some union of the constituents, or to apply Metsch's criteria (Theorem \ref{Metsch}) to this graph. The next lemma covers the range of the parameters for which spectral gap is hard to achieve, and the conditions of Metsch's criteria are not always satisfied for $X_{2}$ and $X_{1,2}$. However, we are still able to use the idea of Metsch's proof to show that $k_2$ does not differ much from $k_1$, that will suffice for our purposes.  
\begin{definition}
For a homogeneous configuration $\mathfrak{X}$ and disjoint non-empty sets of edge colors $I$ and $J$ we say that vertices $x, y_1, y_2, ..., y_t$ form a \textit{$t$-claw} (claw of size $t$) in colors $(I, J)$ if $c(x,y_i)\in I$ and $c(y_i, y_j)\in J$ for all distinct $1\leq i,  j\leq t$.  
\end{definition}
\begin{lemma}\label{assoc-claw-proof-bound}
Let $\mathfrak{X}$ be an association scheme of rank 4 and diameter~2 with constituents ordered by degree. Assume additionally, that the inequality $k_2\leq \frac{\varepsilon}{2} k_3$ holds for some $0< \varepsilon \leq \frac{1}{100}$. Suppose for some $0<\delta\leq \frac{1}{100}$ we have
\[p_{2,2}^{2}\geq \frac{1 - \delta}{2}k_2\quad  \text{and} \quad \frac{1}{8}k_2 \leq p_{2,2}^{1} \leq \frac{1}{3} k_2.\]
Then $k_2\leq 20 k_1$.
\end{lemma}
\begin{proof}
Note that the assumptions of the lemma force the inequalities from Lemma~\ref{assoc-param-ineq} to be true.

First we show that under the assumptions of the lemma there are no 3-claws in colors $(2,3)$ in $\mathfrak{X}$. That is, for $x, y_1, y_2, y_3\in \mathfrak{X}$ it is not possible that $c(x,y_i) = 2$ and $c(y_i,y_j) = 3$ for all distinct $i,j\in [3]$. Indeed, suppose such $x, y_i$ exist. Let $U_i = N_2(x)\cap N_2(y_i)$. Then
\[
|U_{i}| = p_{2,2}^{2}\geq \frac{1 - \delta}{2}k_2,\quad  |U_{i}\cap U_{j}| \leq |N_2(y_i)\cap N_2(y_j)|= p_{2,2}^{3}\leq \varepsilon k_2 \quad \text{ and } \]
\[ 
|U_1 \cup U_2 \cup U_3|\leq |N_2(x)| = k_2.
\]
Therefore, we should have $k_2 \geq 3\frac{(1 - \delta)}{2}k_2-3\varepsilon k_2$, a contradiction.
Hence the size of a maximal claw in colors $(2,3)$ is 2. 

We claim now that any edge of color 2 lies inside a clique of size at least $p_{2,2}^{2} - p_{2,2}^{3} - p_{2,1}^{3}$ in $X_{1,2}$. Consider any edge $\{x,y\}$ of color 2. Let $z$ be a vertex which satisfies $c(x,z) = 2$ and $c(y,z) = 3$. Define $$C(x,y) = \{x, y\}\cup\{ w: c(z,w) = 3\, \text{ and } \, c(x,w) = 2,\, c(y,w) = 2\}.$$ 
Observe that $|C(x,y)|\geq 2+ p_{2,2}^{2} - p_{2,2}^{3} - p_{2,1}^{3}$. At the same time, if $z_1, z_2 \in C(x,y)$ satisfy $c(z_1, z_2) = 3$, then $x, z, z_1, z_2$ form a 3-claw in colors $(2,3)$, which contradicts our claim above. Hence, $C(x,y)$ is a clique in $X_{1,2}$.

Assume that there is an edge $\{y_1, y_2\}$ in $C(x,y)$ of color $1$ for some $x,y$. Then \[2k_1+\frac{1}{3}k_2\geq 2\sum\limits_{i = 1}^{3} p_{1,i}^{1} +p_{2,2}^{1}\geq p_{1,1}^{1}+2p_{1,2}^{1}+p_{2,2}^{1}\geq |C(x,y)|-2\geq \frac{1 - \delta - 2\varepsilon}{2}k_2 - k_1.\]
Therefore, $k_2\leq 20 k_1$.

Assume now that all the edges in $C(x,y)$ are of color $2$ for all $x,y$, that is, $C(x,y)$ is a clique in $X_{2}$. Let $\mathcal{C}$ be the set of all maximal cliques in $X_2$ of size at least $p_{2,2}^{2} - p_{2,2}^{3} - p_{2,1}^{3}$. Then we have proved that every edge of color $2$ is covered by at least one clique in $\mathcal{C}$. Consider, two distinct cliques $C_1, C_2 \in \mathcal{C}$. Then there is a pair of vertices $v\in C_1\setminus C_2$ and $u\in C_2\setminus C_1$ with $c(v, u) \neq 2$. Thus $|C_1\cap C_2|\leq \max(p_{2,2}^{1}, p_{2,2}^{3})\leq \frac{1}{3}k_2$.

Suppose first that some pair of distinct cliques $C_1, C_2\in \mathcal{C}$ satisfies $|C_1\cap C_2|\geq 2$ and let $\{x,y\}\in C_1\cap C_2$ be an edge. Then $c(x,y) = 2$ and every vertex in $C_1\cup C_2$ is connected to both $x$ and $y$ by edge of color $2$. Thus, 
\[p_{2,2}^{2}\geq |C_1\cup C_2| - 2 \geq 2(p_{2,2}^{2} - p_{2,2}^{3} - k_1) - \frac{1}{3}k_2,\, \text{ so}\]  
\[\left(\frac{1}{3}+2\varepsilon\right)k_2+2k_1\geq \frac{1}{3}k_2+2\varepsilon k_2 +2k_1 \geq p_{2,2}^{2}\geq \frac{1 - \delta}{2}k_2. \]
Hence, $k_2\leq 20k_1$.

Finally, if for any two distinct cliques $C_1, C_2\in \mathcal{C}$ we have $|C_1\cap C_2|\leq 1$, then every edge of color $2$ lies in at most one clique of $\mathcal{C}$. Above we proved that any edge of color $2$ lies in at least one clique of $\mathcal{C}$, so it lies in exactly one. Therefore, as $p_{2,2}^{2}\geq \frac{1 - \delta}{2}k_2$, we get that either $k_2\leq 20k_1$, or $|C|> \frac{1}{3} k_2$ for any $C\in \mathcal{C}$ and so every vertex lies in at most $2$ cliques from~$\mathcal{C}$. In the latter case, by Lemma \ref{clique-mu-bound}, we get that $p_{2,2}^{1}\leq 4$, which contradicts $p_{2,2}^{1}\geq \frac{1}{8}k_2$, as by Lemma \ref{assoc-param-ineq} we have $k_2\geq \frac{p_{2,2}^{3}}{\varepsilon}\geq \frac{1}{\varepsilon}$.
\end{proof}

Alternatively, we can get a linear inequality between $k_1$ and $k_2$ if we know that $X_{1,2}$ has clique geometry.
\begin{lemma}\label{assoc-k1-k2-bound}
Let $\mathfrak{X}$ be an association scheme of rank 4 on $n\geq 29$ vertices, with diameter~2, constituents ordered by degree and assume inequality $k_2\leq \frac{\varepsilon}{2}k_3$ holds for some $\varepsilon<\frac{1}{10}$. Assume $X_{1,2}$ has a clique geometry such that every vertex belongs to at most $m$ cliques. Then
\[ p_{2,3}^{1} \leq \frac{m^2 -2}{2}k_1\quad \text{ and } \quad k_2 \leq \frac{3}{2-4\varepsilon}(m^{2} - 2) k_1. \]
If additionally,  $X_1$ is strongly regular graph with smallest eigenvalue $-2$, then
\[ p_{2,3}^{1} \leq \frac{m^2 -4}{8}k_1\quad \text{ and } \quad k_2 \leq \frac{3}{8(1-2\varepsilon)}(m^{2} - 4) k_1. \]
\end{lemma}
\begin{proof}
By Lemma \ref{clique-mu-bound} applied to $X_{1,2}$, we know
\begin{equation}
p_{1,1}^{3}+2p_{1,2}^{3}+p_{2,2}^{3} = \mu(X_{1,2}) \leq m^2.
\end{equation}
Since $\mathfrak{X}$ is of diameter $2$, we get $p_{1,1}^{3}$ and $p_{2,2}^{3}$ are at least 1, and $k_1(k_1 -1)\geq k_3$. Thus
\begin{equation}\label{eq-k1-k2-p123}
p_{1,2}^{3}\leq \frac{m^{2} - 2}{2}, \, \text{ so }\, p_{2,3}^{1}\leq \frac{m^2 - 2}{2}\frac{k_3}{k_1} \leq \frac{m^2 - 2}{2}k_1.
\end{equation} 
By Eq.~\eqref{eq-sum-param}, $p_{2,1}^{1}+p_{2,2}^{1}+p_{2,3}^{1} = k_2$, and Corollary \ref{cor2} implies that
$ p_{2,3}^{1}+\varepsilon k_2 \geq \max(p_{2,2}^{1}, p_{2,1}^{1})$. Thus, combining with Eq \eqref{eq-k1-k2-p123}, we get
 $$ \frac{1-2\varepsilon}{3}k_2 \leq p_{2,3}^{1} \leq \frac{m^2 - 2}{2}k_1.
 $$
If $X_1$ is strongly regular with smallest eigenvalue $-2$, we can get better estimates. By Seidel's classification, $X_1$ is either $T(s)$ or $L_{2}(s)$ for some $s$. Thus, either $n = \frac{s(s-1)}{2}$ and $k_1 = 2(s-2)$, or $n = s^2$ and $k_1 = 2(s-1)$. In any case $4k_3\leq k_1^{2}$.  Observe that Corollary~\ref{cor2} implies 
$ p_{2,3}^{2}+\varepsilon k_2 \geq \max(p_{2,2}^{2}, p_{2,1}^{2}) $. Hence, $p_{2,3}^{2} \geq \frac{1-2\varepsilon}{3}k_2$, so $p_{2,2}^{3}\geq \frac{(1-2\varepsilon)(k_2)^2}{3k_3}\geq \frac{4(1-2\varepsilon)}{3}$. At the same time $p_{1,1}^{3}\geq 2$ as $\mu(X_1)\geq 2$ for $X_1 = T(s)$, or $X_1 = L_2(s)$. Thus $p_{i,i}^{3}\geq 2$ for $i=1$ and $i=2$. Therefore, as in Eq. \eqref{eq-k1-k2-p123},
\[ p_{2,3}^{1}\leq \frac{m^2 - 4}{2}\frac{k_3}{k_1} \leq \frac{m^2 - 4}{8}k_1. \]
Again, $p_{2,3}^{1}\geq  \frac{1-2\varepsilon}{3}k_2$ implies the inequality on $k_2$. 
\end{proof} 

Now, we are ready to consider the case when the constituent $X_1$ is strongly regular (case of statement 3 of Proposition \ref{assoc-diam2}). 

\begin{proposition}\label{assoc-x1-strongly-reg}
Let $\mathfrak{X}$ be an association scheme of rank 4 on $n$ vertices  with diameter~2 and constituents ordered by degree. Assume additionally, that the parameters of $\mathfrak{X}$ satisfy $k_2\leq \frac{\varepsilon}{2}k_3$ for $\varepsilon = 10^{-26}$. Suppose that $X_1$ is a strongly regular graph with smallest eigenvalue $-2$. Then
\begin{equation}\label{eq-x1-goal}
 q(Y)+\xi(Y)\leq (1-\varepsilon)k_Y,
 \end{equation}
where either $Y = X_{2}$ and $k_Y = k_2$, or $Y = X_{1,2}$ and $k_Y = k_1+k_2$.

\end{proposition}
\begin{proof}
Note that the assumptions of the proposition force the inequalities from Lemma~\ref{assoc-param-ineq} to be true.

Since $p_{1,1}^2 =p_{1,1}^3\leq \varepsilon k_1$, by Lemma \ref{assoc-x12-approx}, for $\varepsilon_1 = 25\varepsilon^{1/3}\leq \frac{2}{3}10^{-7}$ we have 
\begin{equation}\label{eq-x12-spectral-gap-in-x1}
\xi(X_{1,2}) \leq \frac{p_{1,1}^{1}+p_{1,2}^{2}+p_{2,2}^{2}+\sqrt{(p_{1,2}^{2}+p_{2,2}^{2}-p_{1,1}^{1})^{2}+4p_{2,2}^{1}p_{1,2}^{2}}}{2}+\varepsilon_1(k_1+k_2).
\end{equation} 
Since $\sqrt{x}$ is a concave function this implies
\begin{equation}
\xi(X_{1,2})\leq \max(p_{1,1}^{1}, p_{2,2}^{2}+p_{2,1}^{2})+\sqrt{p_{2,2}^{1}p_{2,1}^{2}}+\varepsilon_1 (k_1+k_2).
\end{equation}
Using that $\lambda_{2}(X_{1,2}) \geq p_{2,2}^{2}+p_{2,1}^{2}$ and $p_{2,1}^{2} = \frac{k_1}{k_2}p_{2,2}^{1}\leq p_{2,2}^{1}$ we can simplify it even more
\begin{equation}
\xi(X_{1,2})\leq \max\left(\frac{k_1}{2}, \lambda_2(X_{1,2})\right)+p_{2,2}^{1}+\varepsilon_1 (k_1+k_2).
\end{equation}
Recall that as in Eq. \eqref{eq-mux12-above} and Eq. \eqref{eq-lambda-above} we have
\begin{equation}\label{eq-x12-mu-lambda-resrt-in-x1}
\mu(X_{1,2})\leq 2\varepsilon(k_1+k_2)\quad  \text{and} \quad \max(\lambda_1(X_{1, 2}), \lambda_2(X_{1,2}))\leq \frac{1+2\varepsilon}{2}(k_1+k_2).
\end{equation}

\noindent \textbf{Case A.} Suppose $p_{2,2}^{2}\geq (2-2\delta) p_{2,2}^{1}$ for $\delta = 10^{-7}$.

\noindent Using Corollary \ref{cor3} for $p_{2,2}^{2}$ we get 
\begin{equation}\label{eq-p212-u}
p_{2,2}^{1}\leq \frac{1+\varepsilon}{4(1-\delta)}k_2 \quad \text{and} \quad  p_{2,1}^{2} =\frac{k_1}{k_2} p_{2,2}^{1} \leq \frac{1+\varepsilon}{4(1-\delta)}k_1.
\end{equation}
Note, by Corollary \ref{cor2}, $\varepsilon k_1+p_{1,3}^1\geq p_{1,1}^1\geq \left(\frac{1}{2}-\varepsilon\right)k_1$, so Eq.~\eqref{eq-sum-param} implies $p_{1,2}^{1}\leq 3\varepsilon k_1$. Therefore, 
\[\lambda_{1}(X_{1,2})\leq \frac{1}{2}k_1+p_{2,2}^1+12\varepsilon k_1.\]
Assume that Eq. \eqref{eq-x1-goal} is not satisfied, then
\begin{equation}\label{eq-x1-caseA}
\begin{multlined}
 (1-\varepsilon)(k_1+k_2)\leq q(X_{1,2})+\xi(X_{1,2})\leq \\
 \leq \max(\lambda_2(X_{1,2}), \frac{1}{2}k_1+p_{2,2}^{1}+6\varepsilon k_1)+\max(\lambda_2(X_{1,2}), \frac{k_1}{2})+p_{2,2}^{1}+\varepsilon_1(k_1+k_2)
 \end{multlined}
 \end{equation}
Observe, that if $\lambda_2(X_{1,2})\leq \frac{1}{2}k_1+p_{2,2}^{1}+6\varepsilon k_1$, then using Eq. \eqref{eq-p212-u} we get a contradiction
$$(1-\varepsilon)(k_1+k_2)\leq k_1+3\frac{1+\varepsilon}{4(1-\delta)}k_2+\varepsilon_1(k_1+k_2)+6\varepsilon k_1.$$ 
Otherwise, Eq. \eqref{eq-x1-caseA} implies
 $$(1-\varepsilon-\varepsilon_1)(k_1+k_2)\leq 2\lambda_2(X_{1,2})+\frac{1+\varepsilon}{4(1-\delta)}k_2, \text{ so }\quad \lambda_2(X_{1,2}) \geq \frac{5}{6}k_1.$$
We estimate the expression under the root sign, using $(\frac{1}{2}-\varepsilon)k_1 \leq p_{1,1}^{1} \leq \frac{k_1}{2}$, $\lambda_2(X_{1,2})\geq \frac{5}{6}k_1$, Eq. \eqref{eq-p212-u}, and inequality $\frac{1+\varepsilon}{4(1-\delta)^2}\leq \frac{1+\varepsilon}{4}+\delta$ for $0\leq\varepsilon,  \delta \leq \frac{1}{2}$.

\[(p_{1,2}^{2}+p_{2,2}^{2}-p_{1,1}^{1})^{2}+4p_{2,2}^{1}p_{1,2}^{2} = (p_{1,2}^{2}+p_{2,2}^{2})^{2} - 2p_{1,1}^{1}(p_{1,2}^{2}+p_{2,2}^{2})+(p_{1,1}^{1})^{2}+4p_{2,2}^{1}p_{1,2}^{2} \leq \]
\[\leq (p_{1,2}^{2}+p_{2,2}^{2})^{2} -2p_{1,1}^{1}p_{1,2}^{2} - (1 - 2\varepsilon)k_1 p_{2,2}^{2} +\frac{(k_1)^{2}}{4} +\frac{1+\varepsilon}{2(1-\delta)^{2}}k_1 p_{2,2}^{2} \leq \]
\[ \leq  (p_{1,2}^{2}+p_{2,2}^{2})^{2} -k_1p_{1,2}^{2} +\left(\frac{(k_1)^2}{4} - \frac{1}{2}k_1p_{2,2}^{2}\right)+(3\varepsilon+2\delta)k_1k_2 \leq \]
\[ \leq  (p_{1,2}^{2}+p_{2,2}^{2})^{2} -\frac{k_1}{2}\left(\lambda_2(X_{1,2}) - \varepsilon k_1 -\frac{k_1}{2}\right)+(3\varepsilon+2\delta)k_1k_2 \leq\]
\[\leq (p_{1,2}^{2}+p_{2,2}^{2})^{2} - \frac{1}{6}(k_1)^{2} + (4\varepsilon+2\delta)k_1k_2.\]
Thus, using that $\sqrt{x}$ is concave, inequality $\sqrt{y^{2} - x^{2}} \leq y-\frac{x^{2}}{2y}$ and Eq. \eqref{eq-x12-mu-lambda-resrt-in-x1}, 
\begin{equation}\label{eq-x12-root-est}
\sqrt{(p_{1,2}^{2}+p_{2,2}^{2}-p_{1,1}^{1})^{2}+4p_{2,2}^{1}p_{1,2}^{2}} \leq p_{1,2}^{2}+p_{2,2}^{2}- \frac{2(k_1)^{2}}{13(k_1+k_2)}+\sqrt{(4\varepsilon+2\delta)k_1k_2}.
\end{equation}
Denote $\varepsilon_4 = \sqrt{4\varepsilon+2\delta}<2^{-1}\cdot 10^{-3}$. Hence, by Eq. \eqref{eq-x12-spectral-gap-in-x1} and Eq. \eqref{eq-x12-root-est},  
\[
\xi(X_{1,2}) \leq \frac{k_1}{4}+(p_{1,2}^{2}+p_{2,2}^{2}) - \frac{(k_1)^{2}}{13(k_1+k_2)}+\left(\frac{1}{2}\varepsilon_4\sqrt{k_1k_2}+\varepsilon_1 (k_1+k_2)\right).
\]
Using Corollary \ref{cor3} for $p_{2,2}^{2}$ and Eq. \eqref{eq-p212-u}, we get

\begin{equation}\label{eq-xi-bound-k1-k2}
\begin{multlined}
\xi(X_{1,2}) \leq \frac{k_1}{4}+\left(\frac{(1+\varepsilon)}{4(1-\delta)}k_1+\frac{(1+\varepsilon)}{2}k_2\right) - \frac{(k_1)^{2}}{13(k_1+k_2)}+\left(\frac{\varepsilon_4}{4}+\varepsilon_1\right) (k_1+k_2)\leq \\
\leq \frac{1}{2}(k_1+k_2) - \frac{(k_1)^{2}}{13(k_1+k_2)^{2}}(k_1+k_2)+ \varepsilon_5(k_1+k_2),
\end{multlined}
\end{equation}
where $\varepsilon_5 = \varepsilon_1+\frac{1}{4}\varepsilon_4+\delta+\varepsilon<6^{-1}\cdot 10^{-3}$. Thus, we want either $q(X_{1,2})$ to be bounded away from $\frac{k_1+k_2}{2}$, or to have $k_1\leq c k_2$ for some absolute constant $c$.

Observe, by Eq. \eqref{eq-p212-u} and Eq. \eqref{eq-xi-bound-k1-k2},
\begin{equation}\label{eq-xi-lambda-1-k1-k2}
\begin{multlined}
\lambda_1 (X_{1,2}) +\xi(X_{1,2}) \leq \left(\frac{k_1}{2}+p_{2,2}^{1}+2\varepsilon k_1\right) + \frac{k_1+k_2}{2}+\varepsilon_5(k_1+k_2)\leq \\
\leq k_1+\frac{3}{4}k_2+(3\varepsilon+\delta)(k_1+k_2)+\varepsilon_5k_2 
\leq (1-\varepsilon)(k_1+k_2),
\end{multlined}
\end{equation}
\begin{equation}
\lambda_{2}(X_{1,2})+\xi(X_{1,2}) \leq \lambda_{2}(X_{1,2})+\frac{k_1+k_2}{2} +\varepsilon_5(k_1+k_2).
\end{equation}
Clearly, 
$$\mu(X_{1,2})+\xi(X_{1,2}) \leq 2\varepsilon(k_1+k_2)+\xi(X_{1,2})\leq (1-\varepsilon)(k_1+k_2).$$
Thus, either we have $q(X_{1,2})+\xi(X_{1,2})\leq (1-\varepsilon)(k_1+k_2)$, or $\lambda_2(X_{1,2})\geq (\frac{1}{2}-\varepsilon_5 - \varepsilon)(k_1+k_2)$.

\noindent Suppose that $\lambda_2(X_{1,2})\geq (\frac{1}{2}-\varepsilon_5 - \varepsilon)(k_1+k_2)$. Recall, by assumption of Case A, 
$$\lambda_2(X_{1,2})\leq p_{2,2}^{2}+\frac{1}{1-\delta}\frac{k_1}{k_2}p_{2,2}^{2}+\varepsilon k_1\leq \left(\frac{1}{1-\delta}\frac{p_{2,2}^{2}}{k_2}+\varepsilon\right)(k_1+k_2).$$
Hence, in this case \[p_{2,2}^{2}\geq \left(\frac{1}{2} - \varepsilon_5 - 2\varepsilon\right)(1-\delta)k_2 .\]

\begin{enumerate}
\item[]\textbf{Case A.1} Assume $p_{2,2}^{1}< \frac{1}{8}k_2$.

\noindent Then $\mu(X_{2})\leq \frac{1}{8}k_2$ and $\lambda(X_2) = p_{2,2}^{2}$. Therefore, $X_2$ satisfies assumptions of Theorem~\ref{Metsch} for $m=2$. Thus, by Lemma \ref{clique-mu-bound} for graph $X_2$ we get $p_{2,2}^{3}\leq m^2 = 4$. At the same time, by Crorllary \ref{cor2} and Eq. \eqref{eq-sum-param} we have $p_{2,3}^{2}\geq \frac{1 - 2\varepsilon}{3}k_2$. Therefore, \[4k_3\geq p_{2, 2}^{3}k_3 = p_{2,3}^{2}k_2\geq k_2\frac{1 - 2\varepsilon}{3}k_2\geq \frac{1}{4}(k_2)^{2}.\]
At the same time, $(k_1)^{2}\geq k_3$, so $k_2\leq 4k_1$.
\item[]\textbf{Case A.2} Assume $p_{2,2}^{1}\geq \frac{1}{8}k_2$.

\noindent Then, as Eq.~\eqref{eq-p212-u} holds, by Lemma~\ref{assoc-claw-proof-bound},   we get that $k_2\leq 20k_1$. 
\end{enumerate}

\noindent Hence, Eq. \eqref{eq-xi-bound-k1-k2} and Eq. \eqref{eq-x12-mu-lambda-resrt-in-x1} imply
\[\lambda_{2}(X_{1,2})+\xi(X_{1,2})\leq (1-\varepsilon)(k_1+k_2).\]
Therefore, using bound on $\mu(X_{1,2})$ and Eq. \eqref{eq-xi-lambda-1-k1-k2}, we get $q(X_{1,2})+\xi(X_{1,2})\leq (1-\varepsilon)(k_1+k_2)$.

\noindent \textbf{Case B.} Suppose $p_{2,2}^{2}\leq (2-2\delta) p_{2,2}^{1}$.

\noindent Recall that by Corollary \ref{cor3}, $2p_{2,1}^{2}\leq (1+\varepsilon)k_1$. Assume $p_{2,2}^{1}\geq \frac{1}{5}k_2$ and $m\leq 5$, then
\begin{equation}\label{eq-x1-l1bound}
2\lambda_1(X_{1,2})-\lambda_{2}(X_{1,2})\geq k_1 +2p_{2,2}^{1} - p_{2,2}^{2} - 2p_{21}^{2} - 3\varepsilon k_1 \geq 
2\delta p_{2,2}^{1} - 4\varepsilon k_1\geq (2m-1)\mu(X_{1,2}).
\end{equation}
Suppose that $\lambda_2(X_{1,2})\geq (\frac{1}{4}+2m\varepsilon)(k_1+k_2)$, then Eq. \eqref{eq-x12-mu-lambda-resrt-in-x1} implies
\begin{equation}\label{eq-x1-l2bound}
2\lambda_2(X_{1,2}) - \lambda_1(X_{1,2})\geq (2m-1)\mu(X_{1,2}).
\end{equation}

\noindent \textbf{Case B.1.} Assume $p_{2,2}^{1}\geq \frac{k_2}{3}$.

\noindent Then
\begin{equation}\label{eq-x1-l1-13}
\lambda_1(X_{1,2})\geq \left(\frac{1}{2} - \varepsilon\right)k_1+\frac{1}{3}k_2\geq \frac{1}{3}(k_1+k_2).
\end{equation}
\begin{enumerate}
\item[]\textbf{Case B.1.a.} Suppose $p_{2,2}^{2}\geq \frac{1}{3}k_2$.

Then $\lambda_2(X_{1,2})\geq \frac{1}{3}(k_1+k_2)$. Thus, in notations of Theorem \ref{Metsch} we get for $X_{1,2}$ that $4\lambda^{(1)} - 6\mu(X_{1,2})\geq k_1+k_2$, and by Eq. \eqref{eq-x1-l1bound}-\eqref{eq-x1-l2bound}, inequality
$2\lambda^{(1)} - \lambda^{(2)}\geq 5\mu(X_{1,2})$ holds.
 Hence, by Theorem \ref{Metsch}, graph $X_{1,2}$ has clique geometry with $m=3$. Thus, by Lemma \ref{assoc-k1-k2-bound}, we have  $k_2 \leq \frac{15}{8(1-2\varepsilon)}k_1\leq 2k_1$. Therefore,
\[\lambda_{1}(X_{1,2})\geq p_{1,1}^{1}+p_{2,2}^{1}\geq \left(\frac{1}{2} - \varepsilon\right)k_1+\frac{1}{3}k_2 >\left(\frac{1}{3}+4\varepsilon\right)(k_1+k_2),\]
\[\lambda_{2}(X_{1,2})\geq p_{2,2}^{2}+2p_{1,2}^{2}\geq \frac{k_2}{4}+\frac{k_1}{3}+\frac{1}{3}\frac{k_2}{2}> \left(\frac{1}{3}+4\varepsilon\right)(k_1+k_2).\]

Therefore, $X_{1,2}$ satisfies Theorem \ref{Metsch} for $m = 2$, and so by Lemma \ref{assoc-clique-geom}, it is strongly regular with smallest eigenvalue $-2$. However,  by Proposition \ref{assoc-x3-strongly-reg} and Remark~\ref{remark-x1-x12}, graphs $X_1$ and $X_{1,2}$ could not be simultaneously strongly regular with smallest eigenvalue $-2$ under the assumptions of this proposition.

\item[]\textbf{Case B.1.b.} Suppose $p_{2,2}^{2}< \frac{1}{3}k_2$.

Then, in particular, $p^{1}_{2,2}\geq p_{2,2}^{2}$, so $q(X_2) = p_{2,2}^{1}$. Take $0\leq \alpha  \leq \frac{1+\varepsilon}{2}\leq \frac{51}{100}$, $0\leq \gamma \leq 1$, so that $p_{2,2}^{1} = \alpha k_2$ and $k_1 = \gamma k_2$. Using Eq. \eqref{x2-eq-bound} compute
\begin{equation}
q(X_2)+\xi(X_2) \leq p_{2,2}^{1}+p_{2,2}^{2}+\varepsilon k_2+ \sqrt{p_{2,2}^{1}p_{1,2}^{2}} +\varepsilon_1 k_2 = p_{2,2}^{2}+\alpha(1+\sqrt{\gamma}+\varepsilon+\varepsilon_1)k_2
\end{equation}

If $p_{2,2}^{2} \leq (1-\alpha(1+\sqrt{\gamma})-\varepsilon_1 - 2\varepsilon) k_2$, then $q(X_2)+\xi(X_2)\leq (1-\varepsilon)k_2$ and we reach our goal. So, assume that $p_{2,2}^{2} \geq (1-\alpha(1+\sqrt{\gamma})-\varepsilon_1 - 2\varepsilon) k_2$. Compute
\begin{equation}
\begin{multlined}
\lambda_2(X_{1,2}) = p_{2,2}^{2}+2p_{1,2}^{2}+p_{1,1}^{2} \geq (1-\alpha(1+\sqrt{\gamma})-\varepsilon_1 - 2\varepsilon) k_2+2\alpha\gamma k_2 \geq \\
\geq \left(\frac{1 - \alpha(1+\sqrt{\gamma}-2\gamma) - \varepsilon_1 - 2\varepsilon}{1+\gamma}\right)(k_1+k_2) \geq \frac{3}{10}(k_1+k_2),
\end{multlined}
\end{equation}
here we use that $1+\sqrt{\gamma} - 2\gamma\geq 0$ for $0\leq \gamma \leq 1$, so expression is minimized for $\alpha = \frac{1+\varepsilon}{2}$ and after that we compute the minimum of the expression for $0\leq\gamma \leq 1$. Thus, by Eq. \eqref{eq-x1-l1bound}-\eqref{eq-x1-l1-13}, graph $X_{1,2}$ has clique geometry for $m = 3$. Hence, by Lemma~\ref{assoc-k1-k2-bound} we have $k_2\leq 2k_1$. This implies that $\frac{1}{2}\leq \gamma \leq 1$. Compute,
\[\min_{1/2\leq \gamma\leq 1}\min_{0\leq\alpha\leq \frac{51}{100}}\left(\frac{1 - \alpha(1+\sqrt{\gamma}-2\gamma) - \varepsilon_1 - 2\varepsilon}{1+\gamma}\right)= \]
\[= \min_{1/2\leq \gamma\leq 1}\left(\frac{1 - \frac{51}{100}(1+\sqrt{\gamma}-2\gamma) - \varepsilon_1 - 2\varepsilon}{1+\gamma}\right)\geq \frac{9}{25}>\frac{1}{3}+2\varepsilon.\]

Therefore, using also Eq.~\eqref{eq-x1-l1bound}-\eqref{eq-x1-l1-13}, we get that $X_{1,2}$ satisfies conditions of Theorem~\ref{Metsch} for $m=2$, so by Lemma \ref{assoc-clique-geom}, graph $X_{1,2}$ is strongly regular with smallest eigenvalue $-2$. However, by Proposition \ref{assoc-x3-strongly-reg} and Remark~\ref{remark-x1-x12}, it is impossible, as $X_1$ is also strongly regular with smallest eigenvalue $-2$.
\end{enumerate}

\noindent \textbf{Case B.2.} Assume $\frac{k_2}{5}\leq p_{2,2}^{1}\leq \frac{k_2}{3} $.

\noindent Then 
\begin{equation}\label{eq-x1-l1-b2}
\lambda_{1}(X_{1,2})\geq p_{1,1}^{1}+p_{2,2}^{1}\geq \left(\frac{1}{2}-\varepsilon\right)k_1 +\frac{1}{5}k_2.
\end{equation}
If $p_{2,2}^{2}\leq (\frac{1}{3} - \varepsilon -\varepsilon_1) k_2$, then
\[q(X_2)+\xi(X_2)\leq \max(p_{2,2}^{2}, p_{2,2}^{1})+p_{2,2}^{2}+\sqrt{p_{2,2}^{1}p_{1,2}^{2}}+\varepsilon_1 k_2\leq \]
\[ \leq \frac{k_2}{3}+(\frac{1}{3} - \varepsilon -\varepsilon_1) k_2 + \frac{k_2}{3} +\varepsilon_1 k_2 \leq (1-\varepsilon)k_2.\]
Else $p_{2,2}^{2}\geq (\frac{1}{3} - \varepsilon -\varepsilon_1) k_2\geq (\frac{1}{4}+10\varepsilon) k_2$, so
\begin{equation}\label{eq-x1-l2-b2}
\lambda_2(X_{1,2})\geq p_{2,2}^{2}+2p_{1,2}^{2}\geq \left(\frac{1}{4}+10\varepsilon\right)k_2+\frac{2}{5}k_1.
\end{equation}
Thus, Eq. \eqref{eq-x1-l1bound} - \eqref{eq-x1-l2bound} and Eq. \eqref{eq-x1-l1-b2}-\eqref{eq-x1-l2-b2} imply, using Theorem \ref{Metsch}, that graph $X_{1,2}$ has clique geometry with $m = 5$. Therefore, by Lemma \ref{assoc-k1-k2-bound}, 
$$\left(\frac{2}{3} - \varepsilon\right)k_2\leq (1 - \varepsilon)k_2 - p_{2,2}^{1}\leq p_{2,3}^{1}\leq \frac{m^2 - 4}{8}k_1,\, \text{ so }\, k_2\leq 4 k_1.$$
Hence, in fact, Eq. \eqref{eq-x1-l1-b2} implies  
%$\lambda_{1}(X_{1,2})\geq \frac{1}{5}(k_1+k_2)+\frac{1}{5}k_1 \geq (\frac{1}{5}+3\varepsilon)(k_1+k_2)$. 
$\lambda_{1}(X_{1,2})\geq \frac{1}{5}k_2+(\frac{1}{2}-\varepsilon)k_1 \geq (\frac{1}{4}+6\varepsilon)(k_1+k_2)$.
Thus, using Eq. \eqref{eq-x1-l2-b2} and Eq. \eqref{eq-x1-l1bound} - \eqref{eq-x1-l2bound}, by Theorem \ref{Metsch}, we get that $X_{1,2}$ has clique geometry for $m = 3$. Thus, we can get better estimate
%$$\left(\frac{2}{3} - \varepsilon\right)k_2\leq (1 - \varepsilon)k_2 - p_{2,2}^{1}\leq p_{2,3}^{1}\leq \frac{m^2 - 4}{8}k_1,\, \text{ so }\, k_2\leq \frac{9}{4(1-2\varepsilon)} k_1<3k_1.$$
%Hence, Eq. \eqref{eq-x1-l1-b2} implies  $\lambda_{1}(X_{1,2})\geq \frac{1}{5}k_2+(\frac{1}{2}-\varepsilon)k_1 \geq \frac{1+\delta}{4}(k_1+k_2)$. Thus, using Eq. \eqref{}, by Theorem \ref{Metsch}, we get that $X_{1,2}$ has clique geometry for $m = 3$. Thus, we can get even better estimate

$$\left(\frac{2}{3} - \varepsilon\right)k_2\leq \frac{m^2 - 4}{8}k_1,\, \text{ so }\, k_2\leq \frac{15}{16(1-2\varepsilon)} k_1<k_1.$$

However this contradicts our assumption that $k_2\geq k_1$, so $p_{2,2}^{2}\geq (\frac{1}{3} - \varepsilon -\varepsilon_1) k_2$ is impossible in this case.

\noindent \textbf{Case B.3.} Assume $p_{2,2}^{1}\leq \frac{k_2}{5}$.

\noindent Then, by the assumption of Case B, $p_{2,2}^{2}\leq (2-2\delta) p_{2,2}^{1}\leq (2-2\delta)\frac{k_2}{5}$, so 
\begin{equation}
\begin{multlined}
q(X_2)+\xi(X_2) \leq \max(p_{2,2}^{2}, p_{2,2}^{1})+p_{2,2}^{2}+\sqrt{p_{2,2}^{1}p_{1,2}^{2}}+\varepsilon_1 k_2\leq 
\\ \leq 2(2-2\delta)\frac{k_2}{5}+ \frac{k_2}{5} +\varepsilon_1 k_2  
\leq \left(1 - \frac{4}{5}\delta+\varepsilon_1\right)k_2\leq (1-\varepsilon)k_2.
\end{multlined}
\end{equation}

\end{proof}

\subsubsection{Constituent that is the line graph of a triangle-free regular graph}

Finally, we will consider the case of the last possible outcome provided by Proposition~\ref{assoc-diam2}, the case when one of the constituents is the line graph of a regular triangle-free graph and is not strongly regular. 

First recall the following classical result due to Whitney.
\begin{theorem}[Corollary to Whitney's Theorem, \cite{Whitney}]\label{Whitney} Let $X$ be a connected graph on $n\geq 5$ vertices. Then the natural homomorphism $\phi :\Aut(X) \rightarrow \Aut(L(X))$ is an isomorphism $\Aut(L(X)) \cong \Aut(X)$. 
\end{theorem}

Observe, that the restriction on the diameter of the line graph gives quite strong bound on the degree of a base graph, as summarized in the following lemma.

\begin{lemma}\label{line-so-k-big}
Let $X$ be a $k$-regular graph on $n$ vertices. If the line graph $L(X)$ has diameter $2$, then $k\geq \frac{n}{8}$. 
\end{lemma}
\begin{proof}
Recall that $L(X)$ has $kn/2$ vertices and degree $2(k-1)$. Since $L(X)$ has diameter $2$, parameters satisfy $4k^2 \geq 4(k-1)^2+2(k-1)+1\geq kn/2$, i.e., $k\geq \frac{n}{8}$.
\end{proof}

\begin{proposition}\label{assoc-line-triangle}
Let $X$ be a connected $k$-regular triangle-free graph on $n\geq 5$ vertices, where $k\geq 3$. Suppose $\mathfrak{X}$ is an association scheme of rank 4 and diameter 2 on $V(L(X)) = E(X)$, such that one of the constituents is equal to $L(X)$ and is not strongly regular. Then any two vertices $u,v\in X$ are distinguished by at least $n/8$ vertices. Therefore,   $\Aut(L(X))$ is of order $n^{O(\log(n))}$ and the motion of $L(X)$ is at least $\frac{|V(L(X))|}{16}$. 
\end{proposition}
\begin{proof}
Denote the constituents of $\mathfrak{X}$ by $Y_i$, $0\leq i\leq 3$, where $Y_0$ is the diagonal constituent and $Y_1 = L(X)$.

Since $Y_1$ has diameter $2$, any induced cycle of $X$ has length at most 5. Graph $X$ is triangle-free, so every induced cycle in $X$ has length 4 or 5, and every cycle of length 4 or 5 is induced.

\noindent \textbf{Case 1.} Suppose that there are no cycle of length 5 in $X$, i.e., it is bipartite. Then for $v\in X$ there are no edges between vertices in $N_2(v)$. The graph $X$ is regular, and every induced cycle has length 4, so for every vertex $w\in N_2(v)$ the neighborhoods $N(w)$ and $N(v)$ coincide. Hence, as $X$ is connected, $X$ is a complete regular bipartite graph. However, in this case, $L(X)$ is strongly regular.

\noindent \textbf{Case 2.} Suppose there is a cycle of length 5.  

Let $v_1v_2v_3v_4v_5$ be any cycle of length 5. Take $u$ different from  $v_2, v_5$ and adjacent to $v_1$. Since constituent $Y_1$ has diameter $2$, edges $v_1u$ and $v_3v_4$ are at distance 2 in $L(X)$, thus there is one of the edges $uv_3$ or $uv_4$ in $X$. Again, $X$ is triangle free, so exactly one of them is in $X$, and without lost of generality assume that $uv_3$ is in $X$. In particular, we get that there is a cycle of length $4$. Denote by $r_{i,j}$ the number of common neighbors of $v_i$ and $v_j$. Then, we have shown that $r_{i, i+2}+r_{i,i+3} = k$ for every $i$, where indices are taken modulo 5. Thus, $r_{i, i+2} = \frac{k}{2}$ for every $1\leq i\leq 5$. 

  Observe, that $v_1v_2$ and $v_3v_4$ have exactly one common neighbor in $L(X)$. At the same time, for any cycle $u_1u_2u_3u_4$ edges $u_1u_2$ and $u_3u_4$ have exactly two common neighbors in $L(X)$. Thus, pairs $(u_1u_2, u_3u_4)$ and $(v_1v_2, v_3v_4)$ belong to different constituents of the association scheme, say $Y_2$ and $Y_3$, respectively. Note, that triple of edges $v_1v_2, v_2v_3, v_3v_4$ shows that $p^{1}_{1,3}$ is non-zero.

Take any $v\in X$ and $u\in N_2(v)$. Suppose that there is no $w\in N_2(v)$ adjacent to $u$. Then by regularity of $X$ we get $N(v) = N(u)$. For any $x,y \in N(v)$ triple $vx, xu, uy$ form a triangle with side colors $(1,1,2)$ and we get a contradiction with $p^{1}_{1,3}\neq 0$. 

Hence, for any $u\in N_2(v)$ there exists $w\in N_2(v)$ adjacent to $u$. Take $x\in N(v)\cap N(u)$ and $y\in N(v)\cap N(w)$. Consider the cycle $vxuwy$, then as shown above, vertices $v$ and $u$ have exactly $\frac{k}{2}$ common neighbors. Thus, they are distinguished by at least $|N(u)\triangle N(v)| = 2(k-\frac{k}{2}) = k$ vertices.

Any adjacent vertices have no common neighbors, so they are distinguished by at least $2k$ vertices. Thus, any two distinct vertices are distinguished by at least $k$ vertices. Therefore, by Lemma \ref{line-so-k-big} any two distinct vertices are distinguished by at least $\frac{n}{8}$ vertices.

 By Lemma \ref{Whitney}, $\Aut(X) \cong \Aut(L(X))$ via natural inclusion $\phi$. Thus, bound on the order $|Aut(L(X))|$ follows from Lemma \ref{Babai-disting-order-group}. Let $W$ be the support of $\sigma \in Aut(X) \cong Aut(L(X))$. We show that every vertex in $W$ is incident to at most one edge fixed by $\sigma$. Consider edge $e$ with ends $w_1, w_2$, where $w_1\in W$. Since $\sigma(w_1)\neq w_1$ the only possibility for $e$ to be fixed is $\sigma(w_1) = w_2$ and $\sigma(w_2) = w_1$. This, in particular implies that $w_2\in W$ as well. Every edge incident with $w_1$ and different from $e$ is sent by $\sigma$ to an edge incident with $w_2$, so is not fixed. Therefore, the support of $\phi(\sigma)\in \Aut(L(X))$ is at least \[\frac{|W|(k-1)}{2}\geq \frac{n/8(k -1)}{2}\geq \frac{nk}{32} = \frac{|V(L(X))|}{16}.\] 
\end{proof}

\subsection{Putting it all together}\label{sec-coherent-thm-subsec}

Finally, we are ready to combine the preceding results into the following theorem.

\begin{theorem}\label{main-coherent}
Let $\mathfrak{X}$ be a primitive coherent configuration of rank $4$ on $n$ vertices . Then one of the following is true.
\begin{enumerate}
\item We have $\motion(\mathfrak{X})\geq \gamma n$, where $\gamma>0$ is an absolute constant.
\item The configuration $\mathfrak{X}$ is a Hamming scheme or a Johnson scheme.
\end{enumerate}
\end{theorem}
\begin{proof}
Suppose first that there is an oriented color. Then, since rank is 4, the only possibility is to have two oriented colors $i, j= i^*$ and one undirected color $t$. It is easy to see that $X_t$ is strongly regular.
Then for $n\geq 29$, by Babai's theorem (Theorem \ref{babai-str-reg-thm}), $\motion(X_t)\geq \frac{1}{8} n$, or $X_t$ is  a $T(s)$, a $L_{2}(s)$ for some $s$ or the complement of one of them. 

The constituent $X_t$ could not be the complement of $L_{2}(s)$ since oriented diameter of $X_{i}$ should be $2$, which contradicts $k_i^{2}\geq n-1$. Indeed, in this case, $2k_i =k_i+k_i^{*} = 2(s-1)$, while $n = s^2$.   

Now, observe that $p_{i,i^{*}}^{i} = p_{i^{*},i}^{i} = p_{i,i}^{i}$. Moreover, by Eq. \eqref{eq-sum-param},
$$k_i+k_{i^*} = p_{i,i}^{i}+p_{i,i^{*}}^{i} + p_{i^{*},i}^{i} + p_{i^*, i^*}^{i}+p^{i}_{i, t}+p^{i}_{i^*, t}+p_{i, 0}^{i}.$$ 
Thus, using Eq. \eqref{eq-sum-param} again, $p_{i, i}^{i}+p_{i^*,i^*}^i\geq \frac{2k_i-k_t-1}{3}$. If $X_t$ is either $T(s)$ or $L_{2}(s)$, then $k_i = k_{i^*}> \frac{n}{3}$ and $k_t< \frac{n}{3}$. Thus every pair of vertices connected by color $i$ is distinguished by at least $\frac{k_i}{3}\geq \frac{n}{9}$ vertices. Hence, by primitivity of $\mathfrak{X}$ and Lemma~\ref{babai-dist} the motion of $\mathfrak{X}$ is at least $\frac{n}{18}$. In the last case, when $X_{t}$ is a complement of $T(s)$, result follows from Lemma~\ref{sun-wilmes-tool} and the inequality $p_{i, i}^{i}+p_{i^*,i^*}^i\geq \frac{k_i}{3}$.

 Now assume that all the colors in $\mathfrak{X}$ are undirected, i.e., $\mathfrak{X}$ is an association scheme. Every constituent of $\mathfrak{X}$ has diameter at most $3$ as rank of $\mathfrak{X}$ is 4. Moreover, as discussed in Section \ref{sec-prelim}, if there is a constituent of diameter $3$, then $\mathfrak{X}$ is induced by a distance-regular graph. In this case the statement follows from Theorem \ref{main-thm}. None of the components can have diameter $1$ as rank is 4. Finally, if $\mathfrak{X}$ is an association scheme of rank 4 and diameter 2, then the statement of the theorem follows from Lemma \ref{assoc-k2-large}, Propositions \ref{assoc-diam2}, \ref{assoc-x3-strongly-reg}, \ref{assoc-x2-strongly-reg}, \ref{assoc-x1-strongly-reg}    and Proposition~\ref{assoc-line-triangle}.
\end{proof}

\section{Summary and open questions}\label{sec-summary}
Recall that the motion of a structure is the minimal degree of its automorphism group. In this paper we studied the problem of classifying primitive coherent configurations with sublinear motion in two special cases: the case of primitive configurations of rank 4 and the case of distance-regular graphs  of bounded diameter (metric schemes of bounded diameter). In the case of rank 4 we achieved a full classification confirming Conjecture~\ref{conj-3} for this case. The case of rank $\leq 3$ was settled by Babai in \cite{Babai-str-reg}.

In the case of primitive distance-regular graphs of bounded diameter, we proved that graphs with sublinear motion possess a rather restrictive combinatorial structure. More specifically, we proved that all such graphs are geometric with bounded smallest eigenvalue.

It is not hard to see that for geometric distance-regular graphs neither the spectral, nor the distinguishing number tool work.

 One way to get a full classification of exceptions would be to prove a full classification of all geometric distance-regular graphs with diameter $d\geq 3$, and smallest eigenvalue $-m$, for a fixed positive integer $m$. Such a classification with the additional restriction $\mu\geq 2$ is conjectured by Bang and Koolen in \cite{Bang-Koolen-conj} (see Conjecture~\ref{conj-bang-k}). However, for $\mu=1$ no conjecture has been stated. It is possible, however, that it will be considerably easier to prove the following conjecture directly.

\begin{conjecture}
Let $X$ be a primitive geometric distance-regular graph of diameter $d\geq 3$. Then there exists a constant $\gamma_d>0$ s.t. one of the following is true.
\begin{enumerate}
\item $\motion(X)\geq \gamma_d n$.
\item $X$ is a Johnson graph $J(s, d)$ or a Hamming graph $H(d, s)$.
\end{enumerate}
\end{conjecture}

\noindent A possible step toward this conjecture is to deal with the following problem first.

\begin{problem}
Improve significantly the constant $m_d$ in the statement of Theorem~\ref{main-general-case}. Ideally, show $m_d = d$.
\end{problem}

Switching from the distance-regular case to arbitrary primitive coherent configurations of rank $r\geq 5$ leads to serious technical issues in trying to implement our approach.  A significant obstacle is the difficulty of a spectral analysis of the constituents of the coherent configuration. Namely, for configurations of rank 4 we analyzed the spectral gap ``by hand'' through Propositions \ref{assoc-spectral-radius} and \ref{assoc-x12-approx}. For coherent configurations of higher rank we need more general techniques.

\begin{problem}
Do there exist $\varepsilon, \delta>0$ such that the following statement holds? If the minimal distinguishing number  $D_{\min}(\mathfrak{X})$ of a primitive coherent configuration $\mathfrak{X}$ satisfies $D_{\min}(\mathfrak{X})<\varepsilon n$, then the spectral gap for the symmetrization of one of the constituents $X_i$ of $\mathfrak{X}$ is $\geq \delta k_i$. What $\delta$ can we achieve? 
\end{problem}

We would like to point out, that even $\delta k_i$ spectral gap for one of the constituents is not sufficient for an application of the spectral tool (Lemma \ref{mixing-lemma-tool}). However, we expect that a result of this flavor would introduce important techniques to the analysis.

We also would like to mention that there should be a reasonable hope to prove Conjecture~\ref{conj-3} for the case when no color is overwhelmingly dominant. The following result easily follows from distinguishing number analysis. In particular, in the case of bounded rank it gives an $\Omega(n)$ bound on the motion.

\begin{proposition}\label{mindeg-bounded-degree}
Fix $0<\delta <1$ and integer $r\geq 3$. Let $\mathfrak{X}$ be a primitive coherent configuration of rank $r$ on $n$ vertices. Assume that each constituent $X_i$ has degree $k_i\leq \delta n$.
Then \[\motion(\mathfrak{X})\geq D_{\min}(\mathfrak{X})\geq\frac{\min(\delta, 1-\delta)}{6(r-1)}n.\]  
\end{proposition}
\begin{proof}
The condition $k_i\leq \delta n$, for all $i$, implies that there exists a set $I$ of colors  such that $\sum\limits_{i\in I}k_i = \alpha n$ for $\min(\delta, 1-\delta)/2 \leq \alpha \leq 1/2$. Fix any vertex $u$ of $\mathfrak{X}$. 
We want to show that for some vertex $v$ the inequality $D(u,v)\geq \alpha n/3$ holds. 

Suppose this is not true. Denote $N_I(u) = \{z|\, c(u,z)\in I\}$. Let us count the number of pairs $(v,z)$ with $z\in N_{I}(u)$ and $c(v,z)\in I$ in two different ways. Since $\sum\limits_{i^*\in I}k_i = \alpha n$ and $z\in N_I(u)$, there are $\alpha^2 n^2$ such pairs. On the other hand, for every $v$ we have $D(u,v) \leq \alpha n/3$, so at least $2\alpha n/3$ vertices $z\in N_{I}(u)$ are paired with $v$. Therefore, the number of pairs in question is at least $n\cdot \frac{2\alpha}{3} n$. This contradicts the condition $0<\alpha\leq\frac{1}{2}$.  

Therefore, there exists a pair of vertices with $D(u,v)\geq \alpha n/3$. Finally, the configuration $\mathfrak{X}$ is primitive, so by Lemma~\ref{babai-dist} we get $\motion(\mathfrak{X})\geq D_{\min}(\mathfrak{X})\geq \frac{\alpha}{3(r-1)}n$.
\end{proof}

However, when the rank is unbounded, this seemingly simple case of Conjecture \ref{conj-3} (every constituent has degree $\leq \delta n$) is still open. To avoid exceptions we relax the conjectured lower bound to $\Omega(n/\log(n))$.

\begin{conjecture}\label{conj-bound-degree}
Fix $0<\delta<1$. Let $\mathfrak{X}$ be a primitive coherent configuration on $n$ vertices. Assume that every constituent has degree $\leq \delta n$. Then $\motion(\mathfrak{X}) = \Omega(n/\log(n))$.
\end{conjecture}  

Next we observe that Cameron schemes satisfy Conjecture \ref{conj-bound-degree}.

\begin{proposition}\label{prop-camer-nlogn}
Fix $0<\delta<1$.  Consider a Cameron group $(A_m^{(k)})^d \leq G\leq S_m\wr S_d$ acting on $n = \binom{m}{k}^d$ points and let $\mathfrak{X} = \mathfrak{X}(G)$ be the corresponding Cameron scheme. Assume that every constituent of $\mathfrak{X}$ has degree $\leq \delta n$. Then $\motion(\mathfrak{X}) = \Omega(n/\log(n))$. 
\end{proposition}
\begin{proof}
We can assume $k\leq m/2$. Note that then the rank of $\mathfrak{X}$ is equal to $kd+1$. 

Case 1. Suppose that $k\leq m/3$. Then
\[n = \binom{m}{k}^d\geq \left(\frac{m-k}{k}\right)^{kd}\geq \left(\frac{2k}{k}\right)^{kd} = 2^{kd}\]
Thus $kd \leq \log(n)$ in this case, and statement follows from Proposition~\ref{mindeg-bounded-degree}.

Case 2. Suppose that $m/3<k\leq m/2$. By Lemma~\ref{cameron-min-deg}, we have that as $m\rightarrow \infty$ the inequality $\motion(\mathfrak{X}) \geq \alpha n$ holds for some $\alpha>0$. At the same time, by the proof of Lemma~\ref{cameron-min-deg} we know that the motion of $\mathfrak{X}$ does not depend on $d$. Thus as $\motion(X) \geq \alpha n$ is violated just by finite number of pairs $(m,k)$, we still have $\motion(\mathfrak{X}) = \Omega(n)$ in this case.
\end{proof}

We observe that the bound in Conjecture \ref{conj-bound-degree}, if true, is nearly tight, as the example of Hamming schemes $H(tm,m)$ with $t = -\lfloor \log(\delta)\rfloor$ shows. We note that for $m\geq 3$ the Hamming schemes $H(k,m)$ are primitive.

\begin{proposition} Consider Hamming scheme $H(tm, m)$ with $t = -\lfloor \log(\delta)\rfloor$ on $n = m^{tm}$ points. Then its maximum constituent degree satisfies $k_{\max} \leq \delta n$ and the motion satisfies \[\motion(H(tm, m))= O\left(\frac{n\log\log(n)}{\log(n)}\right).\]
\end{proposition}
\begin{proof}
Note that the maximum degree is $k_{\max} = (m-1)^{tm}$. Then 
\[k_{\max} = \left(\frac{m-1}{m}\right)^{mt}n \leq \eee^{-t}n\leq \delta n.\]

The motion of $H(tm, m)$ is realized by a 2-cycle in the first coordinate, and is equal to $2n/m$. The number of vertices is $n = m^{mt} = \eee^{tm\log(m)}$, so $m\log(m) = \log(n)/t$. Thus $m>\log(n)/(t\log\log(n))$. Hence, the motion satisfies
\[\motion(H(tm, m))\leq \frac{2n\log\log(n)t}{\log(n)} = O\left(\frac{n\log\log(n)}{\log(n)}\right).\]
\end{proof}

\end{document}